\documentclass[journal]{IEEEtran}

\usepackage{graphicx}
\usepackage{epsfig} 
\usepackage{mathptmx}
\usepackage{times}
\usepackage{amsmath}
\usepackage{amssymb}
\usepackage[section]{placeins}
\usepackage{placeins}
\usepackage{amsmath}
\usepackage{multirow}
\usepackage{graphicx}
\usepackage{epstopdf}
\usepackage[normalem]{ulem}
\usepackage{hyperref}
\usepackage{subcaption}
\usepackage{algorithm,tabularx}
\usepackage{algpseudocode}
\usepackage{amsfonts}
\usepackage{amsthm}
\usepackage{cases}%
\usepackage{romannum}
\providecommand{\U}[1]{\protect\rule{.1in}{.1in}}

\newcommand{\R}{\mathbb{R}}

\newcommand{\mb}[1]{\mathbf{#1}}
\newtheorem{assumption}{Assumption}
\newtheorem{theorem}{Theorem}

\newtheorem{lemma}{Lemma}
\newtheorem{remark}{Remark}
\newtheorem{definition}{Definition}
\useunder{\uline}{\ul}{}
\makeatletter
\newcommand{\multiline}[1]{  \begin{tabularx}{\dimexpr\linewidth-\ALG@thistlm}[t]{@{}X@{}}
#1
\end{tabularx}
}
\makeatother

\usepackage{algpseudocode}
\algdef{SE}[DOWHILE]{Do}{doWhile}{\algorithmicdo}[1]{\algorithmicwhile\ #1}%

\usepackage{caption}
\usepackage{subcaption}
\usepackage{dblfloatfix} 
\usepackage{textcomp}
\usepackage{cite}
\usepackage{seqsplit}%
\begin{document}
%
\title{Asymptotic Analysis for Greedy Initialization of Threshold-Based Distributed Optimization of Persistent Monitoring on Graphs}
%
%
%

\author{Shirantha Welikala,~\IEEEmembership{Student Member,~IEEE,}
        and Christos G. Cassandras,~\IEEEmembership{Fellow,~IEEE,}
\thanks{$^{\star}$Supported in
part by NSF under grants ECCS-1509084, DMS-1664644, CNS-1645681, by AFOSR under
grant FA9550-19-1-0158, by ARPA-E's NEXTCAR program under grant
DE-AR0000796 and by the MathWorks.}
\thanks{Authors are with the Division of Systems Engineering and Center for Information and Systems Engineering, Boston University, Brookline, MA 02446, \texttt{{\small \{shiran27,cgc\}@bu.edu}}.}
\thanks{Manuscript received July 22, 2019; revised July22, 2019.}}

%
%

\markboth{Submitted to Automatica}%
{Shell \MakeLowercase{\emph{et al.}}: Bare Demo of IEEEtran.cls for IEEE Journals}
%



\maketitle

\begin{abstract}
This paper considers the optimal multi-agent persistent monitoring problem defined for a team of agents on a set of nodes (targets) interconnected according to a fixed network topology. The aim is to control this team so as to minimize a measure of overall node state uncertainty evaluated over a finite time interval. A class of distributed threshold-based parametric controllers has been proposed in prior work to control agent dwell times at nodes and next-node destinations by enforcing thresholds on the respective node states. Under such a Threshold Control Policy (TCP), an on-line gradient technique was used to determine optimal threshold values. However, due to the non-convexity of the problem, this approach often leads to a poor local optima highly dependent on the initial thresholds used. To overcome this initialization challenge, we develop a computationally efficient off-line greedy technique based on the asymptotic analysis of the network system. This analysis is then used to generate a high-performing set of initial thresholds. Extensive numerical results show that such initial thresholds are almost immediately (locally) optimal or quickly lead to optimal values. In all cases, they perform significantly better than the locally optimal solutions known to date.
\end{abstract}

\begin{IEEEkeywords}
Multi-agent Systems, Hybrid Systems, Optimization, Trajectory Planning,
\end{IEEEkeywords}

%
\IEEEpeerreviewmaketitle

\section{Introduction}\label{Sec:Introduction}

\IEEEPARstart{P}{ersistent} monitoring of a dynamically changing environment using a cooperating fleet of mobile agents has many applications across different domains. For example, in smart cities \cite{Rezazadeh2019}, transportation systems \cite{Yamashita2003} and manufacturing plants \cite{Liaqat2019}, a team of agents can be used to monitor different regions for congestion, disruptions or any other dynamic events of interest. Further, in a smart grid \cite{Caprari2010,Fan2018,Menendez2017}, a team can be used to inspect power plants and transmission lines to maintain a reliable and safe power system. Additional applications include patrolling \cite{Huynh2010}, surveillance \cite{Aksaray2015,Maza2011}, data collecting \cite{Zhou2018,Smith2011,Khazaeni2018}, coverage \cite{Sun2020}, particle tracking \cite{Shen2011} and sensing \cite{Zhou2019,Trevathan2018,Lin2013}.

Monitoring problems in general can be classified based on the nature of the \emph{environment}, \emph{objective} and \emph{dynamics }involved. In particular, based on the nature of the environment, a monitoring problem may have a finite set of \textquotedblleft points of interest\textquotedblright \ \cite{Rezazadeh2019,Welikala2019P3} or lack thereof \cite{Lin2013,Maini2018} in the environment to be monitored. Based on the nature of the objective, different monitoring problems van be formulated to optimize event-counts \cite{Yu2015}, idle-times \cite{Hari2019}, error covariances \cite{Pinto2020} or visibility states \cite{Maini2018} related to the environment. Finally, based on the nature of the environment dynamics, a monitoring problem can be either deterministic \cite{Yu2016,Zhou2019,Song2014} or stochastic \cite{Rezazadeh2019,Lan2013}.

The persistent monitoring problem considered in this paper is focused on an $n$-Dimensional ($n$-D) environment containing a finite number of points of interest (henceforth called \textquotedblleft targets\textquotedblright). The objective of the agent team is to collect information from (i.e., sense) each target to reduce an \textquotedblleft uncertainty\textquotedblright\ metric associated with the target state. In particular, the dynamics of each target's uncertainty metric are deterministic and the global objective is to minimize an overall measure of target uncertainties via controlling the agent trajectories.

The work in \cite{Zhou2018} has addressed the problem of persistent monitoring in $1$-D environments by formulating it as an optimal control problem and reducing it to a parametric optimization problem. The corresponding optimal parameters have been determined by following a gradient descent process with the gradients evaluated on-line through Infinitesimal Perturbation Analysis (IPA) \cite{Cassandras2010b}.
In contrast to the $1$-D case, finding the solution to the problem of persistent monitoring in $2$-D environments is much more complicated \cite{Lin2013}. However, as a remedy, the works in \cite{Lin2013,Khazaeni2018,Pinto2020} propose to constrain agent trajectories to certain standard families of parametric trajectories (e.g., elliptical, Lissajous and Fourier) and then to use IPA to obtain the optimal agent trajectories within these families. Nevertheless, as pointed out in \cite{Zhou2019}, all these problems are non-convex, hence standard gradient-based methods often lead to poor local optima dependent on the initial conditions used.

To overcome the challenges mentioned above, \cite{Zhou2019} exploits the network structure of the monitoring system by adopting a graph topology to abstract targets and feasible inter-target agent trajectories as graph nodes and edges, respectively. Note that this graph abstraction has the added advantage of accounting for physical obstacles that might be present in the environment. In this Persistent Monitoring on Networks (PMN) paradigm, an agent trajectory is fully characterized by a \emph{sequence of targets} to be visited and the corresponding \emph{sequence of dwell times} to be spent at each visited target. Therefore, the controller that optimizes a given objective should yield such a (target, dwell time) sequence for all the agents. Clearly, this optimization problem is significantly more complicated than the NP-hard traveling salesman problem \cite{Bektas2006} which only involves finding an optimal sequence of nodes to visit.

As an alternative, \cite{Welikala2020P5} uses a Receding Horizon Control (RHC) technique that requires each agent to repeatedly solve a smaller optimization problem to determine its optimal trajectory using only the local information available to it. This on-line control approach has the advantage of being distributed and parameter-free; on the other hand, it cannot exploit any global information regarding the underlying network structure. In contrast, the parametric approach taken in \cite{Zhou2019} uses a class of controllers characterized by threshold parameters which can be optimized in an on-line distributed manner using gradient descent. However, due to the non-convexity of the associated objective function with respect to these thresholds, this gradient-based approach often converges to a poor local optimum that is highly dependent on the initial thresholds which in \cite{Zhou2019} are generated randomly. Nevertheless, the solutions obtained by this parametric controller may be drastically improved by an \emph{initial off-line step} in which the global information available regarding the underlying network structure is exploited to determine a set of high-performing thresholds used to initialize the IPA-based on-line gradient descent process. This process subsequently converges to an improved set of (still locally optimal) thresholds. Our contribution in this paper can be seen more broadly as a systematic approach to select effective initial conditions for gradient-based methods that solve non-convex optimization problems pertaining to a large class of dynamic multi-agent systems beyond persistent monitoring. In particular, this is accomplished by analyzing the asymptotic behavior of such systems and using the resulting optimal control policies to initialize a parametric class of controllers.

For the PMN systems considered in this paper, our contributions include: 
(\romannum{1}) The asymptotic analysis of single-agent PMN systems with the agent constrained to follow a periodic and non-overlapping sequence of targets (also called \textquotedblleft target-cycle\textquotedblright), 
(\romannum{2}) A graph partitioning process that enables the extension of this analysis to the deployment of multiple agents, and 
(\romannum{3}) A computationally efficient, off-line technique that constructs a high-performing set of thresholds for PMN problems. 
As shown through extensive simulation results, the initial thresholds provided by this initialization technique are often immediately optimal (still local). Thus, in such cases, an effective initialization eliminates the need for any subsequent gradient descent process.

The paper is organized as follows. Section \ref{Sec:ProblemFormulation} provides the problem formulation and reviews the threshold-based control policy (TCP) proposed in \cite{Zhou2019}. Section \ref{Sec:DenseGraphsSingleAgent} includes the asymptotic analysis and a candidate threshold initialization technique, assuming the underlying PMN problem is single-agent and the network is sufficiently dense. Next, Section \ref{Sec:SparseGraphsSingleAgent} generalizes the asymptotic analysis and the threshold initialization technique proposed in Section \ref{Sec:DenseGraphsSingleAgent} to any network (still assuming a single-agent PMN scenario). Subsequently, Section \ref{Sec:MultiAgents} further generalizes the proposed threshold initialization technique to multi-agent systems. A sufficient number of simulation examples are discussed in each section, and, they have been compared with the respective state of the art solution \cite{Zhou2019}. Finally, Section \ref{Sec:Conclusion} concludes the paper.

\section{Problem Formulation} \label{Sec:ProblemFormulation}
Consider a two-dimensional mission space with $M$ targets (nodes) labeled $\mathcal{T}=\{1,2,\ldots,M\}$ and $N$ agents labeled $\mathcal{A}=\{1,2,\ldots,N\}$. Each target $i\in\mathcal{T}$ is located at a fixed position $X_{i}\in\mathbb{R}^{2}$. Each agent $a\in\mathcal{A}$ is allowed to move in the mission space, therefore, its trajectory is denoted by $\{s_{a}(t)\in\mathbb{R}^{2},t\geq0\}$. Target locations and initial agent locations are prespecified.

\paragraph*{\textbf{Target Uncertainty Model}} 
We follow the same model used in \cite{Zhou2019}. 
Each target $i\in\mathcal{T}$ has an associated \emph{uncertainty state} $R_{i}(t)\in\mathbb{R}$ with the following properties: 
(\romannum{1}) $R_{i}(t)$ increases at a rate $A_{i}$ when no agent is visiting it, 
(\romannum{2}) $R_{i}(t)$ decreases at a rate $B_{i}N_{i}(t)-A_{i}$ where $B_{i}$ is the uncertainty removal rate by an agent and $N_{i}(t)=\sum_{a=1}^{N}1_{\{s_{a}(t)=X_{i}\}}$ is the number of agents present at target $i$ at time $t$, and 
(\romannum{3}) $R_{i}(t)\geq0,\ \forall t\geq0$. 
The target uncertainty state dynamics for any $i\in\mathcal{T}$ are 
\begin{equation}\label{Eq:TargetDynamics}
    \dot{R}_i(t) = 
    \begin{cases}
    0 & \mbox{ if } R_i(t) = 0 \mbox{ and } A_i \leq B_iN_i(t),\\
    A_i - B_iN_i(t) & \mbox{ otherwise,}  
    \end{cases}
\end{equation}
with, $A_{i},B_{i}$ and $R_{i}(0)$ values pre-specified. As pointed out in \cite{Zhou2019}, this target uncertainty model has an attractive queueing system interpretation where each $A_{i}$ and $B_{i}N_{i}(t)$ can be thought of as an arrival rate and a controllable service rate respectively for each target viewed as a node in a queueing network.

\paragraph*{\textbf{Agent Model}} 
In some persistent monitoring models \cite{Zhou2018}, each agent $a\in\mathcal{A}$ is assumed to have a finite sensing range $r_{a}>0$ which allows it to decrease $R_{i}(t)$ whenever $\Vert s_{a}(t)-X_{i}\Vert\leq r_{a}$. However, we follow the approach used in \cite{Zhou2019} where $r_{a}=0$ is assumed and $N_{i}(t)$ is used to replace the joint detection probability of a target $i\in\mathcal{T}$. This simplifies the analysis and enables accommodation of the target graph topology \cite{Zhou2019}.

As we will see next, contributions of this paper are invariant to the used dynamic model of the agents (partly due to the embedded graph model). Therefore, we do not explicitly state an agent model. 




\paragraph*{\textbf{Objective Function \cite{Zhou2019}}}
The objective of this persistent monitoring system is to minimize a measure of \emph{mean system uncertainty} $J_{T}$ (evaluated over a finite time horizon $T$), where
\begin{equation}
J_{T}=\frac{1}{T}\int_{0}^{T}\sum_{i=1}^M R_{i}(t)dt,
\label{Eq:ObjectiveFunction1}%
\end{equation}
by controlling agent motion.


\paragraph*{\textbf{Target Topology (Graph)}} 
We embed a directed graph topology $\mathcal{G}=(\mathcal{V},\mathcal{E})$ into the 2D mission space where the graph vertices represent the targets ($\mathcal{V}=\{1,2,\ldots,M\}=\mathcal{T}$), and the graph edges represent inter-target trajectory segments (which may be curvilinear paths with arbitrary shapes so as to account for potential obstacles in the mission space) available for agents to travel ($\mathcal{E}\subseteq\{(i,j):i,j\in \mathcal{V}\}$). In particular, the shape of each trajectory segment $(i,j)\in\mathcal{E}$ can be considered as a result of a lower level optimal control problem which aims to minimize the travel time that an agent takes to go from target $i$ to target $j$ while accounting for potential constraints in the mission space and agent dynamics. For the purpose of this paper, we assume each trajectory segment $(i,j)\in\mathcal{E}$ to have a fixed such optimal travel time value $\rho_{ij}\in\mathbb{R}_{\geq0}$. Based on $\mathcal{E}$, the \emph{neighbor set} $\mathcal{N}_{i}$ of target $i\in\mathcal{V}$ is defined as 
$$\mathcal{N}_{i}\triangleq\{j:(i,j)\in \mathcal{E}\}.$$

Based on the target dynamics \eqref{Eq:TargetDynamics} and agent sensing capabilities assumed, to minimize the objective $J_{T}$ \eqref{Eq:ObjectiveFunction1} it is intuitive that each agent has to \emph{dwell} (i.e., remain stationary) only at targets that it visits in its trajectory. Moreover, based on the embedded target topology $\mathcal{G}$ that constrains the agent motion, it is clear that when an agent $a\in\mathcal{A}$ leaves a target $i\in\mathcal{V}$ its next target would be some target $j\in\mathcal{N}_{i}$ that is only reachable by \emph{traveling} on edge $(i,j)\in\mathcal{E}$ for a time duration of $\rho_{ij}$. In essence, this \emph{dwell-travel} approach aims to minimize the agent time spent outside of targets - which is analogous to minimizing the idle time of servers in a queueing network.

Each time an agent $a\in\mathcal{A}$ arrives at a target $i\in\mathcal{V}$, it has to determine a \emph{dwell time} $\tau_{i}^{a}\in\mathbb{R}_{\geq0}$ and a \emph{next visit} target $v_{i}^{a}\in\mathcal{N}_{i}$ (see Fig. \ref{Fig:discreteEventSystem}). Therefore, for the set of agents, the optimal control solution that minimizes the objective $J_{T}$ takes the form of a set of optimal dwelling time and next visit target sequences. Determining such an optimal solution is a challenging task even for the simplest PMN problem configurations due to the nature of the involved search space.

\begin{figure}[!h]
    \centering
    \includegraphics[width=2.7in]{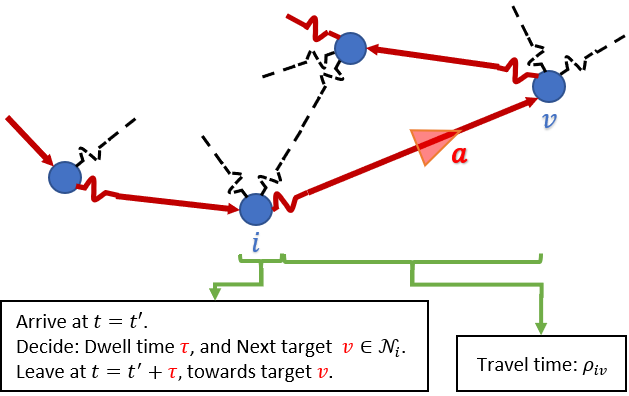}
    \caption{Agent behavior defined by its decision sequence.}
    \label{Fig:discreteEventSystem}
\end{figure}

\paragraph*{\textbf{Threshold Control Policy}} 
To address the aforementioned challenge, we adopt the threshold-based control policy (TCP) proposed in \cite{Zhou2019}. In particular, under this TCP, each agent $a\in\mathcal{A}$ makes its decisions by adhering to a set of pre-specified parameters denoted by $\Theta^{a}\in\mathbb{R}^{M\times M}$ which serve as thresholds on target uncertainties. Note that the $(i,j)$\textsuperscript{th} parameter in the $\Theta^{a}$ matrix is denoted as $\theta_{ij}^{a}\in\mathbb{R}_{\geq0}\ \forall i,j\in\mathcal{V}$.

We denote the set of neighbors of a target $i$ that violates their thresholds (i.e., have higher uncertainty values than respective thresholds) with respect to an agent $a$ residing in target $i$ at time $t$ by $\mathcal{N}_{i}^{a}(t)\subseteq\mathcal{N}_{i}$ (also called \emph{active neighbors}) where 
\begin{equation}
\mathcal{N}_{i}^{a}(t)\triangleq\{j:R_{j}(t)>\theta_{ij}^{a},\ j\in
\mathcal{N}_{i}\}. \label{Eq:ThresholdViolatedNeighbors}%
\end{equation}
Assume an agent $a$ arrives at target $i$ at a time $t=t^{\prime}$. Then, the dwell time $\tau_{i}^{a}$ to be spent at target $i$ is determined by: 
(\romannum{1}) the diagonal element $\theta_{ii}^{a}$ based on the \emph{threshold satisfaction} condition $R_{i}(t)<\theta_{ii}^{a}$ and
(\romannum{2}) the \emph{active neighbor existence} condition $\vert \mathcal{N}_{i}^{a}(t) \vert>0$ at $t=t^{\prime}+\tau_{i}^{a}$ ($\vert \cdot\vert$ is the cardinality operator). 
Subsequently, agent $a$'s next visit target $v_{i}^{a}$ is chosen from the set of active neighbors $\mathcal{N}_{i}^{a}(t)\subseteq\mathcal{N}_{i}$ using the off-diagonal thresholds $\{\theta_{iv}^{a}:v\in\mathcal{N}_{i}^{a}(t)\}$ at $t=t^{\prime}+\tau_{i}^{a}$. Formally,
\begin{equation}
\begin{aligned} \tau_i^a :=& \underset{\tau \geq 0}{\mathrm{arginf}}\ \mathbf{1}\left\{ [R_{i}(t^{\prime}+\tau) < \theta_{ii}^a] \ \& \  [\vert \mathcal{N}_{i}^{a}(t^{\prime}+\tau) \vert > 0 ] \right\}, \\ v_i^a :=& \underset{v \in \mathcal{N}_{i}^{a}(t^{\prime}+\tau_i^a)}{\mathrm{argmax}} \left\{ R_v(t^{\prime}+\tau_i^a)-\theta_{iv}^a\right\}. \end{aligned} \label{Eq:DiscreteEventSystem}%
\end{equation}
While the first condition in the $\tau_{i}^{a}$ expression in \eqref{Eq:DiscreteEventSystem} ensures that agent $a$ will dwell at target $i$ until at least its own uncertainty $R_{i}(t)$ drops below $\theta_{ii}^{a}$, the second condition ensures that when agent $a$ is ready to leave target $i$ there will be at least one neighbor $v\in\mathcal{N}_{i}$ whose uncertainty $R_{v}(t)$ has increased beyond the threshold $\theta_{iv}^{a}$. The $v_{i}^{a}$ expression in \eqref{Eq:DiscreteEventSystem} implies that $v_{i}^{a}$ is the neighboring target of $i$ chosen from the set $\mathcal{N}_{i}^{a}(t^{\prime}+\tau_{i}^{a})\subseteq\mathcal{N}_{i}$ with the largest threshold violation. In all, the update equations in \eqref{Eq:DiscreteEventSystem} define each agent's dwell time and next visit decision sequence under the TCP.

A key advantage of this TCP approach is that based on
\eqref{Eq:ThresholdViolatedNeighbors} and \eqref{Eq:DiscreteEventSystem}, each agent now only needs to use the neighboring target state information. Thus, each agent operates in a \emph{distributed} manner. An example target topology and an agent threshold matrix are shown in Fig. \ref{Fig:thresholds}. Note that when certain edges are missing in the graph, the respective off-diagonal entries in $\Theta^{a}$ are irrelevant and hence denoted by $\theta_{ij}^{a}=\infty$.

\begin{figure}[!h]
    \centering
    \includegraphics[width=3.2in]{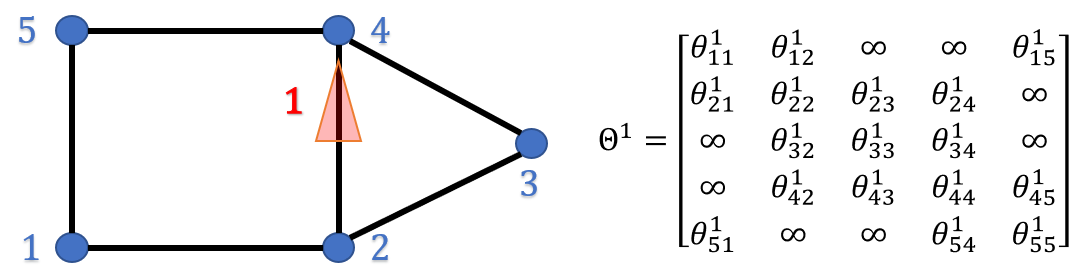}
    \caption{An example target topology with five targets and one agent with its threshold parameters}
    \label{Fig:thresholds}
\end{figure}

\paragraph*{\textbf{Discrete Event System View}} 
Under the TCP mentioned above, the behavior of the PMN system is fully defined by the set of agent decision sequences
\[
\mathcal{U}(\Theta)=\{(\tau_{i(l)}^{a}(\Theta^{a}),v_{i(l)}^{a}(\Theta
^{a})):l\in\mathbb{Z}_{>0},\ a\in\mathcal{A}\},
\]
where $\Theta\in\mathbb{R}^{M\times M\times N}$ is the collection of all agent threshold matrices and $i:\mathbb{Z}_{>0}\rightarrow\mathcal{V}$ (in other words, $i(l)$ is the $l$\textsuperscript{th} target visited by agent $a$).

Moreover, according to \eqref{Eq:DiscreteEventSystem}, the complete persistent monitoring system can be modelled as a discrete event system (DES), specifically as a \emph{deterministic automata with outputs} \cite{Cassandras2010}. In that case, the \emph{state} would be the $N \times 2$ tuple with agent modes and their residing/pursuing target information. The \emph{event set} is the set of: (\romannum{1}) all possible agent arrivals and departures at/from targets, (\romannum{2}) the instances where a target uncertainty reaches $0$ from above, and (\romannum{3}) a `start' and an `end' events triggered only at times $t=0$ and $t=T$ respectively. The \emph{output} can be considered as a vector $\bar{R}^{k} = [R_i(t^k)]_{i\in\mathcal{V}} \in \R^M$ evaluated at all the event times $\{t^k: k\in \{0,1,\ldots,K\}\}$ with $t^0 = 0$ and $t^K = T$. Note that under some TCP $\Theta$, we can get the state and output trajectories of the DES along with its event times by simulating \eqref{Eq:DiscreteEventSystem}.

Considering the dependence of both the state trajectory and the output trajectory on the chosen set of parameters $\Theta$, the performance metric $J_{T}$ in \eqref{Eq:ObjectiveFunction1} depends on the parameters $\Theta$. Therefore, within the TCP class of agent controllers, we aim at determining an Optimal TCP (OTCP) $\Theta^{\ast}$ such that 
\begin{equation}
\Theta^{\ast}=\underset{\Theta\geq\mathbf{0}}{\mathrm{argmin}}\ J_{T}%
(\Theta)=\frac{1}{T}\sum_{i=1}^{M}\sum_{k=0}^{K}\int_{t^{k}}^{t^{k+1}}%
R_{i}(t)dt. \label{Eq:ObjectiveFunction2}%
\end{equation}
Differentiating the cost $J_{T}(\Theta)$ w.r.t. parameters $\Theta$ yields,
\begin{equation}
\nabla J_{T}(\Theta)=\frac{1}{T}\sum_{i=1}^{M}\sum_{k=0}^{K}\int_{t^{k}%
}^{t^{k+1}}\nabla R_{i}(t)dt,
\end{equation}
where $\nabla\equiv\frac{\partial}{\partial\Theta}$. It should be noted that even though event times are dependent on the TCP $\Theta$, when taking the derivative \cite{Flanders1973}, the effect of it gets canceled out since we have fixed $t^{0}=0$ and $t^{K}=T$ \cite{Cassandras2010b}. Further, using the linear behavior of the target uncertainty dynamics in \eqref{Eq:TargetDynamics}, and the way that we have designed our event space, following Lemma 1 in \cite{Zhou2019} we can easily show that $\nabla R_{i}(t)=\nabla R_{i}(t^{k})\ \forall t\in(t^{k},t^{k+1}]$. Therefore, the gradient $\nabla J_{T}(\Theta)$ becomes a simple summation:
\begin{equation}
\nabla J_{T}(\Theta)=\frac{1}{T}\sum_{i=1}^{M}\sum_{k=0}^{K}\nabla R_{i}%
(t^{k})(t^{k+1}-t^{k}).
\end{equation}

\paragraph*{\textbf{Previous Work in \cite{Zhou2019}}} 
In \cite{Zhou2019}, where the class of TCP controllers was introduced, the use of Infinitesimal Perturbation Analysis (IPA) \cite{Cassandras2010b} is extensively discussed so as to evaluate $\nabla R_{i}(t^{k})$ (hence, $\nabla J_{T}(\Theta)$) on-line and in a distributed manner. This enables the use of a gradient descent algorithm: 
\begin{equation}
\Theta^{(l+1)}=\left[  \Theta^{(l)}-\beta^{(l)}\nabla J_{T}(\Theta
^{(l)})\right]  ^{+}, \label{Eq:GradientDescent}%
\end{equation}
to update the TCP $\Theta$ iteratively. In \eqref{Eq:GradientDescent}, the projection operator $[\cdot]^{+}=\max\{0,\cdot\}$ is used. The step size $\beta^{(l)}$ is selected such that it diminishes according to the standard conditions $\sum_{l=1}^{\infty}\beta^{(l)}=\infty$ and $\lim_{l\rightarrow \infty}\beta^{(l)}=0$ \cite{Bertsekas2016nonlinear}. Note that each iteration $l$ of \eqref{Eq:GradientDescent} uses the data collected from a single trajectory (i.e., $\forall t \in[0,T]$) to evaluate $\nabla J_{T}(\Theta^{(l)})$.

The work in \cite{Zhou2019} uses a hybrid system model to construct realizations of this persistent monitoring system. However, in our formulation above, we have shown that it can be done via using a simple discrete event system model \eqref{Eq:DiscreteEventSystem}. The use of a DES model results in faster and efficient simulations and provides more intuition about the underlying decision making process. However, this modeling discrepancy will not affect any of our comparisons/conclusions made with respect to \cite{Zhou2019}.

\paragraph*{\textbf{Initialization:} $\Theta^{(0)}$} 
The work in \cite{Zhou2019} has used a randomly generated set of initial thresholds as $\Theta^{(0)}$ for \eqref{Eq:GradientDescent}. Due to the non-convexity of the objective function \eqref{Eq:ObjectiveFunction2}, the resulting value of $\Theta$ when \eqref{Eq:GradientDescent} converges is a local minimum that depends heavily on $\Theta^{(0)}$. Hence, a carefully selected high-performing $\Theta^{(0)}$ can be expected to provide significant improvements over the local minimum obtained from randomly selected $\Theta^{(0)}$. Motivated by this idea, we first investigate the structural and behavioral properties of the underlying PMN system. Then, that knowledge is used to construct a candidate for $\Theta^{(0)}$.

\paragraph*{\textbf{Simulation Results}} 
In the ensuing discussion, we consider the problem configurations shown in Fig. \ref{Fig:SASE1RandomInitialization}(a) and Fig. \ref{Fig:MASE1RandomInitialization}(a) as running examples. In those diagrams, blue circles represent the targets while black lines represent available path segments that agents can take to travel between targets. Red triangles and the yellow vertical bars indicate the agent locations and the target uncertainty levels, respectively. Moreover, since both of those quantities are time-varying ($s_a(t)$ and $R_i(t)$), in figures we indicate their state at the terminal time $t=T$ in a simulation where the best TCP found $\Theta^*$ is used. Note that the problem configurations shown in Fig. \ref{Fig:MASE1RandomInitialization}(a) is a multi-agent situation with $N=3$. 

The Fig. \ref{Fig:SASE1RandomInitialization}(b) and Fig. \ref{Fig:MASE1RandomInitialization}(b) shows the evolution of $J_T(\Theta^{(l)})$ when the TCP $\Theta^{(l)}$ is updated according to \eqref{Eq:GradientDescent} using gradients $\nabla J_T(\Theta^{(l)})$ given by the IPA method as proposed in \cite{Zhou2019}. The termination condition used for \eqref{Eq:GradientDescent} is $\Vert \Theta^{(l+1)}-\Theta^{(l)}\Vert_\infty \leq \epsilon$ where $\epsilon$ is a small positive number. If the termination condition occurs at the iteration $l=L$, then $\Theta^* = \Theta^{(L)}$ is used in generating the Fig.  \ref{Fig:SASE1RandomInitialization}(a) and Fig. \ref{Fig:MASE1RandomInitialization}(a) as mentioned before.  

The proposing persistent monitoring solution technique (discussed in the ensuing sections) including the method proposed in \cite{Zhou2019} were implemented in a JavaScript based simulator which is made available at \href{http://www.bu.edu/codes/simulations/shiran27/PersistentMonitoring/}{\seqsplit{http://www.bu.edu/codes/simulations/shiran27/PersistentMonitoring/}}. Readers are invited to reproduce the reported results and also to try new problem configurations using the developed interactive simulator. It should be highlighted that all the problem parameters (numerical values) can be customized in the developed simulator.
 
In simulation examples used in this paper, numerical values of the underlying problem parameters have been chosen as follows. The target parameters were chosen as $A_i=1,\ B_i=10$ and $R_i(0) = 0.5,\ \forall i \in \mathcal{V}$. Also, the used target location co-ordinates (i.e., $X_i$) are specified in each problem configuration figure. Note that in all the examples, all the targets have been placed inside a $600\times 600$ mission space. The interested time period (i.e., the time horizon) was taken as $T=500$. Each agent is assumed to follow the first order dynamics (as in \cite{Zhou2019}) with a maximum speed of $50$ units per second. The initial locations of the agents were chosen such that they are uniformly distributed among the targets at $t=0$ (i.e., $s_a(0) = X_i$ with $i = 1+(a-1)*\mathrm{round}(M/N)$). In cases where the initial TCP $\Theta^{(0)}$ is randomly generated, each finite element in $\Theta^{(0)}$ matrices is chosen from uniform random distribution $\mathrm{unif}(0,10)$. Also, when using the gradient descent in \eqref{Eq:GradientDescent}, diminishing step sizes $\beta^{(l)} = \frac{0.25}{\sqrt{l}}$ was used.  

\begin{figure}[!h]
     \centering
     \begin{subfigure}[b]{0.35\columnwidth}
         \centering
         \includegraphics[width = \textwidth]{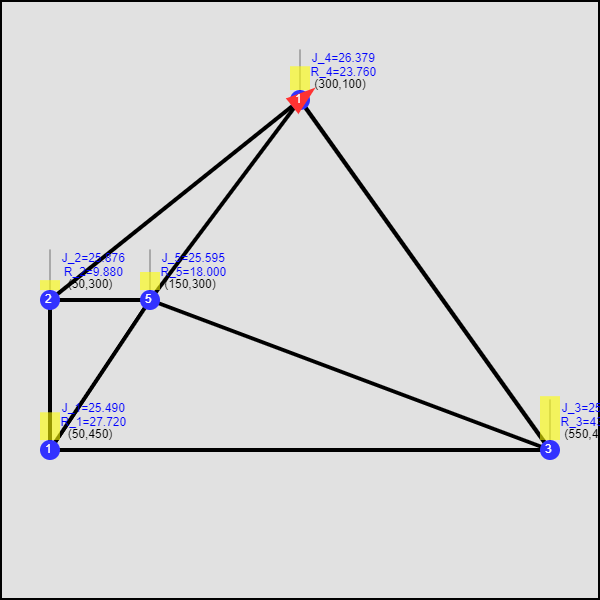}
         \caption{Config. at $t=T$.}
         
     \end{subfigure}
     \hfill
     \begin{subfigure}[b]{0.6\columnwidth}
         \centering
         \includegraphics[width = \textwidth]{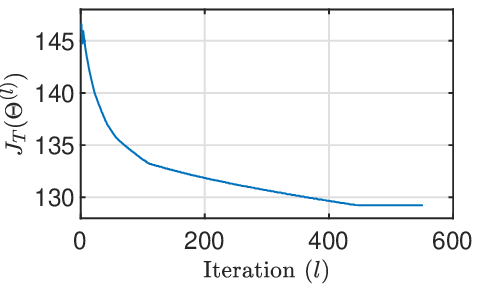}
         \caption{Cost vs iterations plot.}
         
     \end{subfigure}
    \caption{Single agent simulation example 1 (SASE1): Starting with a random $\Theta^{(0)}$, converged to a TCP with the cost $J_T = 129.2$.}
    \label{Fig:SASE1RandomInitialization}
\end{figure}

\begin{figure}[!h]
     \centering
     \begin{subfigure}[b]{0.35\columnwidth}
         \centering
         \includegraphics[width = \textwidth]{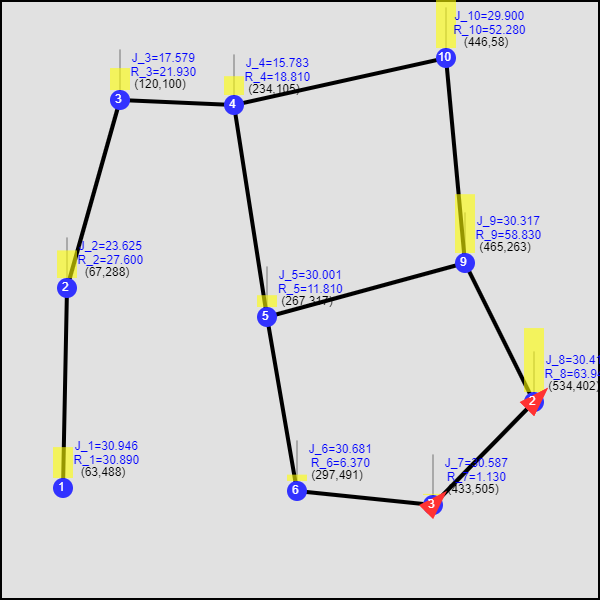}
         \caption{Config. at $t=T$.}
         
     \end{subfigure}
     \hfill
     \begin{subfigure}[b]{0.6\columnwidth}
         \centering
         \includegraphics[width = \textwidth]{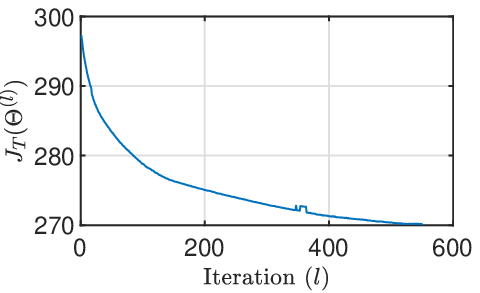}
         \caption{Cost vs iterations plot.}
         
     \end{subfigure}
    \caption{Multi-agent simulation example 1 (MASE1): Starting with a random $\Theta^{(0)}$, converged to a TCP with the cost $J_T = 270.2$.}
    \label{Fig:MASE1RandomInitialization}
\end{figure}

\section{Single-Agent PMN Solution on Sufficiently Dense Graphs}\label{Sec:DenseGraphsSingleAgent}
\paragraph*{\textbf{Main Conclusions of \cite{Zhou2019}}} 
It is proved in \cite{Zhou2019} that in a single-agent persistent monitoring system it is optimal to make the target uncertainty $R_{i}(t)=0$ on each visit of agent $a$ at target $i$. In other words, in the OTCP, $\theta_{ii}^{a}=0$. Moreover, experimental results in \cite{Zhou2019} provide some intuition about better performing agent behaviors: 
(\romannum{1}) after a brief initial transient phase, each agent converges to a (steady-state) periodic behavior where it cycles across a fixed subset of targets, and, 
(\romannum{2}) in this steady state, agents do not tend to share targets with other agents.

\paragraph*{\textbf{Our Approach}} 
We now aim to use the aforementioned observations to efficiently construct better performing (favorable) agent trajectories which can be used to initialize the TCP allowing the gradient descent scheme \eqref{Eq:GradientDescent} to achieve much better performance compared to random initialization approach used in \cite{Zhou2019}. Such trajectories take the form of a \emph{target-cycle} on the given graph. Therefore, we need to construct a set of target-cycles (one per agent) in the given graph topology.

In this section, as a starting point (and to make the problem tractable), we only focus on single-agent persistent monitoring problems on \emph{sufficiently dense} target topologies. More precisely, we consider a given target topology $\mathcal{G} = (\mathcal{V},\mathcal{E})$ to be a `sufficiently dense' graph, if $\mathcal{G}$ is \emph{bi-triangular}. We formally define the concept of bi-triangularity in Definition \ref{Def:BiTriangular}. 

\begin{definition} \label{Def:BiTriangular}
A directed graph $\mathcal{G} = (\mathcal{V},\mathcal{E})$ with $\vert\mathcal{V} \vert> 3$ is \textbf{bi-triangular} if for all $(i,j) \in\mathcal{E}$ there exists $k,l \in\mathcal{V}$ such that $(i,k),(k,j) \in\mathcal{E}$, $(i,l),(l,j)\in\mathcal{E}$, and $k \neq l$.
\end{definition}

The following assumption formally states the conditions we assume in the analysis given in this section. However, it is worth mentioning that, in the forthcoming sections we will completely relax this assumption.   

\begin{assumption}
\label{As:ProblemConfigurationType1} 
(\romannum{1}) Only one agent is available (i.e., $\mathcal{A}=\{a\}$) and 
(\romannum{2}) The given target topology $\mathcal{G}=(\mathcal{V},\mathcal{E})$ is bi-triangular.
\end{assumption}

Under this Assumption \ref{As:ProblemConfigurationType1}, we only search for a single target-cycle (agent trajectory) in the given graph $\mathcal{G}$. Moreover, exploiting the assumed dense nature of the given graph, we propose an iterative greedy scheme to construct a high performing target-cycle. This constructed agent trajectory is then transformed to a TCP as $\Theta^{(0)}$ for the subsequent use in gradient descent \eqref{Eq:GradientDescent} to obtain an OTCP $\Theta^*$. 

We note that this single-agent persistent monitoring setup was introduced in \cite{Welikala2019P3} without proofs or explicit algorithms. Sections \ref{Sec:SparseGraphsSingleAgent} and \ref{Sec:MultiAgents} provide the generalizations to arbitrary networks and multi-agent systems respectively, both not included in \cite{Welikala2019P3}.

\subsection{Analysis of an Unconstrained Target-Cycle}

We formally define a \emph{target-cycle} as a finite sequence of targets selected from $\mathcal{V}$ of the given graph $\mathcal{G} = (\mathcal{V},\mathcal{E})$ such that the corresponding sequence of edges also exists in $\mathcal{E}$. An \emph{unconstrained target-cycle} is a target-cycle with no target on it being repeated. We define $\mathcal{C}$ to be the set of all possible unconstrained target-cycles on the graph $\mathcal{G}$. A generic unconstrained target-cycle in $\mathcal{C}$ is denoted by $\Xi_{i}=\{i_{1},i_{2},\ldots,i_{m}\}\subseteq\mathcal{V}$, where $i_{j}\in \mathcal{V},\ \forall j\in\{1,2,\ldots,m\}$ and $m=\vert\Xi_{i} \vert\leq M$. The corresponding sequence of edges (fully defined by $\Xi_{i}$) are denoted by $\xi_{i} = \{(i_{m},i_{1}),(i_{1},i_{2}),\ldots,(i_{m-1},i_{m})\}\subseteq\mathcal{E}$.

Since we aim to greedily construct a target-cycle which results in a high-performing mean system uncertainty value (i.e., a low $J_{T}$ in \eqref{Eq:ObjectiveFunction1}), we need to have an assessment criterion for any given arbitrary target-cycle. Thus, we define the \emph{steady state mean cycle uncertainty} metric $J_{ss}(\Xi_{i})$:
\begin{equation}
J_{ss}(\Xi_{i})=\lim_{T\rightarrow\infty}\frac{1}{T}\int_{0}^{T}\sum_{j\in
\Xi_{i}}R_{j}(t)dt. \label{Eq:SteadyStateMeanCycleUncertainty0}%
\end{equation}
We now present a computationally efficient off-line method to evaluate $J_{ss}(\Xi_{i})$ for any $\Xi_{i}\in\mathcal{C}$. For notational convenience, we first relabel $\Xi_{i}$ and its targets as $\Xi=\{1,2,\ldots,n,n+1,\ldots,m\}$ by omitting the subscript $i$ (see Fig. \ref{Fig:cycleGeometry1}). We then make the following assumption regarding the agent's behavior on a corresponding target-cycle $\Xi\in\mathcal{C}$.

\begin{assumption}
\label{As:AgentBehavior1} 
After visiting a target $n \in\Xi$, the agent will leave it if and only if the target uncertainty $R_{n}$ reaches zero.
\end{assumption}

Here, the `only if' component follows from the aforementioned theoretical result in \cite{Zhou2019}: it is optimal to make the target uncertainty $R_{n}(t)=0$ whenever the agent visits target $n\in\Xi$. The `if' component restricts agent decisions by assuming the existence of an active neighbor to target $n$ as soon as $R_{n}(t)=0$ occurs in \eqref{Eq:DiscreteEventSystem}. At this point, we recall that our main focus is only on initializing\ \eqref{Eq:GradientDescent} and thus any potential sub-optimalities arsing from the use of Assumption \ref{As:AgentBehavior1} will be compensated by the eventual use of \eqref{Eq:GradientDescent}.

\begin{figure}[!h]
    \centering
    \includegraphics[width=3in]{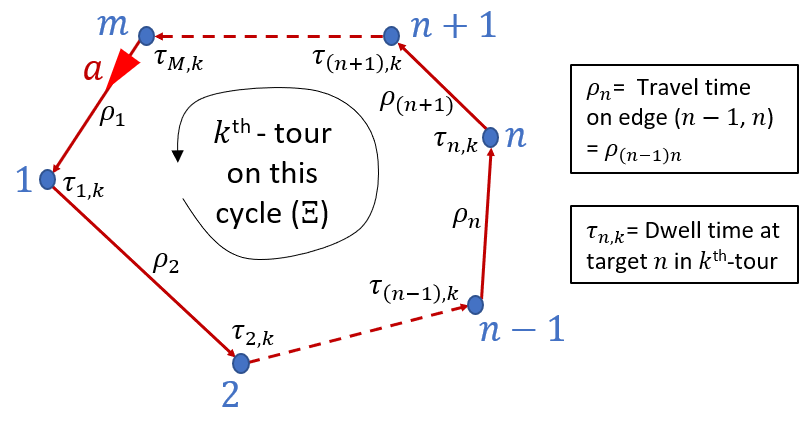}
    \caption{A generic single agent unconstrained cycle $\Xi$.}
    \label{Fig:cycleGeometry1}
\end{figure}

A \emph{tour} on the target-cycle $\Xi$ (shown in Fig. \ref{Fig:cycleGeometry1}) starts/ends when the agent (i.e., $a$) leaves the last target $m$ to reach target $1$. The dwell time spent on a target $n\in \Xi$ when the agent is in its $k$\textsuperscript{th} tour on $\Xi$ is denoted as $\tau_{n,k}^{a}$ and the travel time spent on an edge $(n-1,n)\in \mathcal{E}$ is $\rho_{(n-1)n}$ by definition. Without any ambiguity, we use the notation $\tau_{n,k}$ and $\rho_{n}$ (with $\rho_{1}=\rho_{m1}$) to represent these two quantities respectively. Moreover, target $n$'s uncertainty level at the end of the $k$\textsuperscript{th} tour is denoted by $R_{n,k}$. Under this notation, the trajectory of the target uncertainty $R_{n}(t)$ over $k$\textsuperscript{th} and $(k+1)$\textsuperscript{th} tours is shown in Fig. \ref{Fig:cycleGraphs1}.

\begin{figure}[!h]
    \centering
    \includegraphics[width=3.2in]{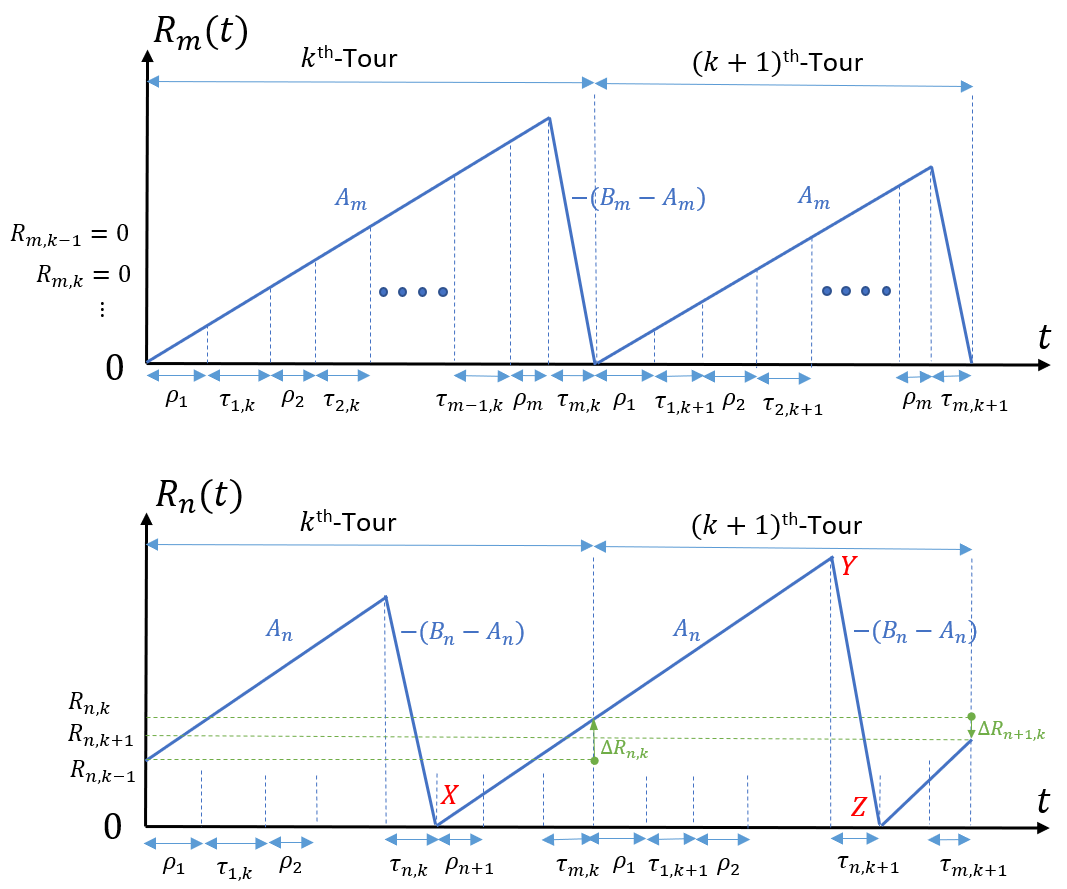}
    \caption{Variation of target uncertainties during agent tours.}
    \label{Fig:cycleGraphs1}
\end{figure}

The geometry of the $XYZ$ triangle shown in Fig. \ref{Fig:cycleGraphs1} can be used to derive the dynamics of target $n$'s dwell time $\tau_{n,k}$ (w.r.t. $k$) as
\begin{equation}
\label{Eq:DwellTimeDynamics1}(B_{n}-A_{n})\tau_{n,k+1} = A_{n} \Big( \sum
_{i=n+1}^{m}\left[  \rho_{i} + \tau_{i,k}\right]  + \sum_{i=1}^{n-1}\left[
\rho_{i}+\tau_{i,k+1}\right]  + \rho_{n} \Big).
\end{equation}
Setting $\alpha_{n} \triangleq\frac{B_{n}-A_{n}}{A_{n}}$ and $\rho_{\Xi} \triangleq\sum_{i=1}^{m}\rho_{i}$ (the total cycle travel time), the above relationship \eqref{Eq:DwellTimeDynamics1} can be simplified as:
\begin{equation}\label{Eq:DwellTimeDynamics2}
        -\sum_{i=1}^{n-1} \tau_{i,k+1} +  \alpha_n \tau_{n,k+1} = \rho_{\Xi} + \sum_{i=n+1}^m \tau_{i,k}.
\end{equation}
Now, \eqref{Eq:DwellTimeDynamics2} can be written for all $n\in\Xi$ in a compact form using the vectors $\bar{\tau}_k = [\tau_{1,k}, \tau_{2,k}, \ldots , \tau_{m,k}]^T$, $\bar{\alpha} = [\alpha_{1}, \alpha_{2}, \ldots , \alpha_{m}]^T$ and $\bar{1}_m = [1,1,\ldots,1]^T\in \R^m$, as,
\begin{equation}\label{Eq:DwellTimeDynamics3}
    \Delta_1 \bar{\tau}_{k+1} = \Delta_2 \bar{\tau}_k + \bar{1}_m \rho_{\Xi},
\end{equation}
where $\Delta_{2}\in\mathbb{R}^{m \times m}$ is the strictly upper triangular matrix with all non-zero elements being 1 and $\Delta_{1} = diag(\bar{\alpha})-\Delta_{2}^{T}$. The affine linear system expression in \eqref{Eq:DwellTimeDynamics3} describes the evolution of agent dwell times at targets on the target-cycle $\Xi$ over the number of tours completed $k$. To get an explicit expression for the steady state mean cycle uncertainty $J_{ss}(\Xi)$ defined in \eqref{Eq:SteadyStateMeanCycleUncertainty0}, we make use of the following three lemmas.

\begin{lemma}
\label{Lm:ShermonMorrison} 
(Shermon-Morrison lemma, \cite{Miller1981}) Suppose $A\in\mathbb{R}^{m\times m}$ is an invertible matrix and $u,v\in \mathbb{R}^{m\times1}$ are vectors. Then, $det(A+uv^{T})=(1+v^{T}A^{-1}u)det(A)$ and
\[
(1+v^{T}A^{-1}u)\neq0\iff(A+uv^{T})^{-1}=A^{-1}-\frac{A^{-1}uv^{T}A^{-1}%
}{1+v^{T}A^{-1}u}.
\]
\end{lemma}

\begin{lemma}
\label{Lm:SteadyStateDwellTimes1} 
When $\sum_{i=1}^{m} \frac{A_{i}}{B_{i}} < 1$, the dynamic system given in \eqref{Eq:DwellTimeDynamics3} has a feasible equilibrium point $\bar{\tau}_{eq}$ (reached at $k=k_{eq}$),
\begin{equation}
\label{Eq:SteadyStateDwellTimes1}
\bar{\tau}_{eq} = \left(  \frac{\bar{\beta}}{1-\bar{1}_{m}^{T}\bar{\beta}}\right)\rho_{\Xi}, \ \mbox{ i.e., }
\tau_{n,k_{eq}} = \left(  \frac{\beta_{n}}{1-\sum_{i=1}^{m} \beta_{i}}\right)
\rho_{\Xi},
\end{equation}
for all $n \in\Xi$ with $\beta_{n}\triangleq\frac{A_{n}}{B_{n}}$ and
$\bar{\beta} = [\beta_{1},\beta_{2},\ldots,\beta_{m}]^{T}$.
\end{lemma}
\begin{proof}
At $k=k_{eq}$, in \eqref{Eq:DwellTimeDynamics3}, $\bar{\tau}_{k+1} = \bar{\tau}_{k} = \bar{\tau}_{eq}$. Therefore,
\begin{align}
\bar{\tau}_{eq} = (\Delta_1 - \Delta_2)^{-1}\bar{1}_m \rho_{\Xi}.
\end{align}
Using $\Delta_1 = diag(\bar{\alpha}) - \Delta_2^T$ and $diag(\bar{1}_m)+\Delta_2^T + \Delta_2 = \bar{1}_m \bar{1}_m^T$,
\begin{align}
    \bar{\tau}_{eq} = (diag(\bar{\alpha} + \bar{1}_m) - \bar{1}_m \bar{1}_m^T)^{-1} \bar{1}_m \rho_{\Xi}.
\end{align}
The expressions used for $\alpha_n$ and $\beta_n$ gives that $(\alpha_n+1)=1/\beta_n$. Therefore, $(diag(\bar{\alpha} + \bar{1}_m))^{-1} = diag(\bar{\beta})$. Also, note that $diag(\bar{\beta})\bar{1}_m = \bar{\beta}$ and $I_m\in \R^{m \times m}$ is an identity matrix. Now, using Lemma \ref{Lm:ShermonMorrison},
\begin{align}
    \bar{\tau}_{eq} 
    &= diag(\bar{\beta})
    \left(I_m + \frac{\bar{1}_m\bar{1}_m^T diag(\bar{\beta})}{1-\bar{1}_m^T\bar{\beta}}\right) \bar{1}_m\rho_{\Xi}\nonumber\\
    &= diag(\bar{\beta})
    \left(\bar{1}_m + \frac{\bar{1}_m\bar{1}_m^T \bar{\beta}}{1-\bar{1}_m^T\bar{\beta}}\right) \rho_{\Xi}\nonumber\\
    &= \left(\frac{\bar{\beta}}{1-\bar{1}_m^T\bar{\beta}}\right) \rho_{\Xi}.
\end{align}
Components of $\bar{\tau}_{eq}$ are non-negative only when 
$1-\bar{1}_m^T\bar{\beta} > 0$. Thus, using the definition of $\bar{\beta}$, we get $\bar{1}_m^T\bar{\beta} = \sum_{i=1}^m \frac{A_i}{B_i} < 1$.
\end{proof}

In order to establish the stability properties of $\bar{\tau}_{eq}$ given by the Lemma \ref{Lm:SteadyStateDwellTimes1}, we need make the following assumption.

\begin{assumption}
\label{As:SchurStability} 
The matrix $\Delta_{1}^{-1}\Delta_{2}$ is Schur stable \cite{Bof2018}.
\end{assumption}

Note that the eigenvalues of $\Delta_{2}$ are located at the origin as it is a strictly upper triangular matrix. Further, the eigenvalues of $\Delta_{1}^{-1}$ are located at $\{\frac{1}{\alpha_{i}}:i\in\Xi\}$ since $\Delta_{1}$ is a lower triangular matrix with its diagonal elements being $\{\alpha_{i}:i\in\Xi\}$. Using the definition of $\alpha_{i}(=\frac{B_{i}-A_{i}}{A_{i}})$, it is easy to show that $\vert\frac{1}{\alpha_{i}}\vert<1\iff0 \leq \frac{A_{i}}{B_{i}}<\frac{1}{2}$, which is a less restrictive condition than $\sum_{i=1}^{m}\frac{A_{i}}{B_{i}}<1$ required in Lemma \ref{Lm:SteadyStateDwellTimes1}. Therefore, it seems reasonable to conjecture that the eigenvalues of $\Delta_{1}^{-1}\Delta_{2}$ are located within the unit circle; however, to date, we have not provided a formal proof to the statement in Assumption \ref{As:SchurStability}. Nevertheless, since both $\Delta_{1}$ and $\Delta_{2}$ matrices are known, the validity of this assumption for any given system \eqref{Eq:DwellTimeDynamics3} can be verified easily.

\begin{lemma}
\label{Lm:StabilityOfDWellTimeDynamics1} 
Under Assumption \ref{As:SchurStability}, the equilibrium point $\bar{\tau}_{eq}$ given in Lemma \ref{Lm:SteadyStateDwellTimes1} for the affine linear system \eqref{Eq:DwellTimeDynamics3} is globally asymptotically stable (i.e., $\lim_{k\rightarrow\infty}\bar{\tau}_{k}=\bar{\tau}_{eq}$, irrespective of $\bar{\tau}_{0}$).
\end{lemma}
\emph{Proof: }
Let $\bar{e}_k = \bar{\tau}_k - \bar{\tau}_{eq}$ as the steady state error. Then, we can write $\bar{e}_{k+1} = \bar{\tau}_{k+1} - \bar{\tau}_{eq}$ and using \eqref{Eq:DwellTimeDynamics3} and Lemma \ref{Lm:SteadyStateDwellTimes1}, 
\begin{align}
    \bar{e}_{k+1} &= (\Delta_1^{-1}\Delta_2 \bar{\tau}_k + \Delta_1^{-1}\bar{1}_m \rho_{\Xi}) - (\Delta_1^{-1}\Delta_2 \bar{\tau}_{eq} + \Delta_1^{-1}\bar{1}_m \rho_{\Xi})\nonumber \\
    &= \Delta_1^{-1}\Delta_2 \bar{e}_{k}.
\end{align}
Therefore, under Assumption \ref{As:SchurStability}, all the eigenvalues of $\Delta_{1}^{-1}\Delta_{2}$ are within the unit circle. Thus, the equilibrium point $\bar{\tau}_{eq}$ given in \eqref{Eq:SteadyStateDwellTimes1} of \eqref{Eq:DwellTimeDynamics3} is globally asymptotically stable \cite{Bof2018}. (i.e., $\lim_{k\rightarrow\infty}\bar{\tau}_{k}=\bar{\tau}_{eq}$, irrespective of $\bar{\tau}_{0}$). \hfill$\blacksquare$

We now present our main theorem regarding the steady state mean cycle uncertainty \eqref{Eq:SteadyStateMeanCycleUncertainty0} of the PMN system in Fig. \ref{Fig:cycleGeometry1}.

\begin{theorem}
\label{Th:SteadyStateMeanCycleUncertainty} 
Under Assumptions \ref{As:AgentBehavior1} and \ref{As:SchurStability} with $\sum_{i=1}^{m}\frac{A_{i}}{B_{i}}<1$, the generic (single-agent) unconstrained target-cycle $\Xi$ in Fig. \ref{Fig:cycleGeometry1} achieves a steady state mean cycle uncertainty value (i.e., \eqref{Eq:SteadyStateMeanCycleUncertainty0}),
\begin{equation}
\label{Eq:SteadyStateMeanCycleUncertainty1}J_{ss}(\Xi) = \frac{1}{2}(\bar
{B}-\bar{A})^{T}\bar{\tau}_{eq},
\end{equation}
where $\bar{B} = [B_{1},B_{2}, \ldots, B_{m}]^{T},\ \bar{A} = [A_{1},A_{2},\ldots,A_{m}]^{T}$, and $\bar{\tau}_{eq}$ is given in \eqref{Eq:SteadyStateDwellTimes1}.
\end{theorem}
\begin{proof}
Under the given conditions, both Lemma \ref{Lm:SteadyStateDwellTimes1} and Lemma \ref{Lm:StabilityOfDWellTimeDynamics1} apply. Therefore, using \eqref{Eq:SteadyStateMeanCycleUncertainty0} we can write, 
\begin{align}
    J_{ss}(\Xi) = \lim_{T\rightarrow \infty} \frac{1}{T} \int_0^T \sum_{n=1}^m R_n(t)dt = \frac{1}{T_{\Xi}} \int_{\partial T_{\Xi}} \sum_{n=1}^m R_n(t)dt,
\end{align}
where, $T_{\Xi}\triangleq\rho_{\Xi}+\bar{1}_{m}^{T}\bar{\tau}_{eq}$ is the \emph{steady state tour duration} and $\partial T_{\Xi}$ is a time period corresponding to a tour occurring after achieving the steady state. This can be further simplified into 
\begin{align}
    J_{ss}(\Xi) = \sum_{n=1}^m \frac{1}{T_{\Xi}} \int_{\partial T_{\Xi}} R_n(t)dt.
\end{align}
Now, using the $R_{n}(t)$ trajectory shown in Fig. \ref{Fig:cycleGraphs1} note that when the equilibrium is achieved (as $T\rightarrow\infty\implies k\rightarrow\infty$), the final tour uncertainties will remain stationary (i.e., $R_{n,k}=R_{n,k+1}, \forall n \in\Xi$). As a result, the area under the $R_{n}(t)$ trajectory evaluated over a period $T_{\Xi}$ becomes equivalent to that of a triangle where the base is $T_{\Xi}$ and the height is $(B_{n}-A_{n})\tau_{n,\infty}, \forall n \in\Xi$. Therefore,
\begin{align}
    J_{ss}(\Xi) &= \sum_{n=1}^m \frac{1}{T_{\Xi}} \frac{1}{2} T_{\Xi} (B_n-A_n) \tau_{n,\infty}\nonumber \\
     &= \frac{1}{2} (\bar{B} - \bar{A})^T \bar{\tau}_{eq},
\end{align}
where $\bar{\tau}_\infty = \bar{\tau}_{eq}$ given in Lemma \ref{Lm:SteadyStateDwellTimes1}.
\end{proof}

Theorem \ref{Th:SteadyStateMeanCycleUncertainty} provides a means of assessing simple persistent monitoring configurations (like the one shown in Fig. \ref{Fig:cycleGeometry1}) without having to simulate them. We will next discuss the usage of Theorem \ref{Th:SteadyStateMeanCycleUncertainty} in constructing a better performing target-cycle - on the given target topology $\mathcal{G}$.

\subsection{Greedy Target-Cycle Construction}
Under Assumption \ref{As:ProblemConfigurationType1} for the given target topology $\mathcal{G}$, if $\vert \mathcal{C} \vert$ is small, Theorem \ref{Th:SteadyStateMeanCycleUncertainty} can be used to directly identify the best performing (steady state) target-cycle via brute-force search:
\begin{equation}
\Xi^{\ast}=\arg\min_{\Xi\in\mathcal{C}}J_{ss}(\Xi).
\label{Eq:OptimumUnconstrainedCycleSearch}%
\end{equation}

Since this brute-force approach becomes computationally expensive as $\vert\mathcal{C} \vert$ grows exponentially with the number of targets or edges, we propose instead a computationally efficient greedy scheme to construct a sub-optimal target-cycle (denoted as $\Xi^{\#}$) as a candidate for $\Xi^{*}$ in \eqref{Eq:OptimumUnconstrainedCycleSearch}. In this greedy scheme, each iteration search expands a current target-cycle $\Xi$ by adding an unvisited target $i \in\mathcal{V} \backslash\Xi$ to $\Xi$ ($\cdot \backslash\cdot$ is the set subtraction operation). The constructed $\Xi^{\#}\in\mathcal{C}$ is then transformed to a TCP which is used as $\Theta^{(0)}$ in \eqref{Eq:GradientDescent}. Therefore, determining the optimal target-cycle $\Xi^{*}$ is not essential at this stage as opposed to the importance of keeping the overall process of obtaining $\Theta^{(0)}$ efficient.

\paragraph*{\textbf{Estimating Finite Horizon Objective} $J_T$} 

We now define the finite horizon version of $J_{ss}(\Xi_{i})$ in \eqref{Eq:SteadyStateMeanCycleUncertainty0} as the \emph{finite horizon mean cycle uncertainty} $J_{T}(\Xi_{i})$, where 
\begin{equation}
\label{Eq:FiniteHorizonMeanCycleUncertainty}J_{T}(\Xi_{i}) = \frac{1}{T}
\int_{0}^{T} \sum_{j \in\Xi_{i}}R_{j}(t)dt.
\end{equation}
Note that if $\mathcal{V} = \Xi_{i}$, this $J_{T}(\Xi_{i})$ metric is equivalent to the mean system uncertainty metric $J_{T}$ defined in \eqref{Eq:ObjectiveFunction1}.

\paragraph*{\textbf{Contribution of a Neglected Target}} 
Formally, a \emph{neglected target} is a target that is not visited by any agent during the period $[0,T]$. As our main objective $J_{T}$ in \eqref{Eq:ObjectiveFunction1} is evaluated over a finite horizon $T$, if one or more targets are located remotely compared to the rest of the targets, then neglecting such remote targets might be better than trying to visit them. The following lemma characterizes the contribution of such a neglected target to the main objective $J_{T}$ in \eqref{Eq:ObjectiveFunction1}.

\begin{lemma}
\label{Lm:NeglectedCost} 
The contribution of a neglected target $i\in\mathcal{V}$ to the mean system uncertainty $J_{T}$ (defined in \eqref{Eq:ObjectiveFunction1}) is $\left(R_{i,0}+\frac{A_{i} T}{2}\right)  $.
\end{lemma}

\begin{proof}
Consider the original problem configuration where we had $M$ targets: $\mathcal{V} = \{1,2,\ldots,M\}$. Under this setting, the mean system uncertainty $J_T$ defined in \eqref{Eq:ObjectiveFunction1} can be decomposed as, 
\begin{align}
    J_T &= \frac{1}{T}\int_0^T \sum_{j\in\mathcal{V}} R_j(t)dt \nonumber \\
    &= \frac{1}{T}\int_0^T \sum_{j\in\mathcal{V}\backslash \{i\}} R_j(t)dt +\frac{1}{T}\int_0^T R_i(t)dt. \nonumber
\end{align}
Therefore, the second term above represents the contribution of the target $i$ to the overall objective $J_T$. However, since target $i$ is not being visited by an agent during $t \in [0,T]$, $\dot{R}_i(t) = A_i \ \forall t \in [0,T]$. Also note that the initial target uncertainty of $i$ is $R_{i,0}$. Therefore, the contribution of target $i$ can be simplified as 
\begin{align*}
    \frac{1}{T}\int_0^T R_i(t)dt &= \frac{1}{T}\int_0^T R_{i,0} + A_i tdt\\
    &= \left(R_{i,0} + \frac{A_i T}{2} \right).
\end{align*}
\end{proof}

\begin{assumption}
\label{As:EstimationError} 
For any target-cycle $\Xi\in\mathcal{C}$, the difference between the steady state mean cycle uncertainty $J_{ss}(\Xi)$ (defined in \eqref{Eq:SteadyStateMeanCycleUncertainty0}) and the finite horizon mean cycle uncertainty $J_{T}(\Xi)$ (defined in \eqref{Eq:FiniteHorizonMeanCycleUncertainty}) is bounded by some finite constant $K_{e} \in\mathbb{R}_{\geq0}$, i.e., $\vert J_{ss}(\Xi) - J_{T}(\Xi) \vert < K_{e}$.
\end{assumption}

The greedy target-cycle construction scheme uses the $J_{ss}(\cdot)$ metric defined in \eqref{Eq:SteadyStateMeanCycleUncertainty0} to compare the performance of different target-cycles as it can be evaluated efficiently using Theorem \ref{Th:SteadyStateMeanCycleUncertainty}. However, since the original objective $J_{T}$ in \eqref{Eq:ObjectiveFunction1} is evaluated over a finite horizon $T$, the $J_{T}(\cdot)$ metric defined in \eqref{Eq:FiniteHorizonMeanCycleUncertainty} is more appropriate to evaluate (and compare) different target-cycle performances. The above assumption states that $J_{T}(\cdot)$ will always lie within $J_{ss}(\cdot) \pm K_{e}$ and we point out that $K_{e}$ is small whenever:
(\romannum{1}) the steady state tour duration $T_{\Xi}$ and the finite horizon $T$ is such that $T \gg T_{\Xi}$, and 
(\romannum{2}) the dynamics of the steady state error of \eqref{Eq:DwellTimeDynamics3} are faster (i.e., according to Lemma \ref{Lm:StabilityOfDWellTimeDynamics1}, when $\frac{A_{i}}{B_{i}} \ll1$).

\paragraph*{\textbf{Target-Cycle Expansion Operation (TCEO)}}
Recall that we used the notation $\Xi_i = \{i_1,i_2,\ldots,i_m\}$ to represent a generic target-cycle and $\xi_i = \{(i_m,i_1),(i_1,i_2),\ldots,(i_{m-1},i_m)\}$ to represent the respective sequence of edges in the target-cycle $\Xi_i$. 

Omitting the subscript $i$ (for notational convenience) consider a situation where we have a target-cycle $\Xi = \{1,2,\ldots,m\}$ with its respective edge set $\xi = \{(m,1),(1,2),\ldots,(m-1,m)\}$. As shown in Fig. \ref{Fig:basicCycleExpanding}, in order to expand $\Xi$ so that it includes one more target $i$ chosen from the set of neglected targets $\mathcal{V}\backslash\Xi$, we have to: 
(\romannum{1}) replace one edge $(n-1,n)$ chosen from $\xi$ with two new consecutive edges $(n-1,i),(i,n) \in\mathcal{E}$ and
(\romannum{2}) insert the neglected target $i$ into $\Xi$ between targets $n-1$ and $n$. 
Whenever $\vert\mathcal{V} \backslash\Xi\vert>0$, the existence of a such $i$ and $(n-1,n)$ is guaranteed by the bi-triangularity condition in Assumption \ref{As:ProblemConfigurationType1}. Upon executing these two operations, a new (expanded) target-cycle $\Xi^{\prime}$ (and $\xi^{\prime}$) is attained as shown in Fig. \ref{Fig:basicCycleExpanding}. The following theorem derives the \emph{marginal gain} (denoted as $\Delta J_{T}(i\vert\xi,(n-1,n))$) in the main objective $J_{T}$ in \eqref{Eq:ObjectiveFunction1} due to such a target-cycle expansion in terms of $J_{ss}(\cdot)$ in \eqref{Eq:SteadyStateMeanCycleUncertainty0}.   

\begin{figure}[!b]
    \centering
    \includegraphics[width=3.2in]{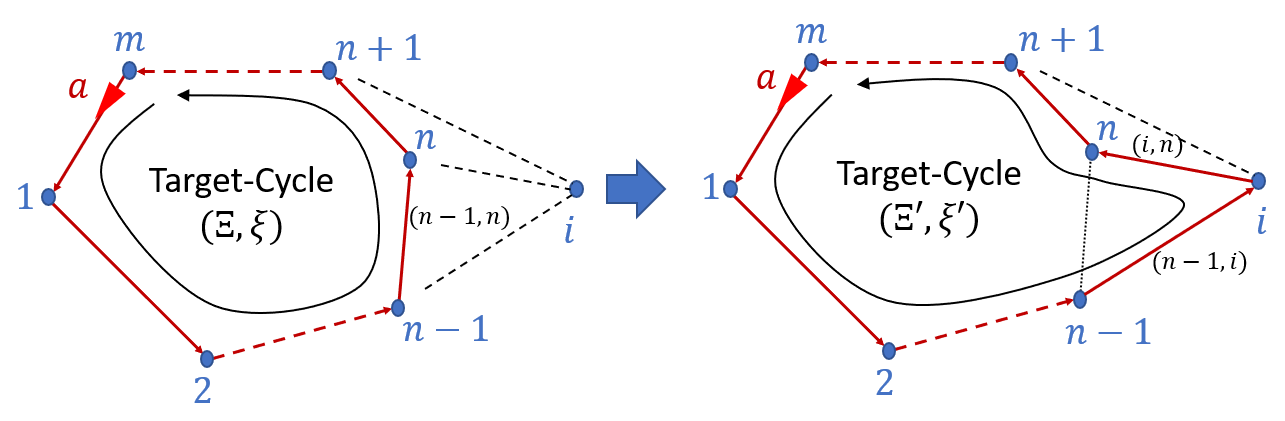}
    \caption{A basic target-cycle expanding operation (TCEO).}
    \label{Fig:basicCycleExpanding}
\end{figure}

\begin{theorem}
\label{Th:MarginalGain1} 
Under Assumptions \ref{As:ProblemConfigurationType1}, \ref{As:AgentBehavior1} and \ref{As:EstimationError}, the marginal gain in the mean system uncertainty $J_{T}$ (defined in \eqref{Eq:ObjectiveFunction1}) due to the basic target-cycle expansion operation (shown in Fig. \ref{Fig:basicCycleExpanding}) is
\begin{equation}
\label{Eq:MarginalGain1}\Delta J_{T}(i \vert\xi, (n-1,n)) = \left(  R_{i,0} +
\frac{A_{i} T}{2} \right)  + J_{ss}(\Xi) - J_{ss}(\Xi^{\prime}).
\end{equation}
Here, $\Xi^{\prime}$ is the expanded cycle and $J_{ss}(\cdot)$ is given in Theorem \ref{Th:SteadyStateMeanCycleUncertainty}. The associated estimation error of this term is $\pm2K_{e}$.
\end{theorem}

\begin{proof}
When target $i$ is neglected, Lemma \ref{Lm:NeglectedCost} gives the mean system uncertainty as
$$\left(R_{i,0} + \frac{A_i T}{2} \right) + J_{T}(\Xi).$$
After the target-cycle expansion, the mean system uncertainty is $J_{T}(\Xi^{\prime})$ (Note that now $i\in\Xi^{\prime}$ and $J_{T}(\cdot)$ is defined in \eqref{Eq:FiniteHorizonMeanCycleUncertainty}). Therefore, the gain in mean system uncertainty is 
$$\left(R_{i,0} + \frac{A_i T}{2} \right) + J_{T}(\Xi) - J_{T}(\Xi').$$
Now, adding and subtracting a $(J_{ss}(\Xi)-J_{ss}(\Xi^{\prime}))$ term and applying Assumption \ref{As:EstimationError} twice (for $J_{T}(\Xi), \ J_{T}(\Xi^{\prime})$ terms) shows that the above ``gain'' can be estimated by the marginal gain expression given in \eqref{Eq:MarginalGain1} (with a tolerance of $\pm2K_{e}$).
\end{proof}

\paragraph*{\textbf{Greedy Algorithm}}
Now, we propose the following greedy scheme (Algorithm \ref{Alg:GreedyTargetCycleConstruction1}) as a means of constructing a sub-optimal target-cycle for \eqref{Eq:OptimumUnconstrainedCycleSearch} under the given target topology. It starts from searching for the best target-cycle of length 2 (i.e. $\Xi \ni \vert \Xi \vert = 2$). The search space length is $\vert \mathcal{E} \vert$ and the obtained solution is then used as the initial target-cycle. Next, the current target-cycle is iteratively expanded by adding an external targets sequentially. The step 6 of algorithm decides the target to be added (and the edge to be removed) via brute-force search. In the $k$\textsuperscript{th} target-cycle expansion step, if $\mathcal{G}$ is fully connected, the search space length is $(k+1)\times (\vert \mathcal{V} \vert - k - 1)$, where $(k+1)$ is the number of edges in the current target-cycle and $(\vert \mathcal{V} \vert - k - 1)$ is the number of remaining neglected targets available. Thus, the search space size remains moderate through greedy iterations.

\begin{algorithm*}[t]
\caption{Greedy target-cycle construction for \eqref{Eq:OptimumUnconstrainedCycleSearch}.}
\begin{algorithmic}[1]
\State \textbf{Input}: Target topology $\mathcal{G}=(\mathcal{V},\mathcal{E})$, where, $\mathcal{V} = \{1,2,\cdots,M\}$, and, $\mathcal{E} \subseteq \{(j,l): j,l \in \mathcal{V}\}.$
\State \textbf{Output}: A sub-optimal target-cycle $\Xi^\#$ (and $\xi^\#$) for \eqref{Eq:OptimumUnconstrainedCycleSearch}.
\State $(j^*,l^*) := \substack{\arg\min \\ (j,l) \in \mathcal{E}} \left[ J_{ss}(\{j,l\}) \right]$ 
\Comment{Best 2-Target-cycle to cover}
\State $\Xi := \{j^*,l^*\}$,  
$\xi := \{(j^*,l^*),(l^*,j^*)\}$ 
\Comment{initial target-cycle}
\Do 
\Comment{Best way to expand $\Xi$}
    \State $[val,\ (i^*,(j^*,l^*)] :=
    \substack{\arg\max\\
    (i,(j,l)):i \in \mathcal{V} \backslash \Xi, \\ (j,l) \in \xi,\  (j,i)\in\mathcal{E},\ (i,l)\in\mathcal{E}} \left[ \Delta J_T(i \vert \xi,(j,l)) \right]$ 
    \Comment{The edge to remove and the target to add} 
    \State \textbf{Replace} $(j^*,l^*) \in \xi$ with $\{(j^*,i^*),(i^*,l^*)\}$
    \Comment{Updating $\xi$}
    \State \textbf{Insert} $i^*$ into $\Xi$ between $j^*$ and $l^*$.
    \Comment{Updating $\Xi$}
\doWhile{$val \geq 0$} \Comment{Until marginal gain become negative}
\State $\Xi^\# := \Xi;\ \xi^\#:= \xi$; \textbf{Return}; 
\end{algorithmic}\label{Alg:GreedyTargetCycleConstruction1}
\end{algorithm*}

\paragraph*{\textbf{TSP Inspired Target-Cycle Refinements}}

The underlying idea behind Algorithm \ref{Alg:GreedyTargetCycleConstruction1} can be seen as a heuristic tour construction approach for an initial solution to the Traveling Salesman Problem (TSP) on $\mathcal{G}$ \cite{Blazinskas2011}. However, there are two fundamental differences in our problem setup compared to a TSP: both the marginal gain function in Theorem \ref{Th:MarginalGain1} and the tour cost function in Theorem \ref{Th:SteadyStateMeanCycleUncertainty} are much more complex. In a TSP, these two functions would be simple distance-based metrics. Moreover, compared to TSPs, in persistent monitoring, a cost value, i.e., $J_{ss}(\cdot)$, cannot be assigned to individual edges of the topology, but can only be assigned to target-cycles using Theorem \ref{Th:SteadyStateMeanCycleUncertainty}.

This compatibility with TSP raises an important question: ``Can't we adopt heuristic tour construction methods used in TSP to replace Algorithm \ref{Alg:GreedyTargetCycleConstruction1}?'' The answer is: It is not possible, because, in persistent monitoring, we cannot assign a cost value to individual edges of the topology separately, we only can assign a cost (using Theorem \ref{Th:SteadyStateMeanCycleUncertainty}) for target-cycles.   

However, once we constructed a sub-optimal target-cycle (Let us denote by $\Xi^\#$) using Algorithm \ref{Alg:GreedyTargetCycleConstruction1}, we can adopt local search (also called local perturbation or tour improvement) techniques introduced in TSP literature. Specifically, we use the conventional \emph{2-Opt} and \emph{3-Opt} techniques \cite{Nilsson2003,Blazinskas2011} to further refine the obtained $\Xi^\#$. The main idea behind a step of these methods is to slightly perturb (See Fig. \ref{Fig:2Opt3Opt}) the shape of $\Xi^\#$ (say into $\Xi'$) and then to see whether $J_{ss}(\Xi')<J_{ss}(\Xi^\#)$. If so, the update $\Xi^\# := \Xi'$ is used.     

\begin{figure}[!h]
    \centering
    \includegraphics[width=3in]{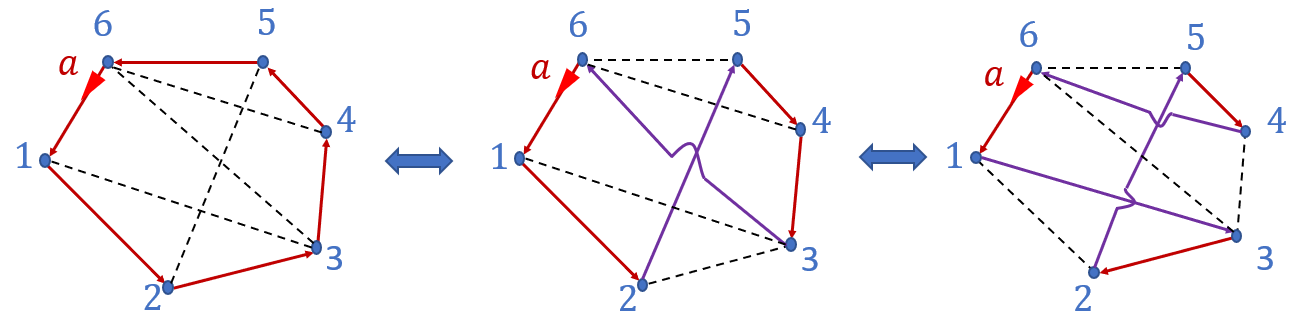}
    \caption{A target-cycle and a possible 2-Opt move and a 3-Opt move (left to right).}
    \label{Fig:2Opt3Opt}
\end{figure}

\subsection{The Initial Threshold Control Policy: $\mathrm{\Theta}^{(0)}$}
Let us denote the obtained refined sub-optimal target-cycle as $\Xi^R$ (and $\xi^R$). Now, we need to convert $\Xi^R$ into a set of threshold control policy (TCP) values: $\Theta^{(0)}$ so that it can be used in \eqref{Eq:GradientDescent} as the initial condition. 

First, note that under Assumption \ref{As:ProblemConfigurationType1}, we only have one agent (i.e., $\mathcal{A} = \{a\}$). Therefore, essentially, $\Theta = \Theta^a \in \R^{M \times M}$. Then, notice that the TCP values in $\Theta^{a(0)}$ should guide the agent $a$ according to the assumptions made in Assumption \ref{As:AgentBehavior1}. Therefore, we can directly deduce a requirement: diagonal entries of $\Theta^{a(0)}$ should be $0$. This will make sure that the agent $a$ will stay till $R_i(t)=0$ when it visited the target $i \in \Xi^R$ (See \eqref{Eq:DiscreteEventSystem}). Next, we need to make sure that under the TCP values $\Theta^{a(0)}$ the agent will follow the intended cyclic trajectory $\Xi^R$. This can be achieved by setting $(i,j)$\textsuperscript{th} entry of $\Theta^{a(0)}$ matrix to $0$ for all $(i,j) \in \xi^R$. All the other (valid) entries of $\Theta^{a(0)}$ matrix should be chosen as $P \in \R$ where $P > T_{\Xi^R}\max_i A_i$ so that agent will remain in the same target-cycle $\Xi^R$. Algorithm \ref{Alg:initialThresholPolicyGeneration} outlines this process and An example input/output for this process corresponding to the target topology in Fig. \ref{Fig:thresholds} is shown in Fig. \ref{fig:initialThresholPolicy}.

\begin{algorithm}[t]
\caption{The algorithm used to generate the initial TCP values $\Theta^{a(0)}$ from the obtained target-cycle $\Xi^R, \xi^R$.}
\begin{algorithmic}[1]
\State Input: Graph $\mathcal{G}=(\mathcal{V},\mathcal{E})$, and the target-cycle $\Xi^R, \xi^R$. 
\State Output: $\Theta^{a(0)}$ \Comment{TCP values for agent $a \in \mathcal{A}$}
\State $\Theta^1 := \mb{0}\in \R^{M \times M}$
\Comment{Placeholders for $\theta_{ij}^{a(0)}$ values}
\For{$i$ in $\Xi^R$}
    \For{$j$ in $\mathcal{V}$}
        \If{$i==j$ or $(i,j) \in \xi^R$}
            \State $\Theta^1[i][j] = 0$;
        \ElsIf{$(i,j) \in \mathcal{E}$}
            \State $\Theta^1[i][j] = P$; 
        \Else
            \State $\Theta^1[i][j] = \infty$;
        \EndIf
    \EndFor
\EndFor
\State $\Theta^{a(0)}:=\Theta^1$; \textbf{Return}; 
\end{algorithmic}\label{Alg:initialThresholPolicyGeneration}
\end{algorithm}

\begin{figure}[!h]
    \centering
    \includegraphics[width=3in]{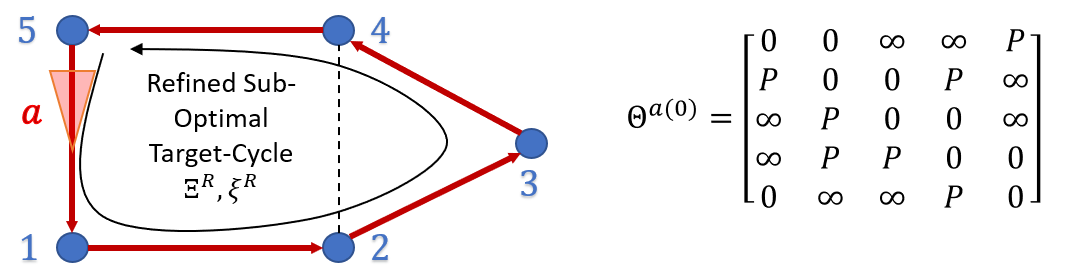}
    \caption{The generated threshold matrix (right) for the refined sub-optimal target-cycle $\Xi^R$ shown (left).}
    \label{fig:initialThresholPolicy}
    \vspace{-5mm}
\end{figure}

\subsection{Simulation Results}
Fig. \ref{Fig:UnconstrainedCycleGenerationProcess} (a)$\rightarrow$(d) shows the intermediate cycles generated by the greedy sub-optimal cycle construction process given in Algorithm \ref{Alg:GreedyTargetCycleConstruction1} when applied for the SASE1 problem configuration (first saw in Fig. \ref{Fig:SASE1RandomInitialization}). The target-cycle shown (as a red contour) in Fig. \ref{Fig:UnconstrainedCycleGenerationProcess} (d) has a $J_{ss}$ value of $135.7$. A subsequently identified profitable refinement step and its final result is shown in Fig. \ref{Fig:UnconstrainedCycleGenerationProcess} (e) and (f) respectively. The $J_{ss}$ value of the target-cycle shown in Fig. \ref{Fig:UnconstrainedCycleGenerationProcess} (f) is $128.7$.

The identified target-cycle (say $\Xi^R$) is then converted to the respective TCP using Algorithm \ref{Alg:initialThresholPolicyGeneration}. Fig. \ref{Fig:SASE1UnconstrainedCycleInitialization}(b) shows that the target-cycle $\Xi^R$ has a $J_T$ value of $121.6$ which cannot be further improved using the gradient steps \eqref{Eq:GradientDescent}. To ensure $\Xi^R$ is a local optimal, after $100$ iterations (at $l=100$), the derived TCP $\Theta^{(0)}$ is randomly perturbed. Then it can be seen that $\Theta^{(l)}$ converges back to the same initial TCP found (with $J_T=121.6$). It is important to note that this solution is better than the best TCP obtained with a random initialization of $\Theta^{(0)}$ (shown in Fig \ref{Fig:SASE1RandomInitialization}), by $+7.6\ (5.88\%)$.

In order to highlight the importance of gradient steps, consider the new single-agent simulation example (SASE2) shown in Fig. \ref{Fig:SASE2RandomInitialization}. In there, when the TCP $\Theta^{(0)}$ is selected randomly, the gradient steps have converged to $J_T = 651.3$. Now, Fig. \ref{Fig:SASE2UnconstrainedCycleInitialization}(a) shows the performance of the TCP given by the identified refined sub-optimal greedy cycle (obtained using Algorithm \ref{Alg:GreedyTargetCycleConstruction1} and \ref{Alg:initialThresholPolicyGeneration}). As the usual next step, when gradient steps are used \eqref{Eq:GradientDescent}, compared to SASE1, we can now observe a further improvement in $J_T$ (See Fig. \ref{Fig:SASE2UnconstrainedCycleInitialization}(b) and (c)) which leads to a TCP $\Theta^*$ with $J_T = 567.0$. Therefore, the overall improvement achieved from deploying the proposing initialization technique is $+84.3\ (12.9\%)$. The main difference between the solutions in Fig. \ref{Fig:SASE2UnconstrainedCycleInitialization}(a) and (b) is that in the former one, agent avoids visiting the target $4$ and strictly follows target-cycle shown in red color, whereas in the latter one, gradient descent steps have updated the TCP such that the agent trajectory now includes the target $4$.

\begin{figure*}[!h]
     \centering
     \begin{subfigure}[b]{0.3\columnwidth}
         \centering
         \includegraphics[width = \textwidth]{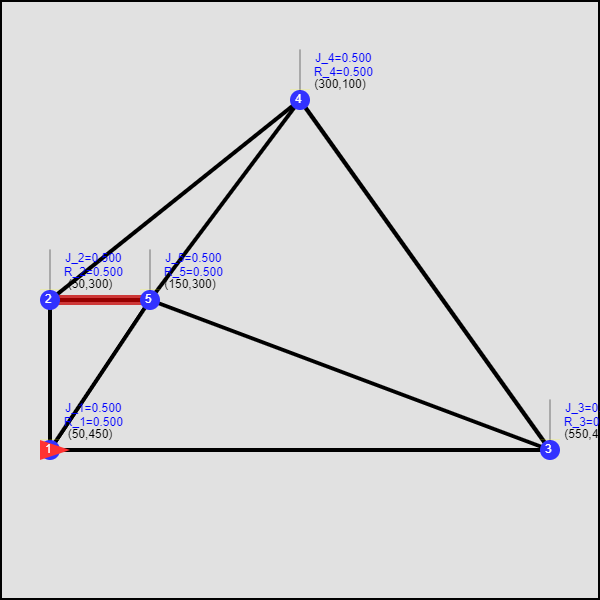}
         \caption{Greedy step 1}
         
     \end{subfigure}
     \hfill
     \begin{subfigure}[b]{0.3\columnwidth}
         \centering
         \includegraphics[width = \textwidth]{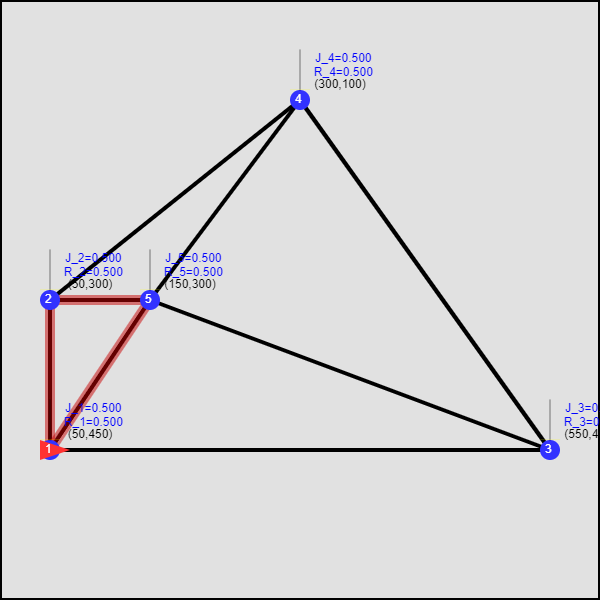}
         \caption{Greedy step 2}
         
     \end{subfigure} 
     \hfill
     \begin{subfigure}[b]{0.3\columnwidth}
         \centering
         \includegraphics[width = \textwidth]{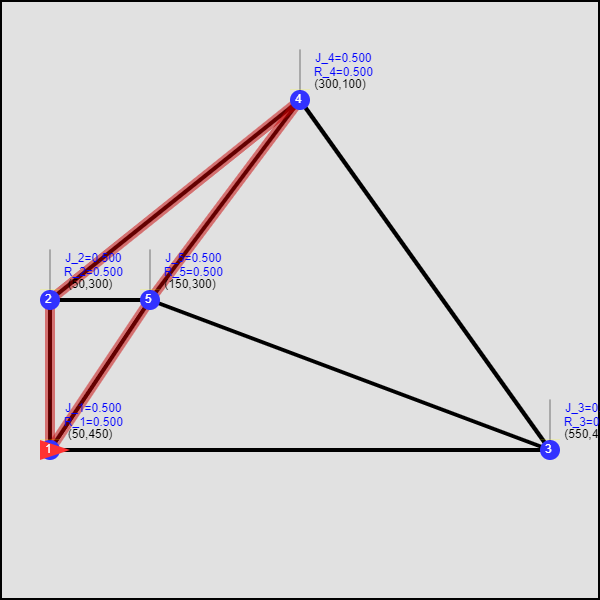}
         \caption{Greedy step 3}
         
     \end{subfigure} 
     \hfill
     \begin{subfigure}[b]{0.3\columnwidth}
         \centering
         \includegraphics[width = \textwidth]{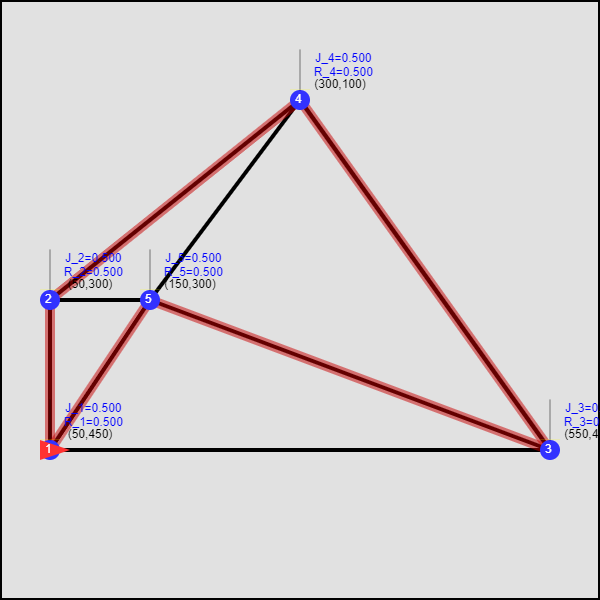}
         \caption{Greedy step 4:$\Xi^\#$}
         
     \end{subfigure} 
     \hfill
     \begin{subfigure}[b]{0.3\columnwidth}
         \centering
         \includegraphics[width = \textwidth]{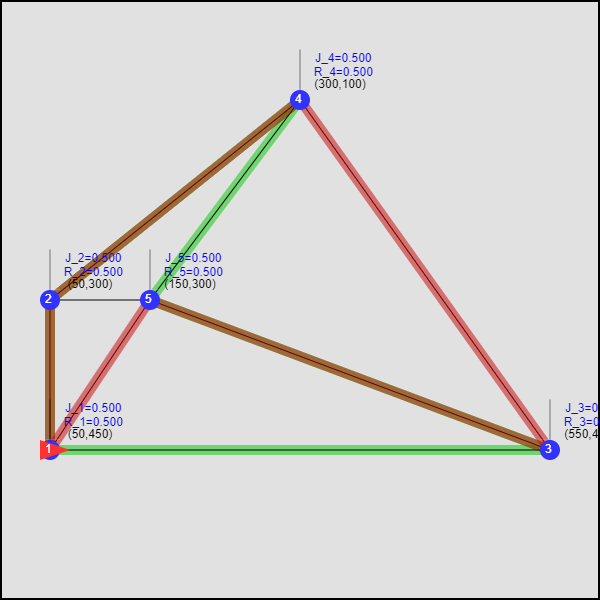}
         \caption{A 2-Opt Step}
         
     \end{subfigure} 
     \hfill
     \begin{subfigure}[b]{0.3\columnwidth}
         \centering
         \includegraphics[width = \textwidth]{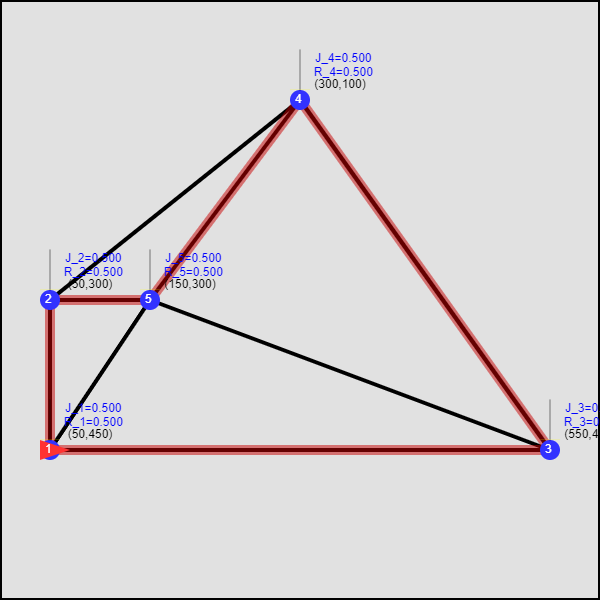}
         \caption{Refined Cycle:$\Xi^R$}
         
     \end{subfigure} 
    \caption{Greedy target-cycle construction iterations (by Algorithm \ref{Alg:GreedyTargetCycleConstruction1}) and a profitable refinement step (a 2-Opt one) observed for the target topology of SASE1.}
    \label{Fig:UnconstrainedCycleGenerationProcess}
\end{figure*}

\begin{figure}[!h]
     \centering
     \begin{subfigure}[b]{0.35\columnwidth}
         \centering
         \includegraphics[width = \textwidth]{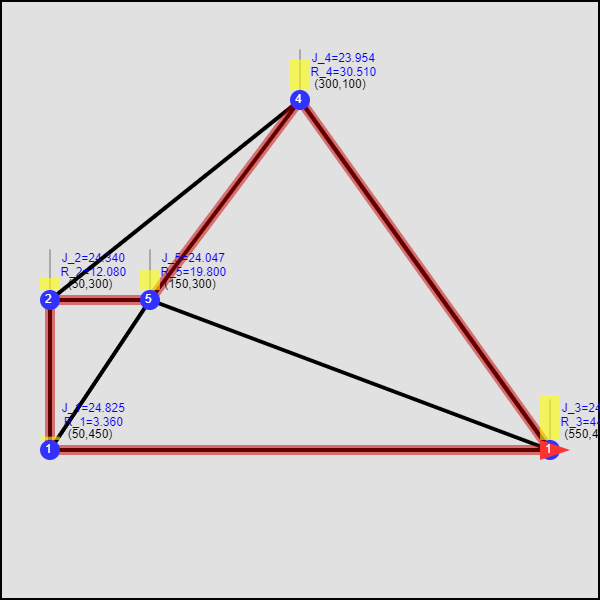}
         \caption{Config. at $t=T$.}
         
     \end{subfigure}
     \hfill
     \begin{subfigure}[b]{0.6\columnwidth}
         \centering
         \includegraphics[width = \textwidth]{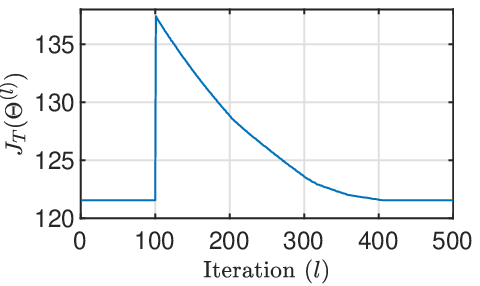}
         \caption{Cost vs iterations plot.}
         
     \end{subfigure}
    \caption{SASE1: The TCP $\Theta^{(0)}$ given by the identified cycle $\Xi^R$ (the red trace in (a)) shows local optimality. At $l=100, \Theta^{(l)}$ is randomly perturbed. Yet, converges back to the initial TCP. Cost $J_T=121.6$ (Improvement $=+7.6$ compared to Fig. \ref{Fig:SASE1RandomInitialization}).}
    \label{Fig:SASE1UnconstrainedCycleInitialization}
\end{figure}

\begin{figure}[!h]
     \centering
     \begin{subfigure}[b]{0.35\columnwidth}
         \centering
         \includegraphics[width = \textwidth]{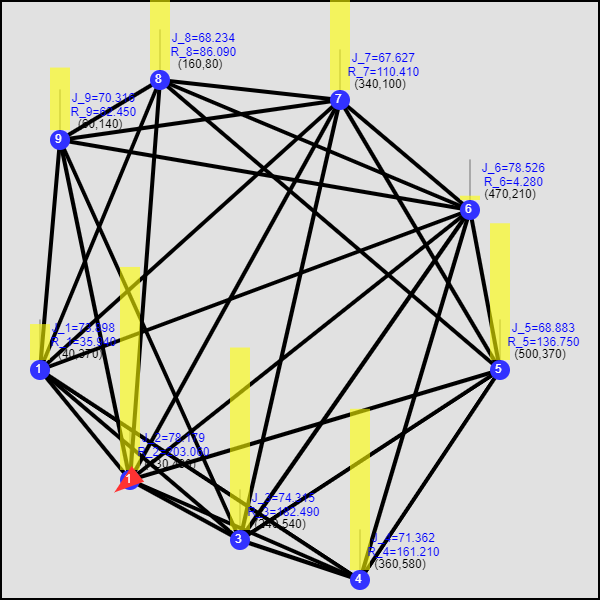}
         \caption{Config. at $t=T$.}
         
     \end{subfigure}
     \hfill
     \begin{subfigure}[b]{0.6\columnwidth}
         \centering
         \includegraphics[width = \textwidth]{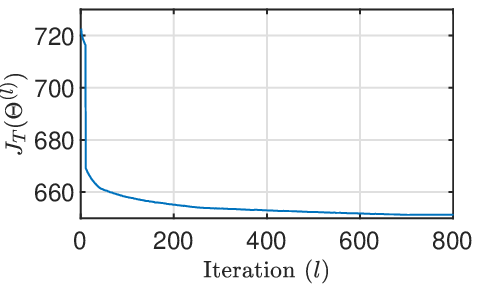}
         \caption{Cost vs iterations plot.}
         
     \end{subfigure}
    \caption{Single agent simulation example 2 (SASE2): Starting with a random $\Theta^{(0)}$, converged to a TCP with the cost $J_T = 651.3$.}
    \label{Fig:SASE2RandomInitialization}
\end{figure}

\begin{figure}[!h]
     \centering
     \begin{subfigure}[b]{0.3\columnwidth}
         \centering
         \includegraphics[width = \textwidth]{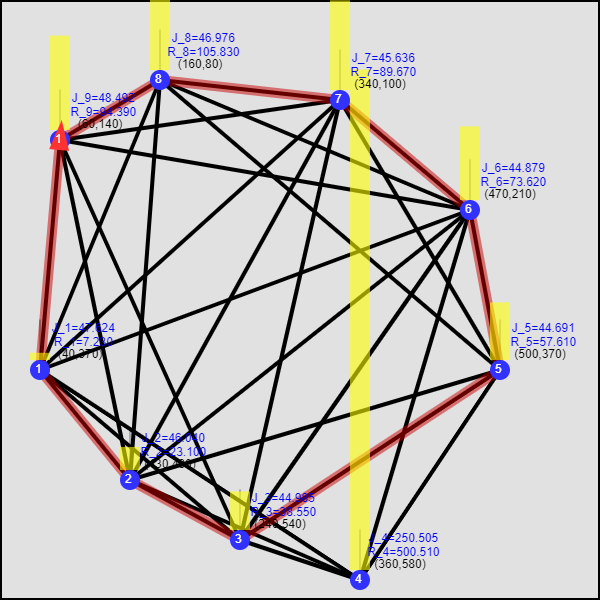}
         \caption{$l=0$, $t=T$.}
         
     \end{subfigure}
     \hfill
     \begin{subfigure}[b]{0.3\columnwidth}
         \centering
         \includegraphics[width = \textwidth]{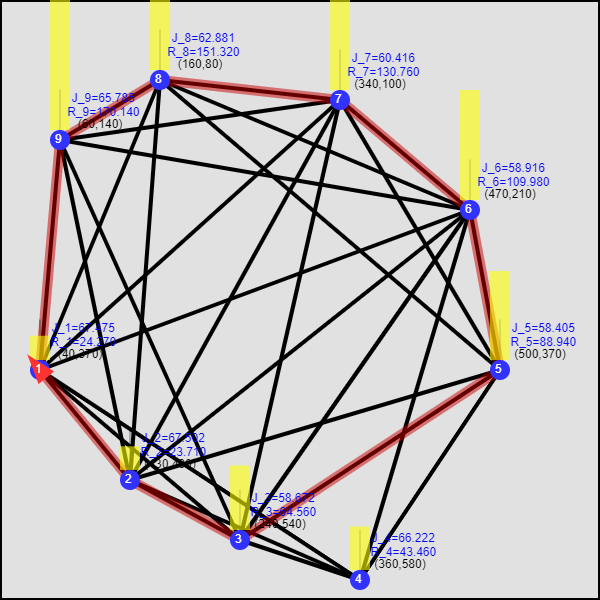}
         \caption{$l=100$, $t=T$.}
         
     \end{subfigure}
     \hfill
     \begin{subfigure}[b]{0.35\columnwidth}
         \centering
         \includegraphics[width = \textwidth]{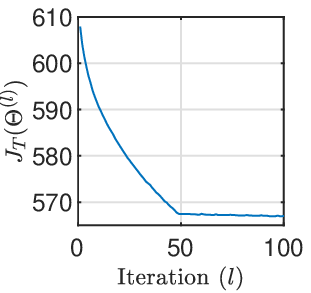}
         \caption{Cost vs iterations.}
         
     \end{subfigure}
    \caption{SASE2: The TCP $\Theta^{(0)}$ given by the identified cycle $\Xi^R$ (the red trace in (a),(b)) with cost $J_T = 607.9$ (improvement $=+43.4$ compared to Fig. \ref{Fig:SASE2RandomInitialization}) is further improved by the IPA based gradient updates \eqref{Eq:GradientDescent}. Final cost $J_T=567.0$ (Improvement $=+40.9$ compared to (a)).}
    \label{Fig:SASE2UnconstrainedCycleInitialization}
\end{figure}

\section{Single-Agent PMN Solution on Sparse (Generic) Graphs}\label{Sec:SparseGraphsSingleAgent}

According to Assumption \ref{As:ProblemConfigurationType1}, the work presented so far assumes the underlying target topology to be bi-triangular (See Def. \ref{Def:BiTriangular}). Under a such setting, we have discussed on how to: (\romannum{1}) analyze, (\romannum{2}) iteratively construct, (\romannum{3}) refine, and, (\romannum{4}) transform (to a set of TCP values), unconstrained target-cycles. However, if the target topology is sparse, the proposed target-cycle construction method can fail. 

This failure (if occurred) originates from the step 6 of the Algorithm \ref{Alg:GreedyTargetCycleConstruction1} - with brute force search operation failing due to having a null feasible search space. To illustrate this, consider the two half-constructed target-cycles shown in Fig. \ref{Fig:CycleConstructionFailure}. At this stage, either of those cycles cannot be expanded as there are no new edges which we can add to the current cycle (i.e., to $\xi$) so that those new edges connects any one of the remaining neglected targets into the current cycle (See Fig. \ref{Fig:basicCycleExpanding}). Further, note that if the graphs shown in Fig. \ref{Fig:CycleConstructionFailure} respectively had the edges $(4,2)$ and  $(4,5)$, this error does not occur. Hence the importance of having a bi-triangular graph for the execution of Algorithm \ref{Alg:GreedyTargetCycleConstruction1} is clear. 

One obvious approach to overcome this assumption violation is by inserting (artificial) edges into the network with higher travel time values. However, while such an approach can make Alg. \ref{Alg:GreedyTargetCycleConstruction1} run without halting, the resulting target-cycle $\Xi^{\#}$ will contain the edges that were artificially introduced, compromising the target-cycle performance $J_{ss}(\Xi^{\#})$.

\begin{figure}[!h]
    \centering
    \includegraphics[width=3.2in]{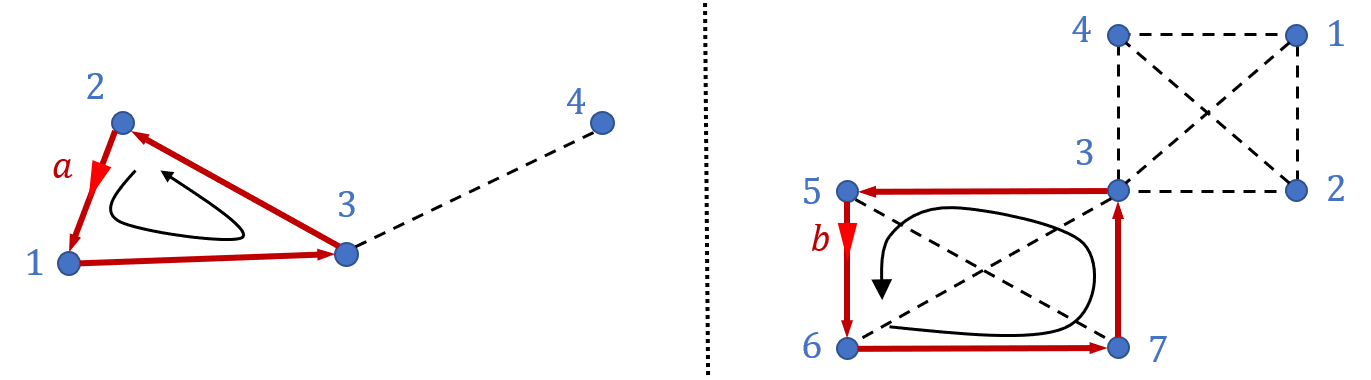}
    \caption{Two example sparse networks where Alg. \ref{Alg:GreedyTargetCycleConstruction1} has halted prematurely while executing target-cycle expansion iterations.}
    \label{Fig:CycleConstructionFailure}
\end{figure}

\subsection{The Concept of `Auxiliary Targets'}

As opposed to introducing artificial edges, we propose to introduce artificial targets (henceforth called \emph{auxiliary targets}) into the network so as to deal with this issue. Unlike artificial edges, an auxiliary target is always associated with a corresponding target in the original network. The physical interpretation of an auxiliary target is provided in the sequel.

Note that if certain targets in a network can be visited more than once, the target-cycle expansion process may not have to be halted due to the lack of edges (sparseness or non-bi-triangularity) in the network. Therefore, we propose to allow targets to be visited more than once during a tour on a target-cycle and we call such target-cycles \emph{constrained target-cycles}. For example, in both networks shown in Fig. \ref{Fig:CycleConstructionFailure}, if target $3$ is allowed to be visited more than once during a tour, we can construct the constrained target-cycles $\bar{\Xi} = \{ 2, 1,3,4,3\}$ and $\bar{\Xi} = \{ 6,7,3,2,1,4, 3,5\}$, respectively. Note that we use the notation ``\,$\bar{\cdot}$\,'' to indicate that the target-cycle is constrained (i.e., some elements are being repeated).

Giving this new flexibility (of allowing multiple visits to (some) targets during a cycle) ensures that we can form a complete target-cycle without any target being neglected - irrespective of the sparseness of the given graph. However, this new flexibility raises a new set of challenges. Specifically, we now need find how to: (\romannum{1}) analyze, (\romannum{2}) iteratively construct, (\romannum{3}) refine, and, (\romannum{3}) transform (to a set of TCP values), such constrained target-cycles. 

To analyze such constrained target-cycles (i.e., to evaluate their $J_{ss}(\cdot)$ values in \eqref{Eq:SteadyStateMeanCycleUncertainty0}), we use the previously mentioned concept of \emph{auxiliary targets}. As we will see in the sequel, replacing the repeated targets with a set of carefully chosen auxiliary targets can transform a constrained target-cycle into an \emph{equivalent} unconstrained target-cycle, enabling the application of Theorem \ref{Th:SteadyStateMeanCycleUncertainty}.

Consider a constrained target-cycle $\bar{\Xi}$ with a target $i \in\bar{\Xi}$ being visited $n$ times during a tour. Then, we first introduce an \emph{auxiliary target pool} $\mathcal{T}_{i} = \{i^{1},i^{2},\ldots,i^{n}\}$ where each auxiliary target $i^{j} \in\mathcal{T}_{i}$ can be thought of as an artificial target located in the same physical location of target $i$ (i.e., at $X_{i}$ in the mission space), but with its own parameters: an uncertainty rate $A_{i}^{j}$ and a sensing rate $B_{i}^{j}$ (to be defined). Next, we replace the repeated elements of target $i$ in $\bar{\Xi}$ with the elements taken from auxiliary target pool $\mathcal{T}_{i}$. Then, we repeat this process for all $i \in\bar{\Xi}$ with $\vert\mathcal{T}_{i} \vert> 1$. This results in an unconstrained target-cycle which we denote as $\Xi$ (i.e., without ``\,$\bar{\cdot}$\,'', we follow this notational convention in the rest of this paper).

For example, consider the constrained target-cycles proposed for the graphs in Fig. \ref{Fig:CycleConstructionFailure}. Now, using the auxiliary target pool $\mathcal{T}_{3} = \{3^{1},3^{2}\}$ (for both graphs), their respective unconstrained target-cycles $\Xi= \{2,1,3^{1},4,3^{2}\}$ and $\Xi= \{6,7,3^{1},2,1,4,3^{2},5\}$ in Fig. \ref{Fig:auxiliaryTargets1} can be obtained. Also, another example is shown in Fig. \ref{Fig:auxiliaryTargets2}.

\begin{figure}[!h]
    \centering
    \includegraphics[width=3.2in]{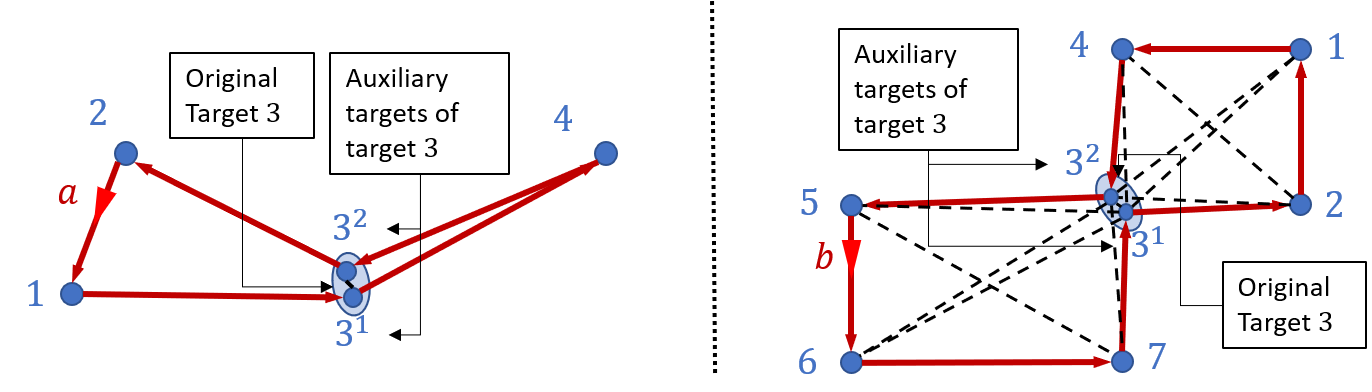}
    \caption{Example 1: Converting constrained target-cycles into unconstrained target-cycles with the use of auxiliary targets.}
    \label{Fig:auxiliaryTargets1}
\end{figure}

\begin{figure}[!h]
    \centering
    \includegraphics[width=3.2in]{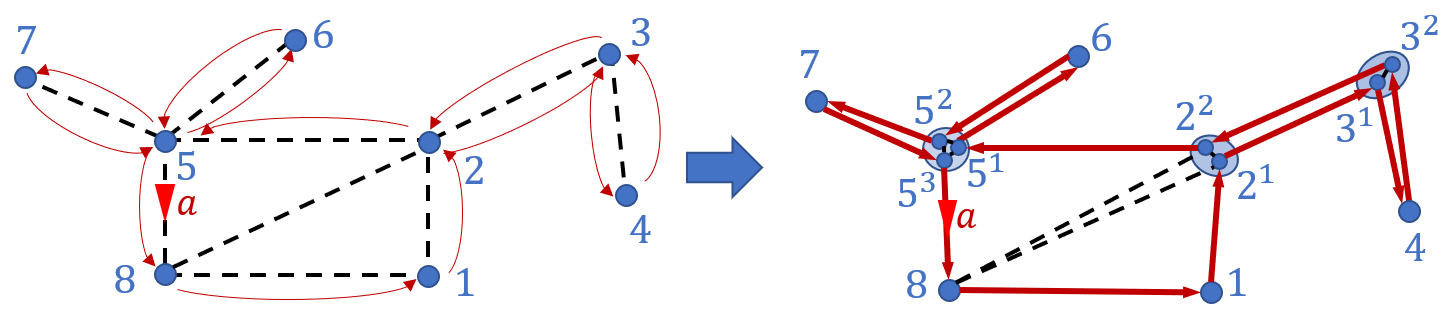}
    \caption{Example 2: Converting constrained target-cycles into unconstrained target-cycles with the use of auxiliary targets.}
    \label{Fig:auxiliaryTargets2}
\end{figure}

\paragraph*{\textbf{Equivalence Criteria}} 

For the analysis of the constrained target-cycles, we enforce the requirement that both the targets in $\Xi$ and $\bar{\Xi}$ should perform/behave in an equivalent manner at steady state. Specifically, we enforce the following \emph{equivalence criteria} between the targets in $\Xi$ and $\bar{\Xi}$.

\begin{enumerate}
\item The dwell time spent at $i^{j} \in\Xi$ is equal to the dwell time spent at $i \in\bar{\Xi}$ on its $j$\textsuperscript{th} visit during a tour.

\item The physical location of $i^{j} \in\Xi$ is the same as that of $i\in\bar{\Xi}$.

\item The contribution to the main objective $J_{T}$ \eqref{Eq:ObjectiveFunction1} by $\mathcal{T}_{i}\subset\Xi$ is equal to that of target $i \in\bar{\Xi}$, during a tour.
\end{enumerate}

The first two conditions ensure that the time required to complete a tour (for an agent) is the same for both $\Xi$ and $\bar{\Xi}$. The third condition implies $J_{ss}(\bar{\Xi}) = J_{ss}(\Xi)$. Hence, if the auxiliary target parameters are known, Theorem \ref{Th:SteadyStateMeanCycleUncertainty} will yield the value of $J_{ss}(\Xi)$. Therefore, it is clear how the concept of auxiliary targets can enable us evaluating the steady state mean cycle uncertainty value of $\Xi$.       

\paragraph*{\textbf{Sub-cycles}}

Notice that each $i^{j} \in\Xi$ can be assigned a \emph{sub-cycle} denoted by $\Xi_{i}^{j} \subset\Xi$ where $\Xi_{i}^{j}$ starts with the immediate next target to $i^{j-1} \in\Xi$ and ends with target $i^{j}$. Therefore, $\Xi$ can be written as a concatenation of sub-cycles of a target $i\in\bar{\Xi}$, i.e., 
\begin{equation}\label{Eq:SubCycleSetsAdditon}
    \Xi = \bigcup_{i^{j} \in\mathcal{T}_{i}} \Xi_{i}^{j}.
\end{equation}
For example, on the unconstrained cycle shown in \ref{Fig:auxiliaryTargets2}, the target $5$ has three auxiliary targets $\mathcal{T}_5 = \{5^1,5^2, 5^3\}$. Their respective sub-cycles would be: $\Xi_5^{1} = \Xi \backslash (\Xi_5^{2}\cup \Xi_5^{3})$, $\Xi_5^{2} = \{6, 5^2\}$, and, $\Xi_5^{3} = \{7, 5^3\}$. If a target $i\in \Xi$ does not have any auxiliary targets, then its sub-cycle (denoted as $\Xi_i^1$) would be $\Xi_i^1 = \Xi$ (i.e., the complete unconstrained target-cycle).

The \emph{sub-cycle unit vector} of $\Xi_{i}^{j}$ is denoted by $\bar{1}_{i}^{j} \in\mathbb{R}^{\vert\Xi\vert}$ and its $n$\textsuperscript{th} element is $1$ only if the $n$\textsuperscript{th} element of $\Xi$ belongs to $\Xi_{i}^{j}$. Therefore, if $\bar{1}_{\vert\Xi\vert}\in\mathbb{R}^{\vert \Xi\vert}$ is a vector of all ones, with respect to target $i\in\bar{\Xi}$, we can write
\begin{equation}
    \bar{1}_{\vert\Xi\vert} = \sum_{i^{j} \in\mathcal{T}_{i}} \bar {1}_{i}^{j}.
\end{equation}
Also, if for some $i\in\bar{\Xi}, \vert\mathcal{T}_{i}\vert=1$, then its sub-cycle unit vector is $\bar{1}_{i}^{1} = \bar{1}_{\vert\Xi\vert}$.

The \emph{sub-cycle matrix} of $\Xi$ is denoted by $\mathbf{1}_{\Xi} \in\mathbb{R}^{\vert\Xi\vert\times\vert\Xi\vert}$ and its $n$\textsuperscript{th} column is the sub-cycle unit vector of the $n$\textsuperscript{th} element of $\Xi$. Note that if $\forall i\in\bar{\Xi}, \vert\mathcal{T}_{i}\vert=1$, then all elements of $\mathbf{1}_{\Xi}$ will be 1. Figure \ref{Fig:SubCyleMatrix} shows an example sub-cycle matrix.

\begin{figure}[!h]
    \centering
    \includegraphics[width=3.2in]{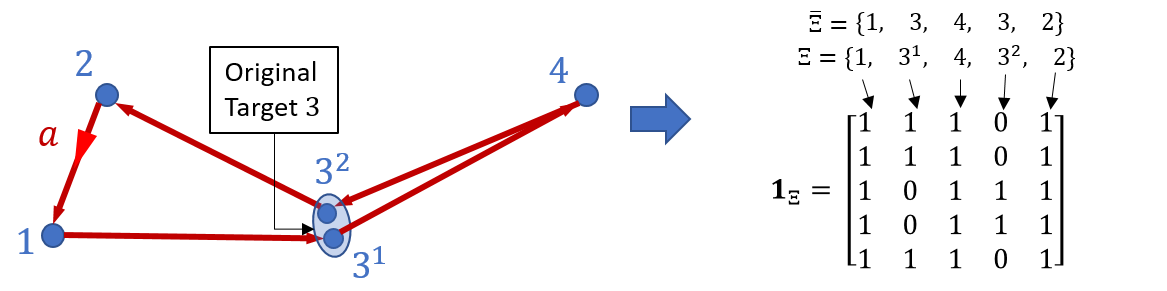}
    \caption{Sub-cycle unit vectors and sub-cycle matrix (right) for a given constrained target-cycle $\bar{\Xi}$.}
    \label{Fig:SubCyleMatrix}
\end{figure}

\subsection{Analysis of Constrained Target-Cycles}

We are now ready to analyze a generic constrained target-cycle $\bar{\Xi}$. Throughout this analysis, we will use the constrained target-cycle example shown in Fig. \ref{Fig:CycleGeometry2} for illustration purposes. Note that in this particular constrained target-cycle, $\bar{\Xi} = \{1,2,\ldots ,n,\ldots,{n+m-1},n\}$ and target $n \in\bar{\Xi}$ is visited twice during a tour. Introducing auxiliary targets $\mathcal{T}_{n} = \{n^{1}, n^{2}\}$, $\bar{\Xi}$ can be converted to its equivalent unconstrained version $\Xi$. The sub-cycles of $n^{1}$ and $n^{2}$ in $\Xi$ are $\Xi_{n}^{1} = \{1,2,\ldots,{n-1},n^{1}\}$ and $\Xi_{n}^{2} = \{{n+1},{n+2},\ldots,{n+m-1},n^{2}\}$, respectively.  

\begin{figure}[!h]
    \centering
    \includegraphics[width=3.2in]{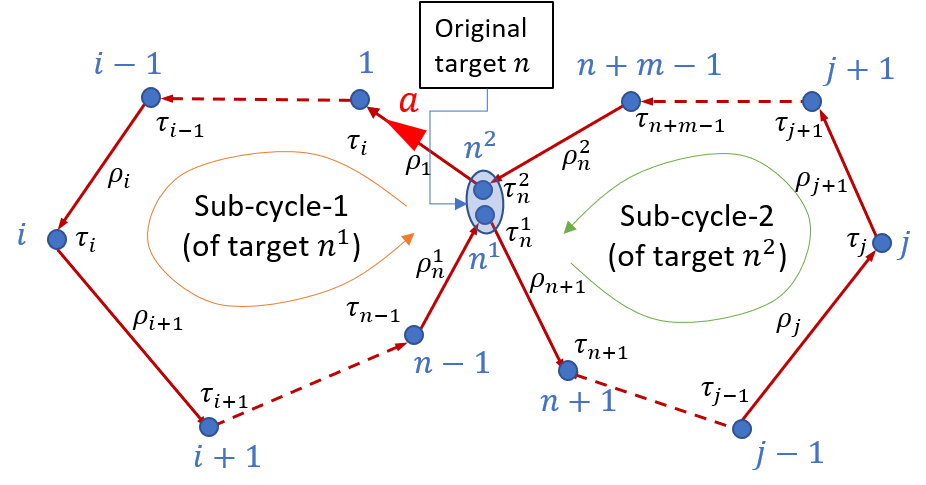}
    \caption{A general constrained target-cycle with target $n$ being visited twice during the cycle.}
    \label{Fig:CycleGeometry2}
\end{figure}

We are now ready to analyze a generic constrained target-cycle $\bar{\Xi}$. Throughout this analysis, we will use the constrained target-cycle example shown in Fig. \ref{Fig:CycleGeometry2} for illustration purposes. Note that in this particular constrained target-cycle, $\bar{\Xi} = \{1,2,\ldots ,n,\ldots,{n+m-1},n\}$ and target $n \in\bar{\Xi}$ is visited twice during a tour. Introducing auxiliary targets $\mathcal{T}_{n} = \{n^{1}, n^{2}\}$, $\bar{\Xi}$ can be converted to its equivalent unconstrained version $\Xi$. The sub-cycles of $n^{1}$ and $n^{2}$ in $\Xi$ are $\Xi_{n}^{1} = \{1,2,\ldots,{n-1},n^{1}\}$ and $\Xi_{n}^{2} = \{{n+1},{n+2},\ldots,{n+m-1},n^{2}\}$, respectively. A \emph{tour} on $\bar{\Xi}$ starts/ends when the agent leaves target $n$ to reach target $1$ and we assume the agent behavior on $\bar{\Xi}$ to follow Assumption \ref{As:AgentBehavior1}. We label the inter-target travel times on $\bar{\Xi}$ same as before (see Figs. \ref{Fig:cycleGeometry1} and \ref{Fig:CycleGeometry2}) and define $\bar{\rho}_{\Xi} = [\rho_{1},\rho_{2},\ldots,\rho^{1}_{n},\ldots,\rho_{n+m-1},\rho^{2}_{n}]^{T}$ as the \emph{travel time vector} of the target-cycle $\bar{\Xi}$. To simplify the analysis, we skip the transient analysis of the constrained target-cycle $\bar{\Xi}$ and directly make the following assumption (see also Remark \ref{Rm:Dynamics2}).

\begin{assumption}
\label{As:StabilityOfConstrainedCycles}
The dwell time dynamics of the constrained target-cycle $\bar{\Xi}$ have a feasible and globally asymptotically stable equilibrium point.
\end{assumption} 

When the dwell time dynamics have converged to the equilibrium point as per the Assumption \ref{As:StabilityOfConstrainedCycles}, the overall system is said to be operating in its steady state. Figure \ref{Fig:CycleGraphs2} shows the steady state behavior of the target uncertainties during a tour on the target-cycle $\bar{\Xi}$. The notation $\bar{\tau}_{\Xi} = [\tau_{1},\tau_{2},\ldots,\tau^{1}_{n},\ldots,\tau_{n+m-1},\tau^{2}_{n}]^{T}$ is used to represent the steady state dwell times of targets in $\Xi$. The following lemma generalizes Lemma \ref{Lm:SteadyStateDwellTimes1} to evaluate $\bar{\tau}_{\Xi}$ for any target-cycle $\bar{\Xi}$.

\begin{figure}[!h]
    \centering
    \includegraphics[width=3.2in]{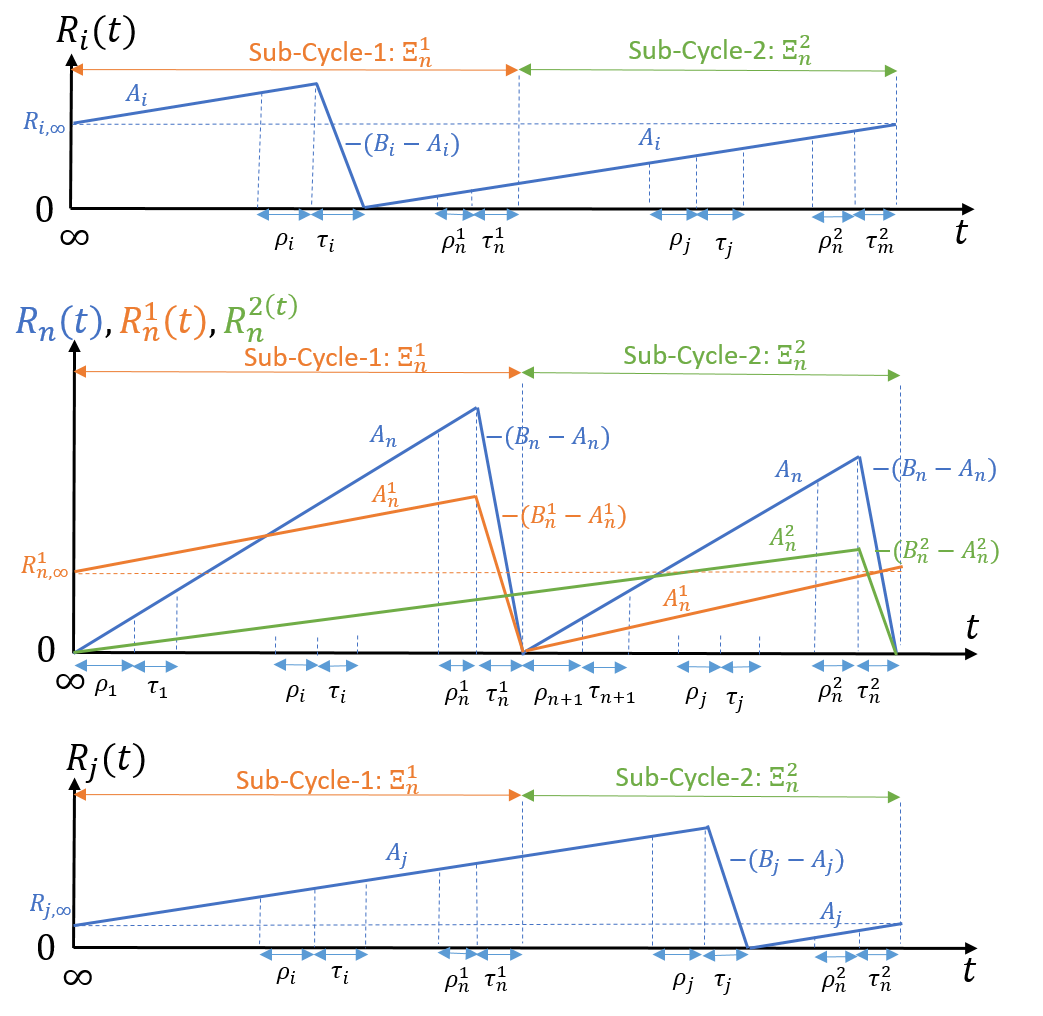}
    \caption{Variation of target uncertainties of the constrained target-cycle shown in Fig. \ref{Fig:CycleGeometry2} - after achieving the steady state.}
    \label{Fig:CycleGraphs2}
\end{figure}

\begin{lemma}\label{Lm:SteadyStateDwellTimes2}
Under Assumptions \ref{As:AgentBehavior1} and \ref{As:StabilityOfConstrainedCycles}, when a single agent traverses a generic constrained target-cycle $\bar{\Xi}$ (with $\Xi$ being the equivalent unconstrained version of $\bar{\Xi}$), the steady state dwell times $\bar{\tau}_{\Xi}$ are given by
\begin{equation}
\label{Eq:SteadyStateDwellTimes2}
\bar{\tau}_{\Xi} = [diag(\bar{\gamma}_{\Xi})
- \mathbf{1}_{\Xi}]^{-1}\mathbf{1}_{\Xi} \bar{\rho}_{\Xi},
\end{equation}
where $\bar{\gamma}_{\Xi}\in\mathbb{R}^{\vert\Xi\vert}$ is such that if the $i$\textsuperscript{th} target of $\bar{\Xi}$ is $j$, then, the $i$\textsuperscript{th} element of $\bar{\gamma}_{\Xi}$ is $\frac{B_{j}}{A_{j}}$, and $\mathbf{1}_{\Xi}$ is the sub-cycle matrix and $\bar{\rho}_{\Xi}$ is the travel time vector of the target-cycle $\Xi$. 
\end{lemma}

\begin{proof}
By inspection of the $R_{n}(t)$ profile in Fig. \ref{Fig:CycleGraphs2}, for each target $n\in\bar{\Xi}$ and for each auxiliary target $n^{j} \in\mathcal{T}_{n}$, considering its corresponding sub-cycle $\Xi_{n}^{j}$'s time period, we can write: $(B_{n}-A_{n}) \tau_{n}^{j} = A_{n}(T_{n}^{j}-\tau_{n}^{j})$ which implies
\begin{equation*}
     B_{n} \tau_{n}^{j} = A_{n} T_{n}^{j} \ \ \forall j \ni n^j \in \mathcal{T}_n,
\end{equation*} 
where $T_{n}^{j}$ is the total time taken to complete the sub-cycle $\Xi_{n}^{j}$. Now, using the sub-cycle unit vectors, we can substitute for $T_{n}^{j}$ to get: \begin{equation*}
    B_{n} \tau_{n}^{j} = A_{n} (\bar{1}_{n}^{j})^{T}(\bar{\rho}_{\Xi} + \bar{\tau}_{\Xi}).
\end{equation*}
This relationship gives $\vert\Xi\vert$ equations which we need to solve for $\bar{\tau}_{\Xi} \in\mathbb{R}^{\vert\Xi\vert}$. Arranging all the equations in a matrix form:
$$
    diag(\bar{\gamma}_{\Xi}) \bar{\tau}_{\Xi} = \mathbf{1}_{\Xi}(\bar{\rho}_{\Xi} + \bar{\tau}_{\Xi}) 
$$
gives the result in \eqref{Eq:SteadyStateDwellTimes2}.
\end{proof}

It is important to note that Lemma \ref{Lm:SteadyStateDwellTimes2} is applicable to any general target-cycle $\Xi$. Hence when the interested target-cycle $\Xi$ is an unconstrained one, it is easy to show that both steady state dwell time values given by Lemma \ref{Lm:SteadyStateDwellTimes2} and Lemma \ref{Lm:SteadyStateDwellTimes1} are identical.   

\begin{remark}
\label{Rm:Dynamics2} 
Note that \eqref{Eq:SteadyStateDwellTimes2} is only valid under Assumption \ref{As:StabilityOfConstrainedCycles}, i.e., if the dwell times observed in the $k$\textsuperscript{th} tour on $\bar{\Xi}$ (say $\bar{\tau}_{\Xi,k}$) converge to an equilibrium point ($\bar{\tau}_{\Xi}$) as $k\rightarrow\infty$. However, based on the form of \eqref{Eq:SteadyStateDwellTimes2}, we can conclude that the conditions for the existence and feasibility of such an equilibrium point are $\vert diag(\bar{\gamma}_{\Xi}) - \mathbf{1}_{\Xi} \vert\neq0$ and $[diag(\bar{\gamma}_{\Xi}) - \mathbf{1}_{\Xi}]^{-1}\mathbf{1}_{\Xi} \bar{\rho}_{\Xi} > 0$, respectively.
\end{remark}

\begin{remark}
Even though we cannot apply Lemma \ref{Lm:ShermonMorrison} to simplify the inverse: $[diag(\bar{\gamma}_{\Xi}) - \mathbf{1}_{\Xi}]^{-1}$, using the fact that: $rank(\mathbf{1}_{\Xi}) = ($number of auxiliary targets in $\Xi)\ll\vert\Xi\vert$, we can apply the Ken-Miller theorem \cite{Miller1981} to efficiently compute this inverse. 
\end{remark}

Using Lemma \ref{Lm:SteadyStateDwellTimes2}, we can find the dwell time vector $\bar{\tau}_{\Xi}$. Further, we already know the travel times vector $\bar{\rho}_{\Xi}$. Therefore, we now can find the \emph{total sub-cycle times} denoted by $T^{j}_n$ for all targets $n^j \in \Xi$, as,
\begin{equation}\label{Eq:TotalSubCycleTime}
    T_{n}^{j} = (\bar{1}_{n}^{j})^{T}(\bar{\rho}_{\Xi}+\bar{\tau}_{\Xi}).
\end{equation}
Furthermore, when the total sub-cycle time metric evaluated for an actual target (which does not have any auxiliary targets), we get the \emph{total cycle time} value denoted by $T_{\Xi}$ where,
\begin{equation}\label{Eq:TotalCycleTime}
    T_{\Xi} = \bar{1}_{\vert\Xi\vert}^{T}(\bar{\rho}_{\Xi}+\bar{\tau}_{\Xi}). 
\end{equation}

\begin{lemma}\label{Lm:AuxiliaryTargetParameters}
Under the same conditions stated in Lemma \ref{Lm:SteadyStateDwellTimes2}, the auxiliary target parameters of any $n^{j} \in\Xi$ (i.e., $A_{n}^{j}$ and $B_{n}^{j}$) are:
\begin{equation}
\label{Eq:AuxiliaryTargetParameters}A_{n}^{j} = \frac{T_{n}^{j}}{T_{\Xi}}%
\frac{\tau_{n}^{j}(B_{n}-A_{n})}{(T_{\Xi}-\tau_{n}^{j})} \mbox{\ \, and \ }
B_{n}^{j} = \frac{T_{n}^{j}(B_{n}-A_{n})}{(T_{\Xi}-\tau_{n}^{j})}.
\end{equation}
\end{lemma}

\begin{proof}
Observing the auxiliary target uncertainty profiles $R_{n}^{1}(t)$ and $R_{n}^{2}(t)$ (of auxiliary targets $n^{1}$ and $n^{2}$, respectively) illustrated in Fig. \ref{Fig:CycleGraphs2} for the target-cycle shown in Fig. \ref{Fig:CycleGeometry2}, note that the shape of these profiles should satisfy the equivalence criteria that we previously established.

Using these graphs, for the complete cycle duration $T_{\Xi}$, we can write $B_n^j \tau_n^j = A_n^j T_n^j$ for $j=1,2$. This result can be generalized to any generic target-cycle $\bar{\Xi}$ as, 
\begin{equation}
\label{Eq:AuxiliaryTargetParametersProofStep1}
    B_{n}^{j} \tau_{n}^{j} =
A_{n}^{j} T_{n}^{j}, \ \ \forall n^{j}\in\mathcal{T}_{n},\ \forall n\in
\bar{\Xi}.
\end{equation}
The above relationship ensures the first condition in the equivalence criteria. 

Now, going back to the case of target-cycle shown in Fig.  \eqref{Fig:CycleGeometry2}, to ensure the third condition in the equivalence criteria, we need: 
\begin{equation}
\begin{aligned}
    \frac{1}{T}\int_{T_{\Xi}}R_n^1(t)dt + \frac{1}{T}\int_{T_{\Xi}}R_n^2(t)dt = 
    \frac{1}{T}\int_{T_{\Xi}}R_n(t)dt\\
    =  
    \frac{1}{T}\int_{T^1_n}R_n(t)dt + \frac{1}{T}\int_{T^2_n}R_n(t)dt.
\end{aligned}
\end{equation}
Here, we can get two equations by equating individual terms in L.H.S. and R.H.S. This result can be generalized to any generic target-cycle $\bar{\Xi}$ as, 
\begin{equation}
\int_{T_{\Xi}} R_{n}^{j}(t)dt = \int_{T_{n}^{j}} R_{n}(t)dt, \ \forall n^{j}\in\mathcal{T}_{n},\ \forall n\in\bar{\Xi}.
\end{equation}
As uncertainty profiles are piece-wise linear, we can evaluate these integrals and simplify this system of equations as:
\begin{equation}\label{Eq:AuxiliaryTargetParametersProofStep2}
    T_{\Xi}(B_{n}^{j} - A_{n}^{j}) =
T_{n}^{j}(B_{n}-A_{n}),\ \ \forall n^{j}\in\mathcal{T}_{n},\ \forall n\in
\bar{\Xi}.
\end{equation}
Finally, we can solve \eqref{Eq:AuxiliaryTargetParametersProofStep1} and \eqref{Eq:AuxiliaryTargetParametersProofStep2} to obtain the auxiliary target parameters: $\{(A_{n}^{j},B_{n}^{j}): \forall n^{j} \in\mathcal{T}_{n}, \ \forall n \in\bar{\Xi}\}$ as in \eqref{Eq:AuxiliaryTargetParameters}.
\end{proof}

Using Lemma \ref{Lm:AuxiliaryTargetParameters}, we can find all the unknown target uncertainty and sensing rate parameters of the targets listed in $\Xi^u$. Now, Let us lump those parameters in vectors $\bar{A}_{\Xi}$ and $\bar{B}_{\Xi}$ respectively. For example, for the target-cycle shown in Fig. \ref{Fig:CycleGeometry2}, $\bar{A}_{\Xi} = [A_1,A_2,\ldots,A^1_n,\ldots,A_{n+m-1},A^2_n]^T$. With this notation, we can propose our main theorem as follows.

\begin{theorem}\label{Th:SteadyStateMeanCycleUncertainty2}
Under Assumptions \ref{As:AgentBehavior1} and \ref{As:StabilityOfConstrainedCycles}, when a single agent traverses a generic constrained target-cycle $\bar{\Xi}$ (with $\Xi$ being the equivalent unconstrained version of $\bar{\Xi}$), the steady state mean cycle uncertainty $J_{ss}(\bar{\Xi})$ (defined in \eqref{Eq:SteadyStateMeanCycleUncertainty0}) is 
\begin{equation}
\label{Eq:SteadyStateMeanCycleUncertainty2}
J_{ss}(\bar{\Xi}) = \frac{1}{2}(\bar{B}_{\Xi} - \bar{A}_{\Xi})^{T}\bar{\tau}_{\Xi}, 
\end{equation}
where $\bar{\tau}_{\Xi}$ is given by Lemma \ref{Lm:SteadyStateDwellTimes2} and auxiliary target parameters included in the vectors $\bar{A}_{\Xi}$ and $\bar{B}_{\Xi}$ are given by Lemma \ref{Lm:AuxiliaryTargetParameters}.
\end{theorem}
\begin{proof}
Note that $\Xi$ is an unconstrained target-cycle. Therefore, we can directly use Theorem \ref{Th:SteadyStateMeanCycleUncertainty} to write 
\begin{equation}
    J_{ss}(\Xi) = \frac{1}{2}(\bar{B}_{\Xi} - \bar{A}_{\Xi})^{T}\bar{\tau}_{\Xi},
\end{equation}
where $\bar{\tau}_{\Xi}$ is given by Lemma \ref{Lm:SteadyStateDwellTimes2} and unknown parameters in $\bar{A}_{\Xi}$ and $\bar{B}_{\Xi}$ are given by Lemma \ref{Lm:AuxiliaryTargetParameters} (or by simply using \eqref{Eq:AuxiliaryTargetParametersProofStep2}). Finally, according to the equivalence criterion 3: $J_{ss}(\bar{\Xi}) = J_{ss}(\Xi)$.  
\end{proof}

\subsection{Greedy Target-Cycle Construction}
Let $\mathcal{D}$ denote the set of all possible target-cycles on $\mathcal{G}$. Compared to $\mathcal{C}$ in \eqref{Eq:OptimumUnconstrainedCycleSearch}, $\mathcal{D} \supseteq\mathcal{C}$ as $\mathcal{D}$ now also includes all the constrained target-cycles. Clearly, $\vert\mathcal{D}\vert= \infty$ and thus, exhaustive search methods (exploiting Theorem \ref{Th:SteadyStateMeanCycleUncertainty2}) cannot be used to determine the best performing target-cycle in $\mathcal{D}$:
\begin{equation}
\label{Eq:OptimumCycleSearch}\bar{\Xi}^{*} = \arg\min_{\bar{\Xi}
\in\mathcal{D}}J_{ss}(\bar{\Xi}).
\end{equation}
Hence, we seek to efficiently construct a sub-optimal target-cycle $\bar{\Xi}^{\#}\in\mathcal{D}$ as a candidate for $\bar{\Xi}^{*}$ in \eqref{Eq:OptimumCycleSearch} using a greedy iterative target-cycle expansion process identical to Alg. \ref{Alg:GreedyTargetCycleConstruction1}.

\begin{figure*}[!ht]
    \centering
    \includegraphics[width=5in]{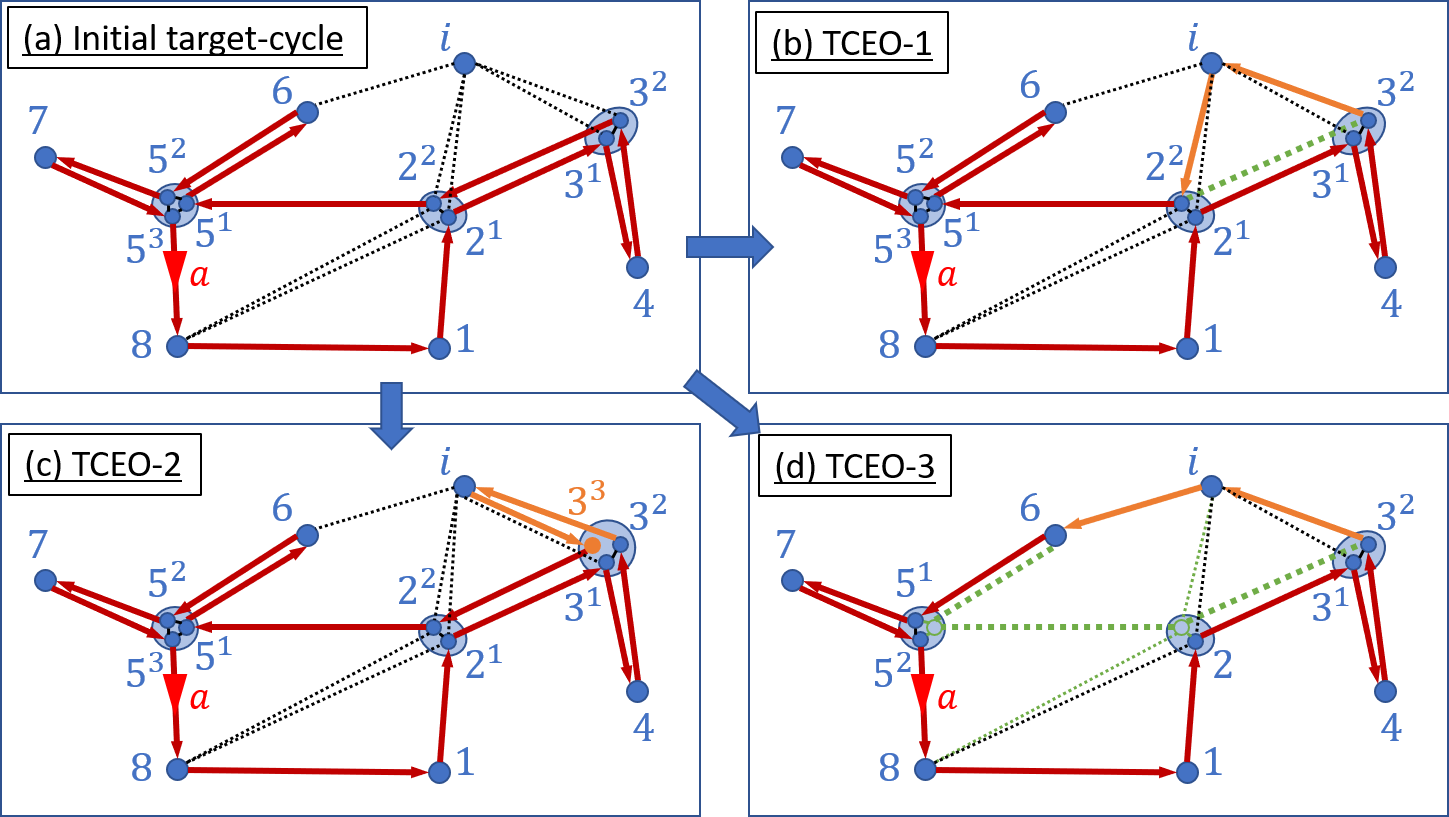}
    \caption{Target-cycle expansion operations.}
    \label{Fig:AdvancedCycleExpanding}
\end{figure*}

\paragraph*{\textbf{Target-Cycle Expansion Operation - Type 1 (TCEO-1)}}
Note that the greedy scheme given in Algorithm \ref{Alg:GreedyTargetCycleConstruction1} iteratively uses a target-cycle expansion operation (TCEO) (shown in Fig. \ref{Fig:basicCycleExpanding}, and also in Fig. \ref{Fig:AdvancedCycleExpanding} (a) $ \rightarrow $ (b)) to construct a sub-optimal unconstrained target-cycle for \eqref{Eq:OptimumUnconstrainedCycleSearch}. Let us label this type of a TCEO as a TCEO-1. The gain in the objective function due to a TCEO-1 is given in Theorem \ref{Th:MarginalGain1} as a function $\Delta J_T(i \vert \xi, (j,l))$ (from now onward, we denote this by $\Delta J_T^1(i \vert \xi, (j,l))$). Also, when started with an unconstrained target-cycle, the TCEO-1 will always result an unconstrained target-cycle. However, in \eqref{Eq:OptimumCycleSearch}, the search space $\mathcal{D}$ contains both unconstrained and constrained target-cycles. Therefore, new TCEOs should be introduced (apart from the TCEO-1).

\paragraph*{\textbf{Target-Cycle Expansion Operation - Type 2 (TCEO-2)}}
Consider a generic target-cycle $\bar{\Xi}$ (with its list of edges being $\bar{\xi}$ and their respective converted versions being $\Xi$ and $\xi$). In TCEO-2, connecting an external target $i \in \mathcal{V} \backslash \bar{\Xi}$ to the $\bar{\Xi}$ is done via creating an additional auxiliary target to one of the targets $j \in \bar{\Xi}$ when $(j,i) \in \mathcal{E}$. Specifically, if this expansion happens at target $j^k \in \Xi$, then, to get the expanded cycle $\bar{\Xi}^{\prime}$ (and $\bar{\xi}^{\prime}$),  
(\romannum{1}) all the auxiliary targets $j^l \in \Xi, l > k$ should be relabelled to $j^{l+1}$ in both $\Xi$ and $\xi$,
(\romannum{2}) two new targets $i$ and $j^{k+1}$ should be inserted after the target $j^k$ in $\Xi$,
(\romannum{3}) two new edges $(j^k,i)$ and $(i,j^{k+1})$ should be inserted after the edge $(\cdot,j^k)$ in $\xi$.

The result of a TCEO-2 $\bar{\Xi}^{\prime}$) will always be a constrained target-cycle. A marginal gain function of the form $\Delta J_T^2(i \vert \Xi, j^k)$ can be proposed to evaluate the gain of a such TCEO-2 step as, (similar to Theorem \ref{Th:MarginalGain1}) 
\begin{equation}\label{Eq:MarginalGain2}
    \Delta J_T^2(i \vert \Xi, j^k) = 
    \left( R_{i,0} + \frac{A_i T}{2} \right) + J_{ss}(\bar{\Xi}) - J_{ss}(\bar{\Xi}').
\end{equation}
An example of this TCEO-2 is shown in Fig. \ref{Fig:AdvancedCycleExpanding} (a) $ \rightarrow $ (c) where an additional auxiliary target ($3^3$) has been created to expand the current target-cycle so that it includes the external target $i$. 

Take $\Vert \bar{\Xi} \Vert$ as the number of distinguishable entries in the set $\bar{\Xi}$. It is important to note that when the graph $\mathcal{G}$ is \emph{connected}, and the target-cycle $\bar{\Xi}$ is such that $\Vert \bar{\Xi} \Vert < \vert \mathcal{V} \vert$, we can always increase $\Vert \bar{\Xi} \Vert$ using TCEO-2 on $\bar{\Xi}$. Also note that when $\Vert \bar{\Xi} \Vert = \vert \mathcal{V} \vert$, the target-cycle goes through all the targets in the given graph. Therefore, this property of TCEO-2 is useful in constructing an improved greedy scheme which can overcome the situations like the ones shown in Fig. \ref{Fig:CycleConstructionFailure} (where Algorithm \ref{Alg:GreedyTargetCycleConstruction1} halts prematurely).

\paragraph*{\textbf{Target-Cycle Expansion Operation - Type 3 (TCEO-3)}}
Let us consider a situation where we have a generic target-cycle $\bar{\Xi}$. We use the notation $[j^k, l^m]$ to represent the ordered set of targets in $\Xi$ that are exclusively in-between the two non-adjacent targets $j^k \in \Xi$ and $l^m \in \Xi$. If: 
(\romannum{1}) an external target $i$ has feasible edges to targets $j^k, l^m \in \Xi$ (i.e., $(j^k,i), (i,l^m) \in \mathcal{E}$) and,  
(\romannum{2}) all the targets in $[j^k,l^m]$ have other auxiliary targets (i.e., $[j^k,l^m] \subseteq Aux \triangleq \Xi\backslash \bar{\Xi}$), then, we can perform the TCEO-3. Specifically, in TCEO-3, following operations are carried out:
(\romannum{1}) The set of targets $[j^k,l^m]$ in $\Xi$ are replaced with $i$,
(\romannum{2}) The respective edges in $\xi$ are replaced with two edges $(j^k,i),(i, l^m)$, and, 
(\romannum{3}) All the auxiliary targets of the targets in $[j^k,l^m]$ are relabelled.

Note that, to perform TCEO-3, $\bar{\Xi}$ should be a constrained target-cycle. However, the resulting expanded target-cycle (say $\bar{\Xi}'$) is not guaranteed to be a constrained target-cycle, as TCEO-3 always cancels out a set of auxiliary targets (specifically the set $[j^k,l^m]$). Similar to \eqref{Eq:MarginalGain1} and \eqref{Eq:MarginalGain2}, following the Theorem \ref{Th:MarginalGain1}, a marginal gain function $\Delta J_T^3(i \vert \Xi, [j^k,l^m])$ where, 
\begin{equation}\label{Eq:MarginalGain3}
    \Delta J_T^3(i \vert \Xi, [j^k,l^m]) = 
    \left( R_{i,0} + \frac{A_i T}{2} \right) + J_{ss}(\bar{\Xi}) - J_{ss}(\bar{\Xi}'),
\end{equation}
can be used to represent the gain in the objective function due to the aforementioned generic TCEO-3. 

An example of this TCEO-3 is shown in Fig. \ref{Fig:AdvancedCycleExpanding} (a) $ \rightarrow $ (d), where the set of auxiliary targets $[3^2,6] = \{2^2,5^1\}$ are cancelled out with the insertion of the external target $i$ in the expanded target-cycle.

\paragraph*{\textbf{Improved Greedy algorithm}}
Now, we propose the improved greedy scheme which can construct a sub-optimal target-cycle for the problem in \eqref{Eq:OptimumCycleSearch} given the target topology. Unlike in Algorithm \ref{Alg:GreedyTargetCycleConstruction1}, in this improved greedy scheme (given below in Algorithm \ref{Alg:GreedyTargetCycleConstruction2}), the search space of each iteration is not limited to unconstrained target-cycles. Therefore, this improved greedy scheme will work even under sparse graph conditions.   

\begin{algorithm*}[!h]
\caption{The proposing improved greedy algorithm for target-cycle construction under assumptions \ref{As:AgentBehavior1} and \ref{As:EstimationError} with a single agent.}
\begin{algorithmic}[1]
\State \textbf{Input}: Target topology $\mathcal{G}=(\mathcal{V},\mathcal{E})$, where, $\mathcal{V} = \{1,2,\cdots,M\}$ and $\mathcal{E} \subseteq \{(i,j): i,j \in \mathcal{V}\}$
\State \textbf{Output}: A sub-optimal target-cycle $\bar{\Xi}^\#$ (and $\bar{\xi}^\#$).
\State $(j^*,l^*) := \arg\min_{(j,l) \in \mathcal{E}}J_{ss}(\{j,l\})$ 
\Comment{Identifying the best 2-Target-cycle to cover}
\State $\bar{\Xi} := \{j^*,l^*\}$,  
$\xi := \{(j^*,l^*),(l^*,j^*)\}$ \Comment{Initial target-cycle} 
\Do 
    \State $[val_1,\ (i^*,(j^*,l^*))_1] := 
    \underset{\substack{(i,(j,l)): i \in \mathcal{V} \backslash \bar{\Xi},\\ (j,l) \in \xi}}{\mathrm{argmax}} 
    \left[\Delta J_T^1(i \vert \xi,(j,l))\right]$
    \Comment{Identifying the best available TCEO-1}
    \State $[val_2,\ (i^*,j^{k*})_2] := 
    \underset{\substack{(i,j^k):i \in \mathcal{V} \backslash \bar{\Xi},\\ j^k \in \Xi}}{\mathrm{argmax}}
    \left[\Delta J_T^2(i \vert \Xi,j^k)\right]$ 
    \Comment{Identifying the best available TCEO-2}
    \State $[val_3,\ (i^*,[j^{k*},l^{m*}])_3] := 
    \underset{\substack{(i,j^k,l^m): i \in \mathcal{V} \backslash \bar{\Xi}, \\
    j^k, l^m \in \Xi,\\ [j^k, l^m] \subseteq Aux} }{\mathrm{argmax}}
    \left[\Delta J_T^3(i \vert \Xi,[j^k, l^m])\right]$
    \Comment{Identifying the best available TCEO-3}
    \If{$(val_1 \geq val_2)$ and $(val_1 \geq val_3)$ and $(val_1 > 0)$ } 
    \Comment{The found TCEO-1 is the most profitable!}
        \State \textbf{Execute} TCEO-1 on $\bar{\Xi}$ using $(i^*,(j^*,l^*))_1$ 
    \ElsIf{$(val_2 \geq val_3)$ and $(val_2 > 0)$}
    \Comment{The found TCEO-2 is the most profitable!}
        \State \textbf{Execute} TCEO-2 on $\bar{\Xi}$ using $(i^*,j^{k*})_2$
    \ElsIf{$(val_3 > 0)$}
    \Comment{The found TCEO-3 is the most profitable!}
        \State \textbf{Execute} TCEO-3 on $\bar{\Xi}$ using $(i^*,[j^{k*},l^{m*}])_3$
    \Else
        \State \textbf{Break} \Comment{No TCEO was found with a positive marginal gain.}
    \EndIf
\doWhile{\textbf{True}}
\State $\bar{\Xi}^\# := \bar{\Xi};\ \bar{\xi}^\#:= \bar{\xi}$; \textbf{Return}; 
\end{algorithmic}\label{Alg:GreedyTargetCycleConstruction2}
\end{algorithm*}

It starts from searching for the best target-cycle of length 2 (i.e. $\bar{\Xi} \ni \vert \bar{\Xi} \vert = 2$) and the obtained solution is then used as the initial target-cycle. Next, the current target-cycle is expanded by adding an external target in each iteration. The steps 6,7 and 8 searches for the best method of expanding the current target-cycle under TCEO-1, TCEO-2, and TCEO-3, respectively. In step 8, the set variable $Aux = \Xi \backslash \bar{\Xi}$ represents the set of all introduced auxiliary targets. Then, steps 9-15 aims to execute the TCEO with the highest marginal gain.

\paragraph*{\textbf{Refining the Obtained Generic Target-Cycle}}
If the sub-optimal target-cycle given by the greedy Algorithm \ref{Alg:GreedyTargetCycleConstruction2} (say $\bar{\Xi}^\#$) is an unconstrained one, we can apply the previously discussed TSP inspired 2-Opt and 3-Opt techniques (see Fig. \ref{Fig:2Opt3Opt}) to further improve (refine) the obtained solution $\bar{\Xi}^\#$. However, when $\bar{\Xi}^\#$ is a constrained target-cycle, applicability of such 2-Opt or 3-Opt techniques is not straight forward - because some targets are being visited more than once during the target-cycle $\bar{\Xi}^\#$ (contrary to the TSP problem scenarios). 

To overcome this challenge, we execute the following set of steps:
\begin{enumerate}
    \item Execute a 2-Opt (or a 3-Opt) operation on the converted version of $\bar{\Xi}^\#$ (i.e., on $\Xi^{\#}$) assuming all the auxiliary targets in a set $\mathcal{T}_i$ are fully connected (with zero distance links) for all targets $i \in \bar{\Xi}^\#$. Take the resulting target-cycle as $\bar{\Xi}^1$.
    \item If $\bar{\Xi}^1$ utilizes one of the aforementioned zero distance links (say the link between $(i^j,i^k)$), it means we can now merge the two auxiliary targets $(i^j,i^k)$ into a single auxiliary target $i^l$ (where $l=min(j,k)$). Take the resulting target-cycle as $\bar{\Xi}^2$.
    \item (Inspired by the TCEO-3) If there exists two distinct non-adjacent  targets $j^k, l^m \in \Xi^{2}$ such that 
    (\romannum{1}) $(j^k, l^m) \in \mathcal{E}$, and (\romannum{2}) $[j^k, l^m] \subseteq Aux = \Xi^{2} \backslash \bar{\Xi}^2$, then,
    we can 
    (\romannum{1}) remove all the auxiliary targets $[j^k, l^m]$ from $\Xi^{2}$, and,  
    (\romannum{2}) replace the respective edges in $\xi^{2}$ with just the edge $(j^k,l^m)$. Take resulting target-cycle as $\bar{\Xi}^3$.
    \item Now, if $J_{ss}(\bar{\Xi}^3) \leq J_{ss}(\bar{\Xi}^\#)$ update $\bar{\Xi}^\# := \bar{\Xi}^3$ and continue to step (1) above.
\end{enumerate}

As an example, see the process shown in Fig. \ref{Fig:RefinementsOnConstrainedTargetCycles} (a) $\rightarrow$ (b) $\rightarrow$ (c) $\rightarrow$ (d). It is a case where a 2-Opt operation has been carried out with a consequent auxiliary target merging step (i.e., refinement step (1) and (2) given above) - to generate a new perturbed target-cycle. Also, the direct step Fig. \ref{Fig:RefinementsOnConstrainedTargetCycles} (a) $\rightarrow$ (d) can be seen as an example for the refinement step (3) given above - where the auxiliary target set $[4,2] = \{3^2\}$ is being removed due to the existence of the direct link $(4,2) \in \mathcal{E}$. 

\begin{figure}[!h]
    \centering
    \includegraphics[width=3.2in]{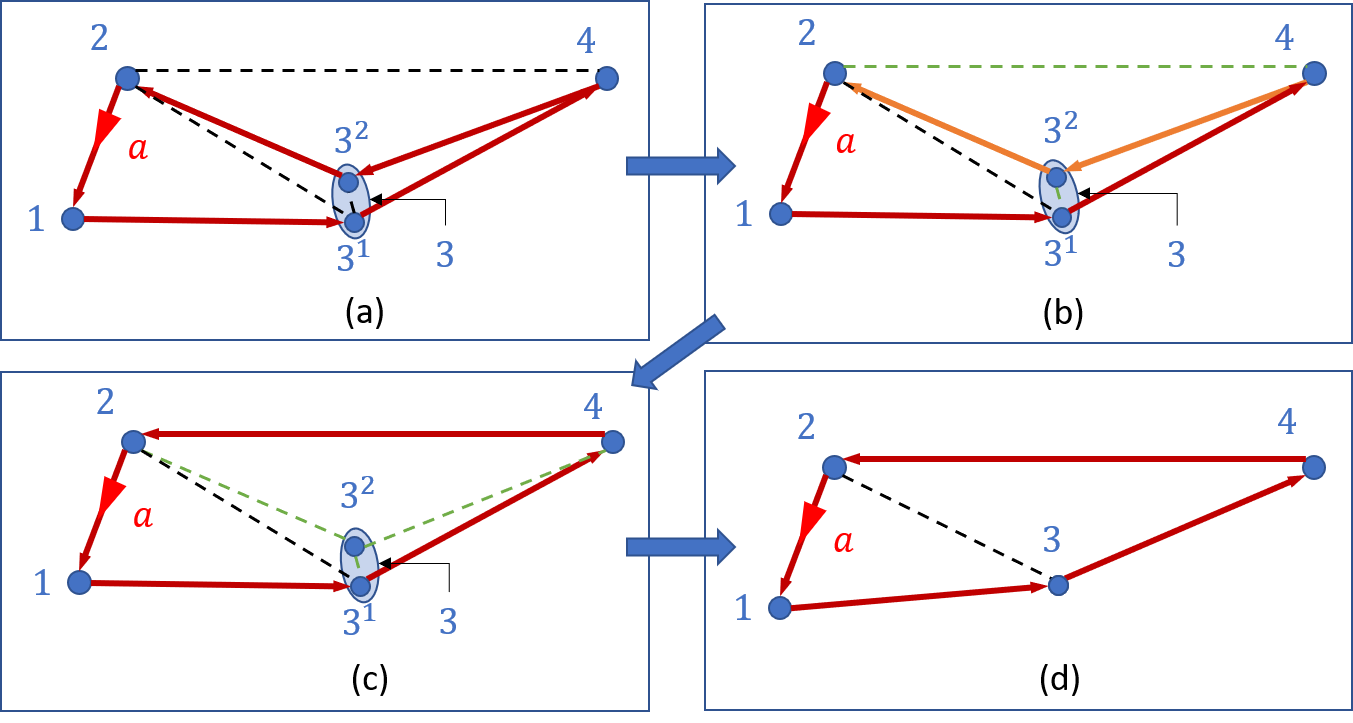}
    \caption{Refinement steps on a constrained target-cycle.}
    \label{Fig:RefinementsOnConstrainedTargetCycles}
\end{figure}

Once each of the possible 2-Opt and 3-Opt operations for the current target-cycle $\bar{\Xi}^\#$ have been evaluated according to the above set of steps, and when none of those resulting target-cycles (i.e., $\bar{\Xi}^3$ cycles) has a lower value than $J_{ss}(\bar{\Xi}^\#)$, then, we call the target-cycle $\bar{\Xi}^\#$ as the refined sub-optimal target-cycle and we represent it using the symbol $\bar{\Xi}^R$.

\subsection{The Initial Threshold Control Policy: $\Theta^{(0)}$}
As we did in the previous section, our final goal is to convert the refined sub-optimal target-cycle $\bar{\Xi}^R$ into a set of threshold control policy values: $\Theta^{(0)}$ so that it can be used in \eqref{Eq:GradientDescent}. Since we still consider only the single agent situations, $\mathcal{A} = \{a\}$, therefore, $\Theta^{(0)} = \Theta^{a(0)}$. Even though now $\bar{\Xi}^R$ can be a constrained target-cycle (as opposed to $\bar{\Xi}^R$ being a simple unconstrained target-cycle as before), we can still use the Algorithm \ref{Alg:initialThresholPolicyGeneration} to get the corresponding $\Theta^{a(0)}$. 

Notice that, now, if a target $i \in \bar{\Xi}^R$ has $n$ auxiliary targets (i.e., when $\vert \mathcal{T}_i \vert = n $) then, $(n+1)$ elements of the $i$\textsuperscript{th} row of $\Theta^{a(0)}$ will be $0$. This is because, from the target $i$, an agent now have to have the ability to go to destinations $j \in \{\Xi_i^{R,k}[1]: k \ni i^k \in  \mathcal{T}_i  \}$ (recall that $\Xi_i^{R,k}$ represents the sub-cycle of the auxiliary target $i^k \in \Xi^{R}$).

On the other hand, this means, when the agent is at target $i$, the target prioritization operation done in step 1 of \eqref{Eq:DiscreteEventSystem} will now compare $n$ distinct destination target uncertainties and pick the target with the maximum uncertainty as the next destination. An example (constrained) target-cycle and the respecting $\Theta^{a(0)}$ obtained from using Algorithm \ref{Alg:initialThresholPolicyGeneration} is given in Fig. \ref{Fig:Thresholds2}.

\begin{figure}[!h]
    \centering
    \includegraphics[width=3.2in]{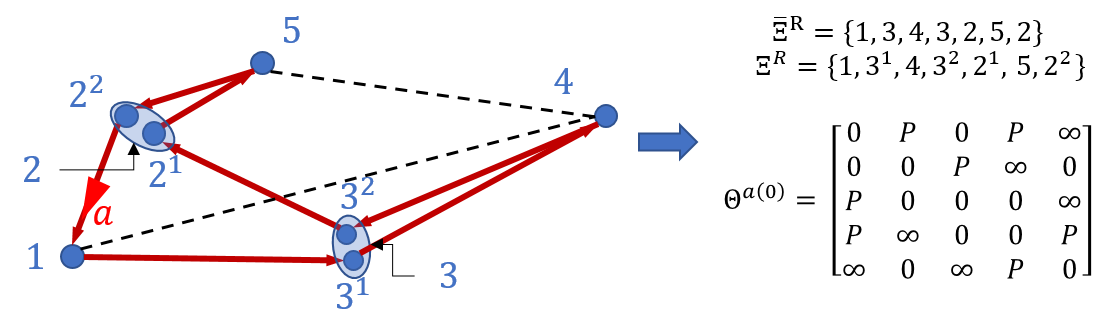}
    \caption{The generated threshold matrix $\Theta^{a(0)}$ for the refined sub-optimal target-cycle $\bar{\Xi}^R$ shown (left).}
    \label{Fig:Thresholds2}
\end{figure}

However, it should be noted that the aforementioned scheme to get $\Theta^{a(0)}$ (i.e. using the Algorithm \ref{Alg:initialThresholPolicyGeneration}) will only work when for any target $i \in \bar{\Xi}^R$, the uncertainty rates of the immediate destinations (i.e., $\{A_j: j \ni j \in \{\Xi_i^{R,k}[1]: k \ni i^k \in  \mathcal{T}_i  \} \}$) are not drastically different from each other. 

Even if it is the case, it can be shown that when target ${j_k} = \Xi_i^{R,k}[1],\ k \ni i^k \in  \mathcal{T}_i$ is used to denote immediate neighbors of target $i$ on the target-cycle $\bar{\Xi}^R$, selecting thresholds such that:
\begin{equation}\label{Eq:ThresholdSelection}
    \Theta^{a(0)}[i][j_k] = \theta^{1(0)}_{ij_k} = A_{j_k}(T_{\bar{\Xi}^R}- \tau_{j_k} - t_{j_k} - \frac{1}{2} T_i^{R,k-1})
\end{equation}
resolves this issue. Then, \eqref{Eq:ThresholdSelection} can be used in step 7 of Algorithm \ref{Alg:initialThresholPolicyGeneration} as opposed to setting $\theta^{1(0)}_{ij_k}=0$, $\forall k \ni i^k \in  \mathcal{T}_i$ and $\forall i \in \bar{\Xi}^R$. (Recall the notation: $T_{\Xi}$ is the total time spent on a target-cycle $\bar{\Xi}$, $\tau_i$ is the steady state dwell time at a target $i$, and, $t_i$ is the time to reach target $i$ from the previous target on the considered target-cycle).

\subsection{Simulation Results}

Figure \ref{Fig:ConstrainedCycleGenerationProcess} (a)$\rightarrow$(d) shows the intermediate cycles generated by the greedy sub-optimal cycle construction process given in Algorithm \ref{Alg:GreedyTargetCycleConstruction2} when applied for the SASE1 problem configuration (See Fig. \ref{Fig:SASE1RandomInitialization}, \ref{Fig:UnconstrainedCycleGenerationProcess} and \ref{Fig:SASE1UnconstrainedCycleInitialization}). The target-cycle shown (as a red contour) in Fig. \ref{Fig:ConstrainedCycleGenerationProcess} (d) is $\bar{\Xi} = \{2, 1, 2, 5, 3, 4, 5\}$ and it has a $J_{ss}$ value of $121.1$. Note that $J_{ss}(\bar{\Xi})$ is better than the $J_{ss}$ value of the unconstrained target-cycle identified before (in Fig. \ref{Fig:UnconstrainedCycleGenerationProcess}(f)) by $+7.6$ ($5.9\%$). This improvement is a result of the increased greedy search space size in Algorithm \ref{Alg:GreedyTargetCycleConstruction2} w.r.t. Algorithm \ref{Alg:GreedyTargetCycleConstruction1}. Also note that in $\bar{\Xi}$, both target $2$ and target $5$ are being visited twice per each cycle.

The identified target-cycle $\bar{\Xi}$ is then converted to the respective TCP using Algorithm \ref{Alg:initialThresholPolicyGeneration}. Fig. \ref{Fig:SASE1ConstrainedCycleInitialization}(b) shows that the target-cycle $\bar{\Xi}$ has a $J_T$ value of $114.9$ which cannot be further improved using the gradient steps \eqref{Eq:GradientDescent}. To ensure $\bar{\Xi}$ is a local optimal, after $100$ iterations (at $l=100$), the derived TCP $\Theta^{(0)}$ is randomly perturbed. Then it can be seen that $\Theta^{(l)}$ converges back to the same initial TCP found (with $J_T=114.6$). It is important to note that this solution is better than the best TCP obtained with a random initialization of $\Theta^{(0)}$ (shown in Fig \ref{Fig:SASE1RandomInitialization}), by $+13.8\ (10.7\%)$. And when compared to the unconstrained target cycle based solution shown in Fig. \ref{Fig:SASE1UnconstrainedCycleInitialization}, it is an improvement of $+6.7\ (5.0\%)$.

In order to illustrate the importance of gradient ascent steps, consider the new single agent simulation example (SASE3) shown in Fig. \ref{Fig:SASE3RandomInitialization}. When the TCP $\Theta^{(0)}$ is selected randomly, the gradient steps have converged to $J_T = 497.9$. Now, Fig. \ref{Fig:SASE3ConstrainedCycleInitialization}(a) shows the performance of the TCP given by the identified refined sub-optimal greedy constrained cycle (obtained using Algorithm \ref{Alg:GreedyTargetCycleConstruction2} and \ref{Alg:initialThresholPolicyGeneration}). As the usual next step, when gradient steps are used \eqref{Eq:GradientDescent}, compared to SASE1, we can now observe a further improvement in $J_T$ (See Fig. \ref{Fig:SASE3ConstrainedCycleInitialization}(b) and (c)) which leads to a TCP $\Theta^*$ with $J_T = 449.5$. Therefore, the overall improvement achieved from deploying the proposing initialization technique is $+48.4\ (9.7\%)$. 

\begin{figure}[!h]
     \centering
     \begin{subfigure}[b]{0.23\columnwidth}
         \centering
         \includegraphics[width = \textwidth]{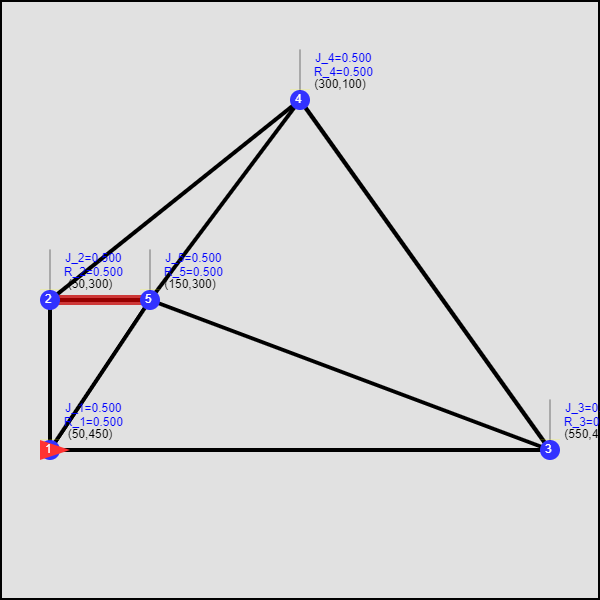}
         \caption{Step 1}
         
     \end{subfigure}
     \hfill
     \begin{subfigure}[b]{0.23\columnwidth}
         \centering
         \includegraphics[width = \textwidth]{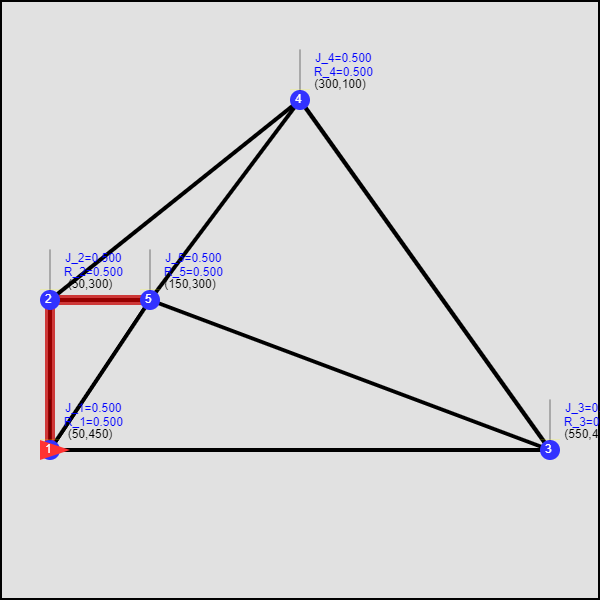}
         \caption{Step 2}
         
     \end{subfigure} 
     \hfill
     \begin{subfigure}[b]{0.23\columnwidth}
         \centering
         \includegraphics[width = \textwidth]{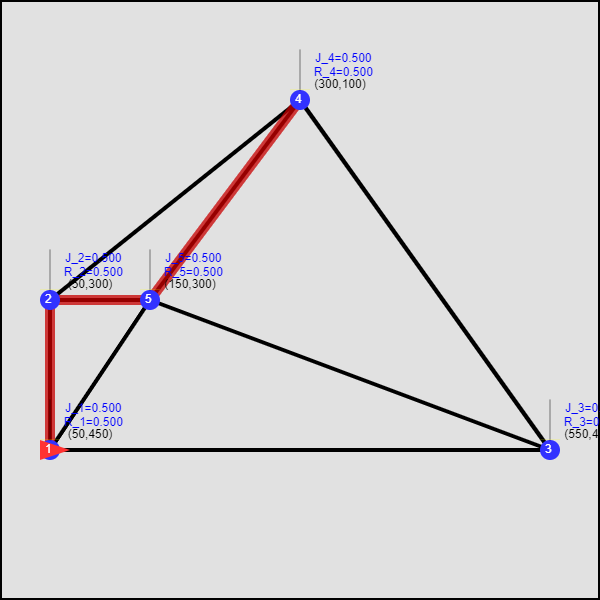}
         \caption{Step 3}
         
     \end{subfigure} 
     \hfill
     \begin{subfigure}[b]{0.23\columnwidth}
         \centering
         \includegraphics[width = \textwidth]{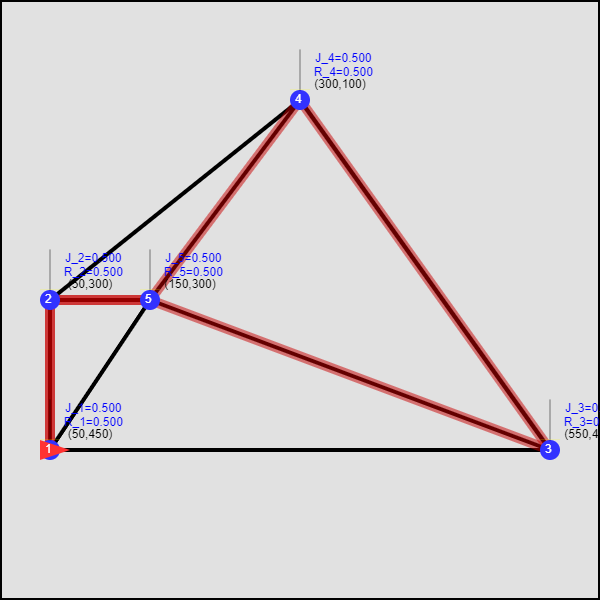}
         \caption{Step 4}
         
     \end{subfigure} 
    \caption{Greedy target-cycle construction iterations (by Algorithm \ref{Alg:GreedyTargetCycleConstruction2}) observed for the target topology of SASE1.}
    \label{Fig:ConstrainedCycleGenerationProcess}
\end{figure}

\begin{figure}[!h]
     \centering
     \begin{subfigure}[b]{0.35\columnwidth}
         \centering
         \includegraphics[width = \textwidth]{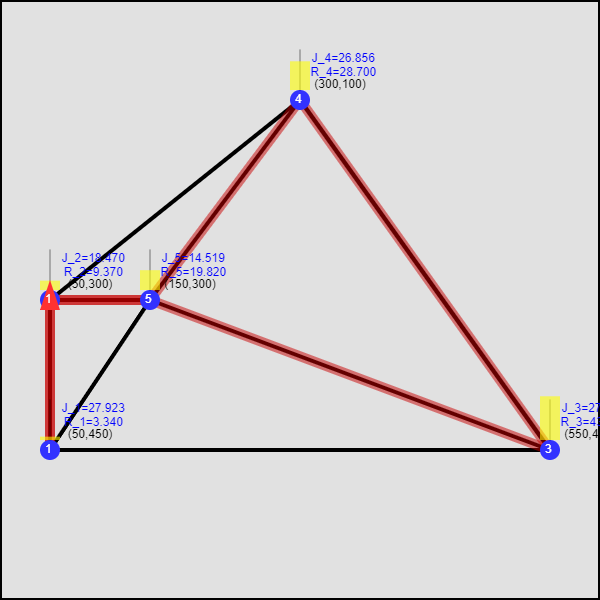}
         \caption{Config. at $t=T$.}
         
     \end{subfigure}
     \hfill
     \begin{subfigure}[b]{0.6\columnwidth}
         \centering
         \includegraphics[width = \textwidth]{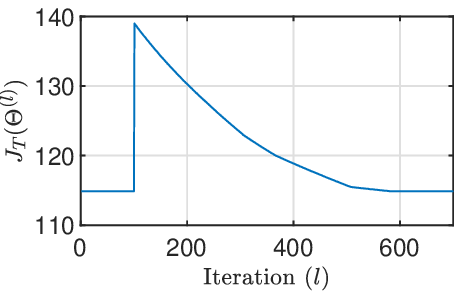}
         \caption{Cost vs iterations plot.}
         
     \end{subfigure}
    \caption{SASE1: The TCP $\Theta^{(0)}$ given by the identified cycle $\bar{\Xi}^R$ (the red trace in (a)) shows local optimality. At $l=100, \Theta^{(l)}$ is randomly perturbed. Yet, converges back to the initial TCP. Cost $J_T=114.9$ (Improvement $=+6.7$ compared to Fig. \ref{Fig:SASE1UnconstrainedCycleInitialization}).}
    \label{Fig:SASE1ConstrainedCycleInitialization}
\end{figure}

\begin{figure}[!h]
     \centering
     \begin{subfigure}[b]{0.35\columnwidth}
         \centering
         \includegraphics[width = \textwidth]{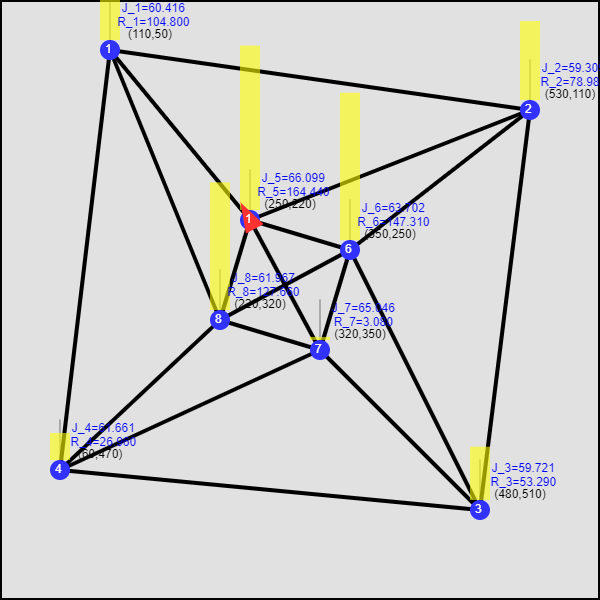}
         \caption{Config. at $t=T$.}
         
     \end{subfigure}
     \hfill
     \begin{subfigure}[b]{0.6\columnwidth}
         \centering
         \includegraphics[width = \textwidth]{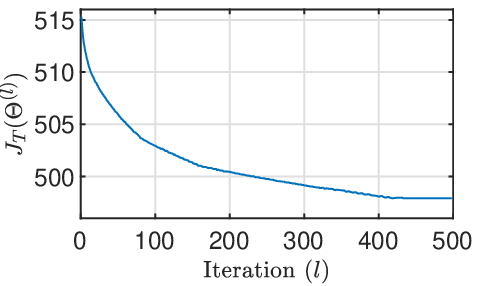}
         \caption{Cost vs iterations plot.}
         
     \end{subfigure}
    \caption{Single agent simulation example 3 (SASE3): Starting with a random $\Theta^{(0)}$, converged to a TCP with the cost $J_T = 497.9$.}
    \label{Fig:SASE3RandomInitialization}
\end{figure}

\begin{figure}[!h]
     \centering
     \begin{subfigure}[b]{0.3\columnwidth}
         \centering
         \includegraphics[width = \textwidth]{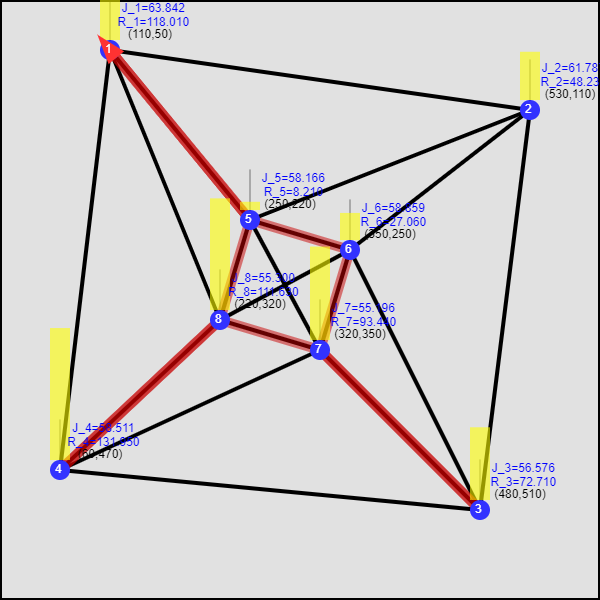}
         \caption{$l=0$, $t=T$.}
         
     \end{subfigure}
     \hfill
     \begin{subfigure}[b]{0.3\columnwidth}
         \centering
         \includegraphics[width = \textwidth]{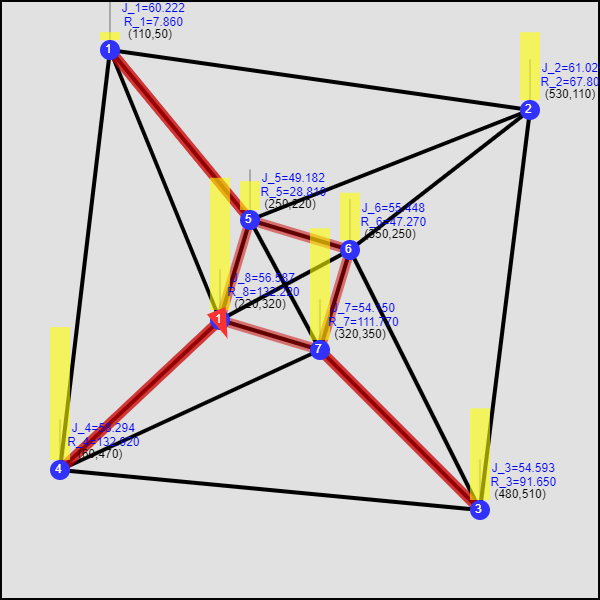}
         \caption{$l=500$, $t=T$}
         
     \end{subfigure}
     \hfill
     \begin{subfigure}[b]{0.35\columnwidth}
         \centering
         \includegraphics[width = \textwidth]{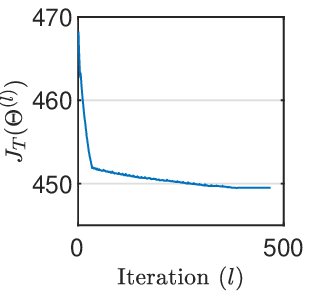}
         \caption{Cost vs iterations.}
         
     \end{subfigure}
    \caption{SASE3:The TCP $\Theta^{(0)}$ given by the identified cycle $\bar{\Xi}^R$ (the red trace in (a),(b)) with cost $J_T = 468.2$ (improvement $=+29.7$ compared to Fig. \ref{Fig:SASE3RandomInitialization}) is further improved by the IPA based gradient updates \eqref{Eq:GradientDescent}. Final cost $J_T=449.5$ (Improvement $=+18.7$ compared to (a)).}
    \label{Fig:SASE3ConstrainedCycleInitialization}
\end{figure}

\section{Multi-Agent PMN Solution} \label{Sec:MultiAgents}
The work presented so far considers only single agent persistent monitoring problems. Now, to extend the developed solution framework into multi-agent persistent monitoring problems, we propose to follow an approach where we first partition the considered graph topology $\mathcal{G}$ into $N$ sub-graphs, and then, allocate each agent $a \in \mathcal{A}$ into different sub-graphs.

Such a `divide-and-conquer' approach was motivated due to three reasons: 
(\romannum{1}) The experimental results of the work in \cite{Zhou2019} suggests that in most graph structures, lower (better) objective function values (i.e., $J_T$ in \eqref{Eq:ObjectiveFunction1}) are attainable when agents do not share targets (as opposed to share targets) in their steady state behavior, 
(\romannum{2}) Extending the developed single agent PMN solution presented in previous sections would be straightforward as once the graph $\mathcal{G}$ is partitioned, it can be thought of as solving a set of independent single agent PMN problems (with each problem on a separate sub-graph), 
(\romannum{3}) Since the overall goal of this work is to find a better initial condition $\Theta^{(0)}$ for the IPA based gradient descent steps in \eqref{Eq:GradientDescent}, we can expect the consequent $\Theta^{(l)}$ updates in \eqref{Eq:GradientDescent} to break the non-cooperative nature of the initial solution $\Theta^{(0)}$ that we propose - if it is sub-optimal.

\subsection{Overview of the Complete Solution}
The steps of the complete solution that we propose for the multi-agent PMN problems introduced in section \ref{Sec:ProblemFormulation} are outlined in the following algorithm. 

\begin{algorithm}[!h]
\caption{The steps of the proposing complete solution to the problem of multi-agent persistent monitoring on networks.}
\begin{algorithmic}[1]
\State \textbf{Input}: \begin{enumerate}
    \item Target topology $\mathcal{G}=(\mathcal{V},\mathcal{E})$, with $\mathcal{V} = \{1,2,\cdots,M\}$ and  $\mathcal{E} \subseteq \{(i,j): i,j \in \mathcal{V}\}$,
    \item Set of agents $\mathcal{A} = \{1,2,\ldots,N\}$,
    \item Initial target uncertainties $\{R_i(0): \in \mathcal{V}\}$, 
    \item Initial agent locations $\{s_a(0): a \in \mathcal{A}\}$,
    \item Finite Time horizon $T$.
\end{enumerate}
\State \textbf{Output}: A locally optimal threshold control policy: $\Theta^*$
\State Partition the given graph $\mathcal{G}$ into $N$ sub-graphs $\{\mathcal{G}_a\}_{a\in \mathcal{A}}$.
\State For each sub-graph $\mathcal{G}_a$, find the refined target-cycle $\bar{\Xi}_a$.
\State Execute iterative refinements to the graph-partitioning.
\State Re-assign agents to cycles based on the shortest distance from initial agent location to the cycles. 
\State Get threshold policies $\Theta^{a(0)}$ of cycles $\bar{\Xi}_a,\ \forall a \in \mathcal{A}$.
\State Use $\Theta^{(0)}$ in \eqref{Eq:GradientDescent} and update $\Theta^{(l)}$ using IPA gradients \cite{Zhou2019}.
\end{algorithmic}\label{Alg:OverallSolution}
\end{algorithm}

In the Algorithm \ref{Alg:OverallSolution}, to execute the steps 4 and 7, we directly utilize the single-agent persistent monitoring techniques discussed in section \ref{Sec:SparseGraphsSingleAgent}. The final step (i.e., step 8) is executed according to \cite{Zhou2019} and was discussed in section \ref{Sec:ProblemFormulation}. The three remaining steps (i.e., step 3, 5 and 6) involve: (\romannum{1}) partitioning the graph, (\romannum{2}) refining the decided graph partitions, and, (\romannum{3}) assigning agents to the graph partitions (i.e., to the cycles on the graph partitions). These three topics will be discussed respectively in the following subsections. 

\subsection{Graph Partitioning using Spectral Clustering}

\paragraph*{\textbf{Introduction}}In order to partition the given target topology $\mathcal{G}$, we use the \emph{spectral clustering} \cite{Luxburg2007} technique which is a commonly used global graph partitioning method. Compared to the traditional clustering techniques such as the k-means method, the spectral clustering technique has few fundamental advantages: (\romannum{1}) it is simple to implement, (\romannum{2}) it can be solved efficiently, and, (\romannum{3}) it delivers better clustering results. These inherent advantages of spectral clustering technique justify its usage in our work. 

In spectral clustering, the graph partitions are derived based on the spectrum of the similarity matrix (also called the affinity matrix) associated with the given graph. In this paradigm, the $(i,j)$\textsuperscript{th} element of the similarity matrix (denoted by $s_{ij} \geq 0$) will represent the similarity between the $i$\textsuperscript{th} node and the $j$\textsuperscript{th} node of the given graph. In terms of the similarity values, the spectral clustering algorithm can be seen as a clustering method which seeks to partition the graph such that nodes in different partitions have a low similarity value between them while the nodes in the same partition have a higher similarity value between them.

\begin{remark}
In a typical data-point clustering application, the graph representation (also called the \textquotedblleft similarity graph\textquotedblright) arises from the known similarity values between the data-points. However, this is not the case for PMN problems where the physical graph $\mathcal{G}$ (which is different from the similarity graph) is known, while the similarity values between its targets are unknown.   
\end{remark}

\paragraph*{\textbf{Deriving Similarity Values}}
We exploit the knowledge of the target topology $\mathcal{G}$ and target parameters to derive appropriate similarity values. Typically, a similarity value $s_{ij} \geq0$ is obtained based on a \emph{disparity} value $d(i,j)$ as
\begin{equation}
\label{Eq:DisparityToSimilarity}
s_{ij} = \exp{\left(-\frac{\vert d(i,j)\vert^{2} }{2\sigma^{2}}\right)},\ \ i,j \in\mathcal{V},
\end{equation}
where $d:\mathcal{V}\times\mathcal{V}\rightarrow\mathbb{R}$ and $\sigma^{2}$ is a user defined scaling parameter that controls how rapidly the similarity $s_{ij}$ falls off with the disparity between $i$ and $j$ (i.e., with
$d(i,j)$) \cite{Ng2001}. This function \eqref{Eq:DisparityToSimilarity} is known as the \emph{Gaussian similarity function}. Note also that the disparity and similarity values are inversely related. We next focus on defining an appropriate disparity metric for the PMN problems. Now, the next step would be to define a good disparity metric for the class of persistent monitoring problems.  

\begin{remark}\label{Rm:InaccurateDisparityMetrics}
For the considered PMN problem setup, neither of using $d(i,j)$ as the \emph{physical distance} (i.e., $\Vert X_{i}-X_{j} \Vert$) nor the \emph{shortest distance} between the targets $i$ and $j$ provides a good characterization to the underlying persistent monitoring aspects of the problem because they disregard target parameters and agent behaviors when monitoring targets.
\end{remark}

Taking the remark \ref{Rm:InaccurateDisparityMetrics} into account, we propose a novel disparity metric named \emph{minimum mean covering cycle uncertainty} (CCU):
\begin{equation}
\label{Eq:DisparityMetricCCU}d(i,j) = d_{CC}(i,j) \triangleq\min_{\bar{\Xi
}:\,i,j \in\bar{\Xi}} J_{ss}(\bar{\Xi}).
\end{equation}
The $\arg\min$ of the above problem is named the \emph{optimal covering cycle} (OCC) and we denote it as $\bar{\Xi}^{*}_{ij}$.

In other words, the OCC $\bar{\Xi}_{ij}^{*}$ is the best way to cover targets $i$ and $j$ in a single target-cycle so that the corresponding steady state mean cycle uncertainty is minimized. Therefore, if the CCU value is higher for a certain target pair, it implies that it is difficult to cover those two targets in a single target-cycle. Hence, it is clear that this disparity metric $d_{CC}(i,j)$ in \eqref{Eq:DisparityMetricCCU} provides a good characterization to the underlying persistent monitoring aspects of the PMN problems - compared to the other two disparity metrics suggested in remark \ref{Rm:InaccurateDisparityMetrics}. As an example, if all the trajectory segments in $\mathcal{E}$ follow the triangle inequality in terms of respective travel times, then, for any  $(i,j)\in\mathcal{E}$, the corresponding OCC is $\bar{\Xi}^{*}_{ij} = \{i,j\}$ and CCU is $J_{ss}(\{i,j\})$.

In order to compute the proposed disparity metric $d_{CC}(i,j), \ \forall i,j \in \mathcal{V}$, we propose the Algorithm \ref{Alg:ModifiedDijkstras} which is a modified version of the famous Dijkstra's algorithm \cite{Ahuja1993} coupled with cycle expanding and refining techniques discussed in section \ref{Sec:SparseGraphsSingleAgent}. Disparity metric values given by Algorithm \ref{Alg:ModifiedDijkstras} are then used in \eqref{Eq:DisparityToSimilarity} to compute the similarity values: $s_{ij}\ \forall i,j \in \{1,2,\ldots,M\}$. Subsequently, the obtained inter-target similarity values are used to form the similarity matrix $S$ where, 
$$S = [\{s_{ij}\}_{\forall (i,j)}] \in \R^{M \times M}.$$   


\begin{algorithm}[!h]
\caption{Modified Dijkstra's algorithm used to find the proposing disparity metric values: $d_{CC}(i,j)\ \forall j \in \mathcal{V}$}
\begin{algorithmic}[1]
\State \textbf{Input}: A start node $i$, Target topology $G = (\mathcal{V},\mathcal{E})$. 
\State \textbf{Output}: $d_{CC}(i,j),\ \forall j \in \mathcal{V}.$
\State $V:= \emptyset;\ U:=\mathcal{V};$ 
\Comment{Visited/unvisited target sets.}
\State $C_{CC}(j):= \{\}, \ \forall j \in \mathcal{V};$
\Comment{Null cycles.}
\State $C_{CC}(i):= \{ i \};$ 
\Comment{A cycle with only $i$ on it.}
\State $J_{CC}(j):=\infty, \ \forall j \in \mathcal{V};\ \ J_{CC}(i):=0;$
\Comment{CCU values.}
\While{$U \neq \emptyset$} \Comment{Till all targets get their own OCC.}
    \State ${j^*} := \arg\min_{j \in U} J_{CC}(j);$
    \State $\bar{\Xi}_{ij^*}^*:= C_{CC}({j^*});$ 
    \State $V:=V \cup {j^*}; \ \ U:=U\backslash {j^*};$ 
    \Comment{${j^*}$ got its own OCC.}
    \For{$k \ni (j^*,k) \in \mathcal{E}$ } 
        \State $\bar{\Xi}^*_{ik} := $ when $\bar{\Xi}^*_{ij^*}$ expanded (and then refined) to include target $k$ in the best way out of three TCEOs discussed;
        \If{$J_{CC}(k) > J_{ss}(\bar{\Xi}^*_{ik})$}
        \Comment{Found a better OCC!}
            \State $C_{CC}(k) := \bar{\Xi}^*_{ik};$ 
            \Comment{OCC updated.}
            \State $J_{CC}(k) := J_{ss}(\bar{\Xi}^*_{ik});$
            \Comment{CCU updated.}
        \EndIf
    \EndFor
\EndWhile
\State \textbf{Return} $J_{CC}(j),\ \forall j \in \mathcal{V};$
\end{algorithmic}\label{Alg:ModifiedDijkstras}
\end{algorithm}

\paragraph*{\textbf{Spectral Clustering Algorithm}}
In the process of spectral clustering, computation of `Weighted Adjacency Matrix ($W$)' and `Degree Matrix ($D$)' is done using the obtained similarity matrix $S$. Then, those matrices are used to compute the Laplacian matrix $L$ (unnormalized) corresponding to the similarity matrix $S$.  

In literature, often the similarity matrix refers to a matrix with disparity values $d(i,j)$ as its elements. However, we have constructed the similarity matrix $S$ by directly transforming the disparity values got via Algorithm \ref{Alg:ModifiedDijkstras} using \eqref{Eq:DisparityToSimilarity}. Therefore, in our case, the weighted adjacency matrix is the similarity matrix. Thus, 

\begin{equation}\label{Eq:WeightedAdjacencyMatrix}
    W = [\{w_{ij}\}_{\forall (i,j)}] = S = [\{s_{ij}\}_{\forall (i,j)}] \in \R^{M \times M}.
\end{equation}
The degree matrix $D$ is obtained using the elements of $W$ as, 
\begin{equation}\label{Eq:DegreeMatrixForSC}
    D = diag([d_1,d_2,\cdots,d_M]^T); \mbox{ with } d_i = \sum_{j=1}^M w_{ij}.
\end{equation}
The unnormalized Laplacian matrix $L$ is then computed as,  
\begin{equation}\label{Eq:UnnormalizedLaplacianForSC}
    L = D - W.
\end{equation}

It should be noted that there are variants of the spectral clustering algorithm which uses normalized Laplacian matrices denoted by $L_{sym}$ and $L_{rw}$ as opposed to using $L$ \cite{Luxburg2007}. Here, 
\begin{eqnarray}
    L_{sym} &= D^{-\frac{1}{2}}LD^{-\frac{1}{2}},\\
    L_{rw} &= D^{-1}L.
\end{eqnarray}

In our work, we use the normalized spectral clustering method proposed in \cite{ShiMalik2000}, which utilizes the normalized Laplacian $L_{rw}$. We chose this method because it has a somewhat relatable (to persistent monitoring) interpretation based on random walks on the similarity graph. Specifically, a random walk on a graph can be seen as a stochastic process where an agent randomly jumps from vertex to vertex. The normalized spectral clustering method proposed in \cite{ShiMalik2000} can be interpreted as trying to find a partition of the similarity graph such that the random walk stays long within the same cluster and seldom jumps between clusters.

Algorithm \ref{Alg:SpectralClusteringAlgorithm} outlines the normalized spectral clustering procedure (based on \cite{ShiMalik2000}) to get the target clusters $\mathcal{V}_1,\mathcal{V}_2,\ldots,\mathcal{V}_N$. Each target cluster $\mathcal{V}_a,\ a\in\mathcal{A}$ can then be used to form a sub-graph out of the given target topology $\mathcal{G}=(\mathcal{V},\mathcal{E})$ as $\mathcal{G}_a = (\mathcal{V}_a,\mathcal{E}_a)$ where $\mathcal{E}_a\subseteq \mathcal{E}$ is the set of intra-cluster edges taken from complete set of edges $\mathcal{E}, \ \forall a \in \{1,2,\ldots,N\}$. Note that the set of inter-cluster edges (i.e., $\mathcal{E}^I = \mathcal{E}\backslash \cup_{a=1}^N \mathcal{E}_a$) are now not included in any of the formed sub-graphs.

\begin{algorithm}[!h]
\caption{Normalized Spectral Clustering \cite{ShiMalik2000}}
\begin{algorithmic}[1]
\State \textbf{Input}: Normalized Laplacian $L_{rw}$.
\State \textbf{Output}: Target clusters $\mathcal{V}_1,\mathcal{V}_2,\ldots,\mathcal{V}_N$.
\State Compute the first $N$ eigenvectors $u_1,u_2,\cdots,u_N$ of $L_{rw}$.
\State Let $U \in \R^{N \times M}$ be the matrix containing $u_1,u_2,\cdots,u_N$ as columns.
\State For $i = 1,\cdots,M$ let $y_i \in \R^N$ be the vector corresponding to the $i$\textsuperscript{th} row of $U$.
\State Cluster the data points $\{y_i\}_{i=1,\cdots,M} \in \R^N$ using a k-means algorithm (where $k=N$) into $N$ clusters $C_1, \cdots, C_N$.
\State \textbf{Return} Target clusters $\mathcal{V}_1,\mathcal{V}_2,\ldots,\mathcal{V}_N$ with $\mathcal{V}_{i} = \{j: j \ni y_j \in C_i\}$
\end{algorithmic}\label{Alg:SpectralClusteringAlgorithm}
\end{algorithm}

\subsection{Balancing the Obtained Graph Partitions}
Once the sub-graphs are formed, as suggested in step 4 of the complete solution procedure given in Algorithm \ref{Alg:OverallSolution}, we follow the refined sub-optimal target-cycle construction procedure (discussed in section \ref{Sec:SparseGraphsSingleAgent}) for each sub-graph. The resulting target-cycle on a sub-graph $\mathcal{G}_{a}$ is denoted as $\bar{\Xi}_{a}$ and is assumed to be assigned to an arbitrary agent $a\in\mathcal{A}$. Note, however, that in Section \ref{SubSec:AssigningAgents}, we will explicitly assign target-cycles to the agents.

The obtained set of clusters is called \emph{balanced} if the steady state mean cycle uncertainties $J_{ss}(\bar{\Xi}_a)$ (on the sub-graph $\mathcal{G}_a$ by agent $a$) are approximately identical for all $a \in \mathcal{A}$. Despite the dependence on the nature of the given target topology $\mathcal{G}$, the spectral clustering method is often able to provide a balanced set of clusters. However, when this is not the case, it is intuitive to think of an inter-cluster target exchange scheme - which aims to balance the set of clusters by iteratively modifying them. 

Note that such a target exchange operation between two given clusters will affect the constructed target-cycles on those clusters. However, since we can evaluate the steady state performance of target-cycles, we can use this knowledge to identify globally beneficial inter-cluster target exchange operations. 

This process can also be seen as a situation where $N$ agents (i.e., $\mathcal{A}$) trying to exchange their owned set of resources (i.e., the targets) among each other so that a global objective (i.e., $\sum_{a=1}^N J_{ss}(\bar{\Xi}_a)$) is minimized. Note that, in our case, each cluster is assigned to an agent. Therefore, we can think of the cluster changes as decisions taken by the assigned agent.

In such a paradigm, note that the net global effect of a target exchange operation between two neighboring agents (i.e., neighboring clusters) can be independently (from others) computed. Therefore, a distributed greedy algorithm is proposed here to search and execute globally beneficial target exchange operations iteratively. An iteration of the proposed algorithm for a generic agent $a\in \mathcal{A}$ is given in Algorithm \ref{Alg:PostTargetClusterRefinement}.

Note that, when the agent $a \in \mathcal{A}$ wants to expand its cluster $\mathcal{V}_a$ by adding a new target $i$ to it, the agent $a$ needs to choose the most beneficial target-cycle expansion operation out of the three discussed TCEOs to expand $\bar{\Xi}_a$ so that it includes the target $i$ (See step 3 and 18 in Algorithm \ref{Alg:PostTargetClusterRefinement} - also similar to step 12 of Algorithm \ref{Alg:ModifiedDijkstras}). In contrast, when an agent $a \in \mathcal{A}$ wants to remove a target $i$ from its cluster $\mathcal{V}_a$, he needs to recompute his target-cycle completely on $\mathcal{V}_a\backslash \{ i \}$ (See step 9 and 22 of Algorithm \ref{Alg:PostTargetClusterRefinement}).  

In all, the proposing Algorithm \ref{Alg:PostTargetClusterRefinement} helps to balance/distribute the persistent monitoring load among the agents (clusters) uniformly. This load balancing technique also relieves the need to have a properly chosen neighborhood width parameter $\sigma$ in \eqref{Eq:DisparityToSimilarity} for the spectral clustering. Also, since the cluster modifications (step 18 and 22) are carried out when only they lead to global cost improvement, the given algorithm should converge after a finite number of iterations. Specifically, the convergence criterion is the event where all the agents fail to find a feasible solution to the step 12 of Algorithm \ref{Alg:PostTargetClusterRefinement}.

\begin{algorithm}[!h]
\caption{An iteration of the proposed post target cluster/cycle refinement algorithm:}
\begin{algorithmic}[1]
\State \textbf{Input}: Agent $a$'s initial target cluster and the target-cycle: $\{\mathcal{V}_a,\bar{\Xi}_a\}$

    \For{each (neighbor) $i \not \in \mathcal{V}_a$ but $i \in \mathcal{V}_b$}
        \State Find the best way to append $i$ to $\bar{\Xi}_a$ 
        \State \textbf{Store}:  $(J^A_{a,i},\bar{\Xi}^A_{a,i}) :=$ resulting gain and the cycle;  
        \State \textbf{Offer}: $f_{a,i}$: $\{a,i,J^A_{a,i}\}$ to agent $b$ \textbf{if} $J^A_{a,i} > 0$. 
    \EndFor
\State ------

    \For{each $i \in \bar{\Xi}_a$ with an external offer}
        \State Recompute a new cycle on $\mathcal{V}_a \backslash \{i\}$ (detaching). 
         \State \textbf{Store}: $(J^D_{a,i},\bar{\Xi}^D_{a,i}) :=$ resulting gain and the cycle;
    \EndFor
    
    \State Best offer: $f_{b^*,i^*}$:   $\substack{\arg\max\\(b,i):J>0}\ J = J^{Net}_{b,a,i} = (J^A_{b,i}+J^D_{a,i})$
    \State \textbf{Acknowledge} $Ack_{b^*,a,i} = \{b^*,a,i,J^{Net}_{b^*,a,i}\}$ to $b^*$.
\State ------

    \State Best acknowledgement received: $Ack_{a,c^*,i}$ (valued $J^{Net}_{a,c^*,i}$).
    \If{$J^{Net}_{a,c^*,i} > J^{Net}_{b^*,a,i}$} 
    \Comment{Received $Ack$ is better than sent.}
        \If{ agent $c^*$ has no other commitments}
            \State $\mathcal{V}_a := \mathcal{V}_a \cup \{i\}$; (also update $\bar{\Xi}_a$). 
            \Comment{Appended $i$.}
        \EndIf
    \ElsIf{$J^{Net}_{a,c^*,i} < J^{Net}_{b^*,a,i}$} 
    \Comment{Sent $Ack$ is better.}
        \If{Agent $b^*$ has no other commitments}
            \State $\mathcal{V}_a := \mathcal{V}_a \backslash \{i\}$; (also update $\bar{\Xi}_a$).
            \Comment{Detached $i$.}
        \EndIf
    \EndIf 

\end{algorithmic}\label{Alg:PostTargetClusterRefinement}
\end{algorithm}

\subsection{Assigning Agents to The Clusters}
\label{SubSec:AssigningAgents}
So far, we have identified a set of target-cycles $\{\bar{\Xi}_{b}:b\in\mathcal{B}\}$ on the corresponding set of balanced sub-graphs of $\mathcal{G}$, where $\mathcal{B}$ is the set of target-cycle indexes (identical to the set $\mathcal{A}$). We now explicitly assign these target-cycles $\{\bar{\Xi}_{b}:b\in\mathcal{B}\}$ to the agents based on initial agent locations $\{s_{a}(0):a\in\mathcal{A}\}$.

First, let us define the assignment cost between an agent $a\in\mathcal{A}$ and a target-cycle $\bar{\Xi}_{b},\, b\in\mathcal{B}$ as $h_{ab}$ where $h_{ab}$ represents the total travel time on the fastest available path to reach any one of the targets in $\bar{\Xi}_{b}$ starting from $s_{a}(0)$. We use Dijkstra's shortest path algorithm \cite{Ahuja1993} to compute all these assignment weights. Subsequently, the assignment problem (between $a$'s and $b$'s) is solved using the shortest augmenting path algorithm \cite{Ahuja1993}.

\paragraph*{\textbf{Generating an Initial TCP: $\Theta^{(0)}$}} 

Let us assume agent $a\in\mathcal{A}$ is optimally assigned to the target-cycle $\bar{\Xi}_{b}$ and the corresponding fastest path from $s_{a}(0)$ to reach $\bar{\Xi}_{b}$ is $\Phi_{ab} = \{{i_{1}},{i_{2}},\ldots,{i_{n}}\}\subset\mathcal{V}$. Note that $i_{n} \in\bar{\Xi}_{b}$ and $X_{i_{1}}=s_{a}(0)$. Next, let us define $\Phi_{ab}^{\prime}= \Phi_{ab} \backslash\{{i_{n}}\}$. We now use Alg. \ref{Alg:initialThresholPolicyGeneration} with $\bar{\Xi}_{b}$ to get a corresponding TCP for agent $a$ as $\Theta^{a}$. Note that Alg. \ref{Alg:initialThresholPolicyGeneration} only assigns the set of rows: $\{j: j \in\bar{\Xi}_{b}\}$ in $\Theta^{a}$ as it is sufficient to keep the agent on the target-cycle $\bar{\Xi}_{b}$ (this corresponds to rows 1-3 in the example TCP $\Theta^{a}$ shown in Fig. \ref{Fig:Thresholds3}). Therefore, to make sure that agent $a$ follows the path $\Phi_{ab}$, we assign the set of rows: $\{j:j \in\Phi_{ab}^{\prime}\}$ in $\Theta^{a}$ as follows (this corresponds to rows 4-5 in the example TCP $\Theta^{a}$ shown in Fig. \ref{Fig:Thresholds3}). If $j$ and $k$ are two consecutive entries in $\Phi_{ab}$, in the $j$\textsuperscript{th} row of $\Theta^{a}$ we set: $\theta_{jj}^{a}=0, \theta_{jk}^{a}=0$ and any other entry $\theta_{jl}^{a}$ is set to $P$ or $\infty$ depending on whether $(j,l) \in\mathcal{E}$. Finally, we set $\Theta^{a(0)}=\Theta^{a}$.

\begin{figure}[!h]
    \centering
    \includegraphics[width=3.2in]{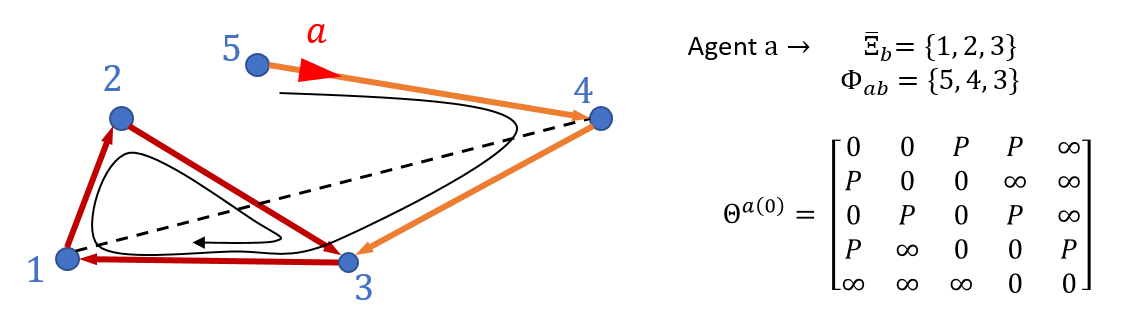}
    \caption{The generated initial TCP $\Theta^{a(0)}$ when the agent $a$ is initially at target $5$ and have been assigned to the target-cycle $\bar{\Xi}_{b} = \{3,1,2\}$ with the fastest path being $\Phi_{ab} = \{5,4,3\}$.}
    \label{Fig:Thresholds3}
\end{figure}

With this, we now have discussed all the steps involved in the overall solution procedure (outlined in Algorithm \ref{Alg:OverallSolution}) that we propose for the PMN problems.

\subsection{Simulation Results}

\paragraph*{\textbf{Effect of Graph Partitioning}} First, we illustrate the effect of graph partitioning by considering the multi-agent simulation example MASE1 given in Fig. \ref{Fig:MASE1RandomInitialization} along with three new such examples named MASE2, MASE3, and MASE4 given in Fig. \ref{Fig:MASE2RandomInitialization}, \ref{Fig:MASE3RandomInitialization} and \ref{Fig:MASE4RandomInitialization} respectively. Note that in MASE2, only two agents are deployed whereas, in the other three configurations, three agents are deployed. Also, note that when the initial TCP $\Theta^{(0)}$ is chosen randomly, the gradient steps have converged to TCPs with $J_T$ values $270.2,\ 91.7,\ 274.0,$ and $201.3$ respectively in each MASE.

Now, when the proposing graph partitioning techniques is applied to each graph (i.e., the step 3 of Algorithm \ref{Alg:OverallSolution}), sub-graphs shown in Fig. \ref{Fig:TargetClusteringResultsWithoughtBalancing1} are formed. Fig. \ref{Fig:TargetClusteringResultsWithoughtBalancing2} shows the generated constrained target-cycles (in red color) within each of the formed sub-graphs under each MASE (i.e., the step 4 of Algorithm \ref{Alg:OverallSolution}). Next, to highlight sole effect of graph partitioning, we skip the graph partitioning refinements (i.e., the step 5 of Algorithm \ref{Alg:OverallSolution}) and continue to generate suitable initial TCPs (i.e., steps 6 and 7 in Algorithm \ref{Alg:OverallSolution}). Finally, when the generated initial TCPs were used in gradient descent \eqref{Eq:GradientDescent} (i.e., step 8 in Algorithm \ref{Alg:OverallSolution}), the observed optimal TCPs had the $J_T$ values as $112.9,\ 45.2,\ 62.5$ and $63.7$ respectively. Fig. \ref{Fig:TargetClusteringResultsWithoughtBalancing2} shows the terminal (i.e., at $t=T$) conditions of the problem configurations under these optimal TCPs. The respective overall improvements achieved from deploying the proposing initialization technique are: $+157.3\ (58.2\%),\ +46.5\ (50.7\%),\ +211.5\ (77.2\%),$ and $+137.6\ (68.4\%)$. Therefore, it is clear that the proposing graph partitioning based method (even without graph partitioning refinements) is capable of delivering much improved solutions.

\begin{figure}[!h]
     \centering
     \begin{subfigure}[b]{0.35\columnwidth}
         \centering
         \includegraphics[width = \textwidth]{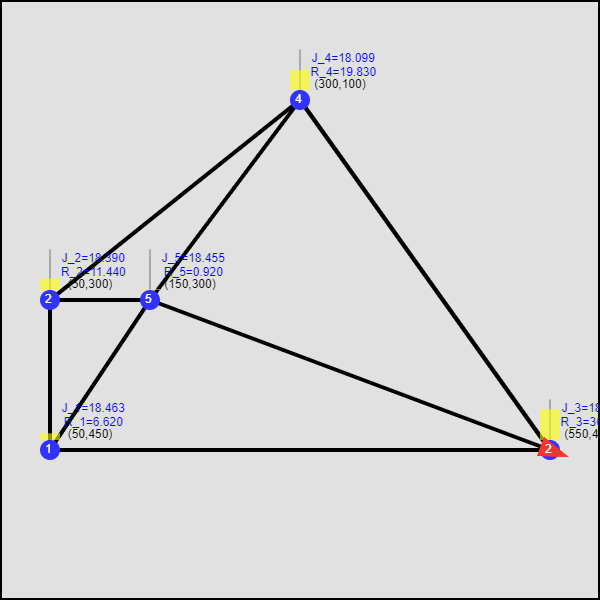}
         \caption{Config. at $t=T$.}
         
     \end{subfigure}
     \hfill
     \begin{subfigure}[b]{0.6\columnwidth}
         \centering
         \includegraphics[width = \textwidth]{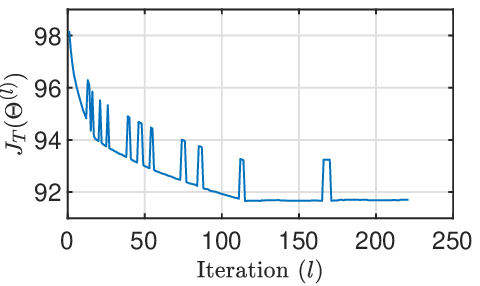}
         \caption{Cost vs iterations plot.}
         
     \end{subfigure}
    \caption{Multi-agent simulation example 2 (MASE2): Starting with a random $\Theta^{(0)}$, converged to a TCP with the cost $J_T = 91.7$.}
    \label{Fig:MASE2RandomInitialization}
\end{figure}

\begin{figure}[!h]
     \centering
     \begin{subfigure}[b]{0.35\columnwidth}
         \centering
         \includegraphics[width = \textwidth]{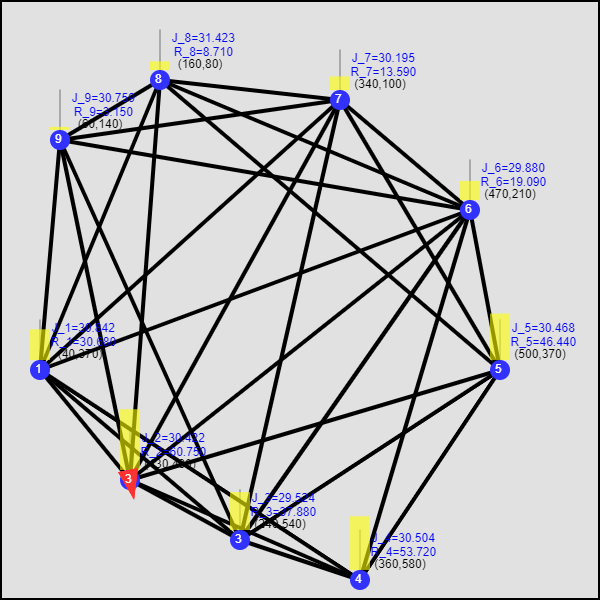}
         \caption{Config. at $t=T$.}
         
     \end{subfigure}
     \hfill
     \begin{subfigure}[b]{0.6\columnwidth}
         \centering
         \includegraphics[width = \textwidth]{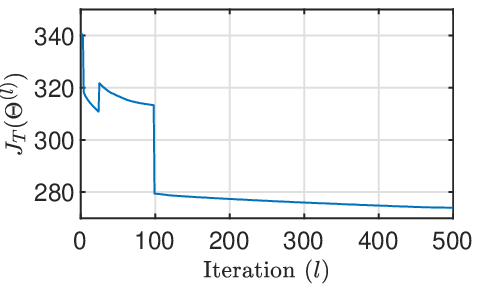}
         \caption{Cost vs iterations plot.}
         
     \end{subfigure}
    \caption{Multi-agent simulation example 3 (MASE3): Starting with a random $\Theta^{(0)}$, converged to a TCP with the cost $J_T = 274.0$.}
    \label{Fig:MASE3RandomInitialization}
\end{figure}

\begin{figure}[!h]
     \centering
     \begin{subfigure}[b]{0.35\columnwidth}
         \centering
         \includegraphics[width = \textwidth]{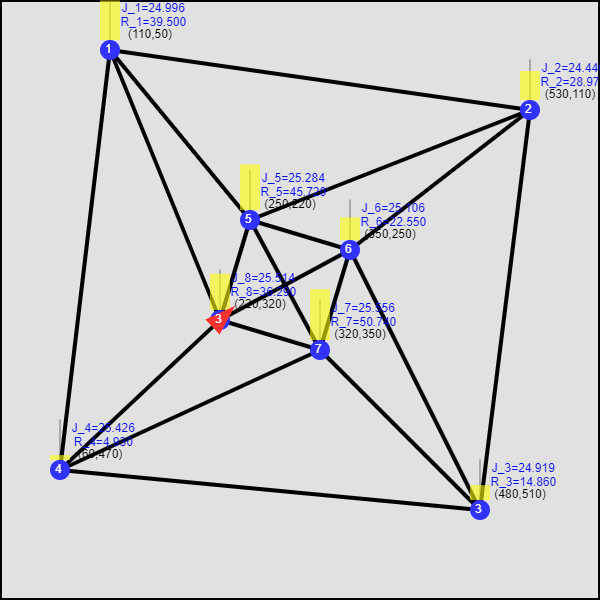}
         \caption{Config. at $t=T$.}
         
     \end{subfigure}
     \hfill
     \begin{subfigure}[b]{0.6\columnwidth}
         \centering
         \includegraphics[width = \textwidth]{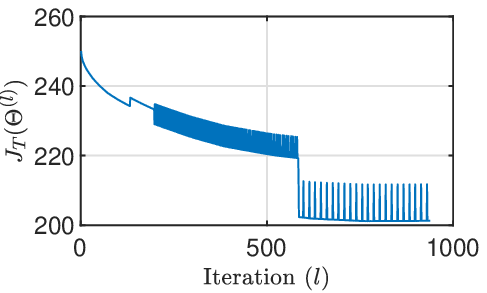}
         \caption{Cost vs iterations plot.}
         
     \end{subfigure}
    \caption{Multi-agent simulation example 4 (MASE4): Starting with a random $\Theta^{(0)}$, converged to a TCP with the cost $J_T = 201.3$.}
    \label{Fig:MASE4RandomInitialization}
\end{figure}

\begin{figure}[!h]
     \centering
     \begin{subfigure}[b]{0.23\columnwidth}
         \centering
         \includegraphics[width = \textwidth]{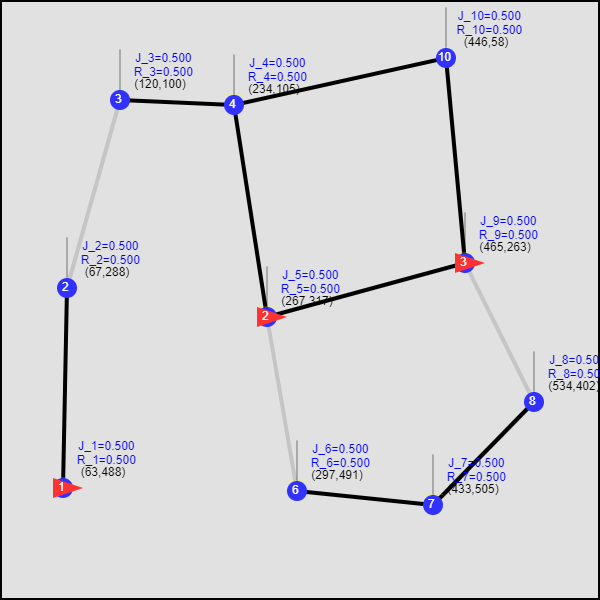}
         \caption{MASE1}
         
     \end{subfigure}
     \hfill
     \begin{subfigure}[b]{0.23\columnwidth}
         \centering
         \includegraphics[width = \textwidth]{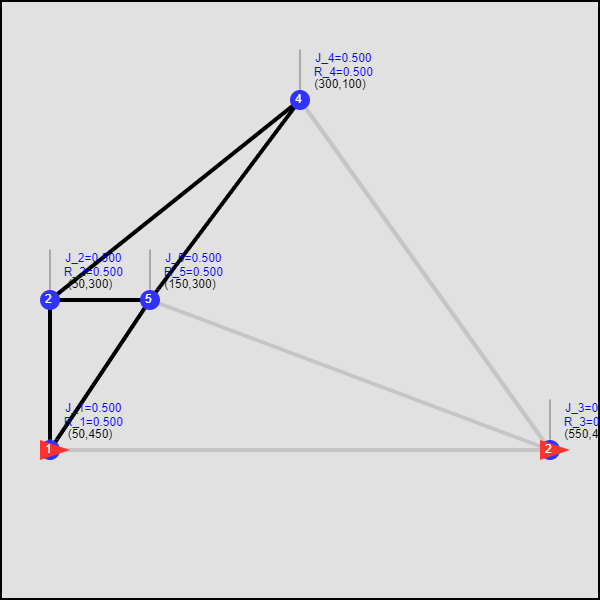}
         \caption{MASE2}
         
     \end{subfigure} 
     \hfill
     \begin{subfigure}[b]{0.23\columnwidth}
         \centering
         \includegraphics[width = \textwidth]{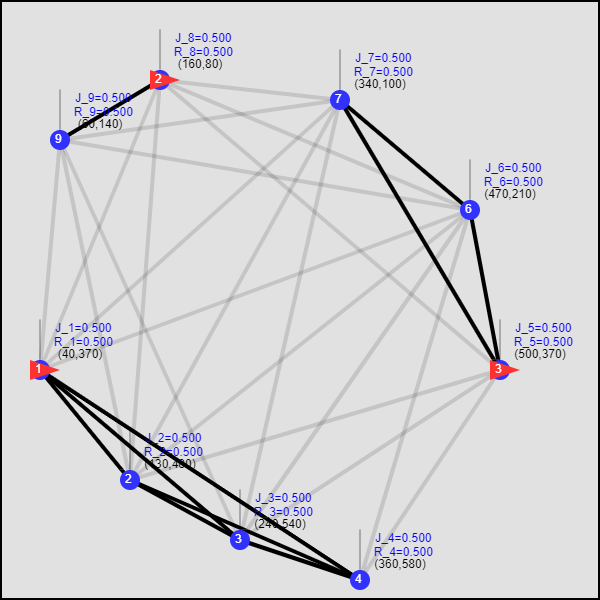}
         \caption{MASE3}
         
     \end{subfigure} 
     \hfill
     \begin{subfigure}[b]{0.23\columnwidth}
         \centering
         \includegraphics[width = \textwidth]{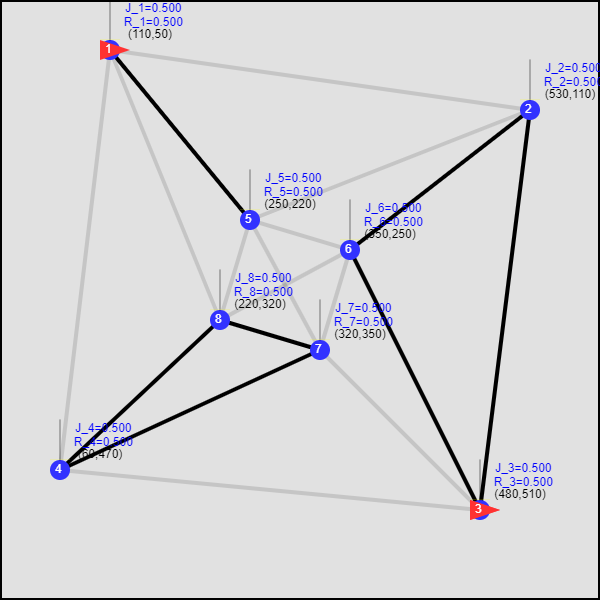}
         \caption{MASE4}
         
     \end{subfigure} 
    \caption{Clustering results obtained for the considered MASEs.}
    \label{Fig:TargetClusteringResultsWithoughtBalancing1}
\end{figure}

\begin{figure}[!h]
     \centering
     \begin{subfigure}[b]{0.23\columnwidth}
         \centering
         \includegraphics[width = \textwidth]{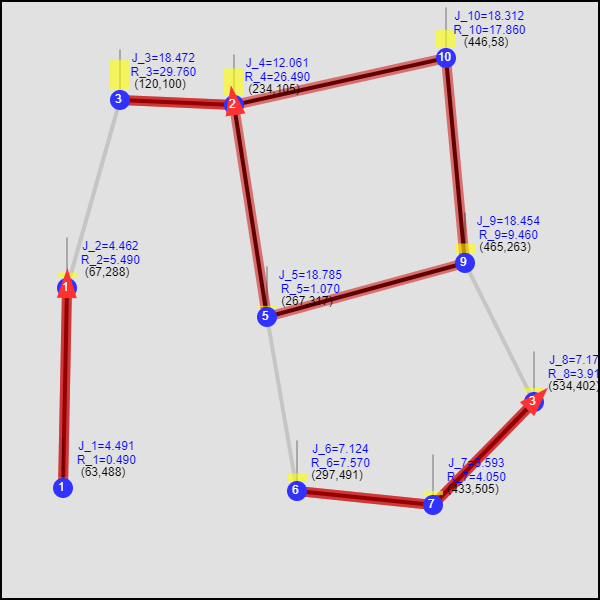}
         \caption{MASE1}
         
     \end{subfigure}
     \hfill
     \begin{subfigure}[b]{0.23\columnwidth}
         \centering
         \includegraphics[width = \textwidth]{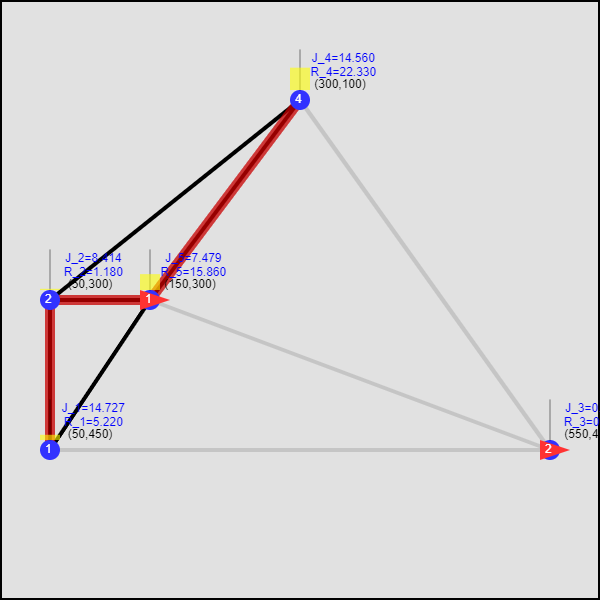}
         \caption{MASE2}
         
     \end{subfigure} 
     \hfill
     \begin{subfigure}[b]{0.23\columnwidth}
         \centering
         \includegraphics[width = \textwidth]{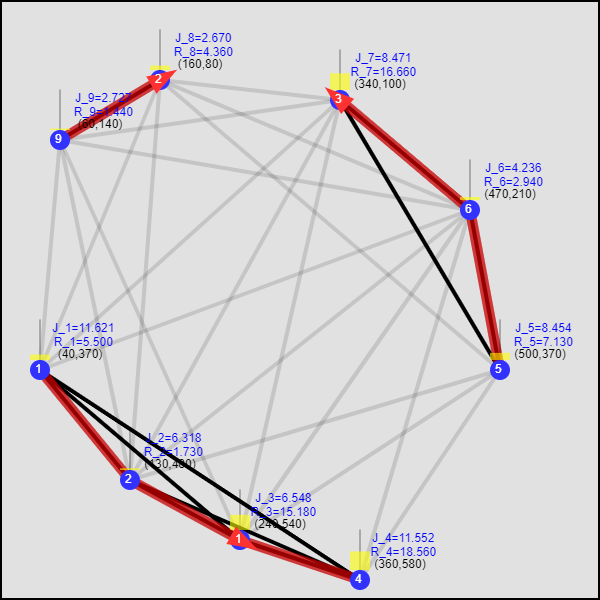}
         \caption{MASE3}
         
     \end{subfigure} 
     \hfill
     \begin{subfigure}[b]{0.23\columnwidth}
         \centering
         \includegraphics[width = \textwidth]{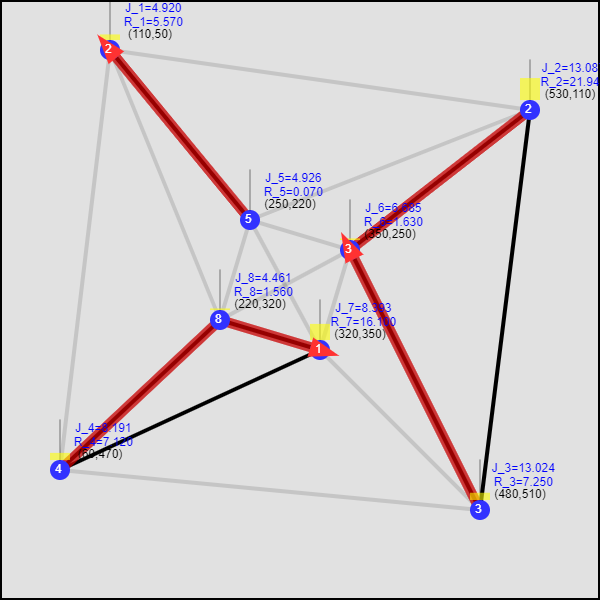}
         \caption{MASE4}
         
     \end{subfigure} 
    \caption{Target-cycles formed on each sub-graphs (red colored contours). The configurations have been drawn at $t=T$ under the optimal TCPs found using Algorithm \ref{Alg:OverallSolution} with its step 5 skipped (i.e., the graph partitioning refinement step). Respective cost (improvement) values compared to Fig. \ref{Fig:MASE1RandomInitialization},\ref{Fig:MASE2RandomInitialization},\ref{Fig:MASE3RandomInitialization} and \ref{Fig:MASE4RandomInitialization}: $112.9\ (+157.3),\ 45.2\ (+46.5),\ 62.5,\ (+211.5) $ and $63.7\ (+137.6)$ }
    \label{Fig:TargetClusteringResultsWithoughtBalancing2}
\end{figure}

\paragraph*{\textbf{Effect of Graph Partitioning Refinements}} 
As the next step, we illustrate the effect of graph partitioning refinements (i.e., the step 5 of Algorithm \ref{Alg:OverallSolution}) using the MASE1. Fig. \ref{Fig:TargetClusteringAndRefinementProcess1} shows the evolution of graph partitions (and the respective target-cycles) through two graph partitioning refinement steps (also called inter-cluster target exchange steps). However, in there, we have used an appropriately computed neighborhood width parameter value ($\sigma = 32.9$) in  \eqref{Eq:DisparityToSimilarity} to obtain the initial graph partitions shown in Fig. \ref{Fig:TargetClusteringAndRefinementProcess1}(a). 

However, Fig. \ref{Fig:TargetClusteringAndRefinementProcess2}(a) shows a situation where the initial graph partitions have been obtained using a poorly chosen $\sigma$ value ($\sigma = 7.67$). The subsequent partitioning refinement steps are shown in Fig. \ref{Fig:TargetClusteringAndRefinementProcess2}(b)$\rightarrow$(f). Notice that terminal set of graph partitions shown in Fig. \ref{Fig:TargetClusteringAndRefinementProcess1}(d) and Fig. \ref{Fig:TargetClusteringAndRefinementProcess2}(f) are identical. Therefore, it illustrates the fact that we can start with a poor set of initial graph partitions and improve upon them iteratively to get to a better set of graph partitions, using the introduced graph partitioning refinement steps (i.e., the Algorithm \ref{Alg:PostTargetClusterRefinement}).

\begin{figure}[!h]
     \centering
     \begin{subfigure}[b]{0.23\columnwidth}
         \centering
         \includegraphics[width = \textwidth]{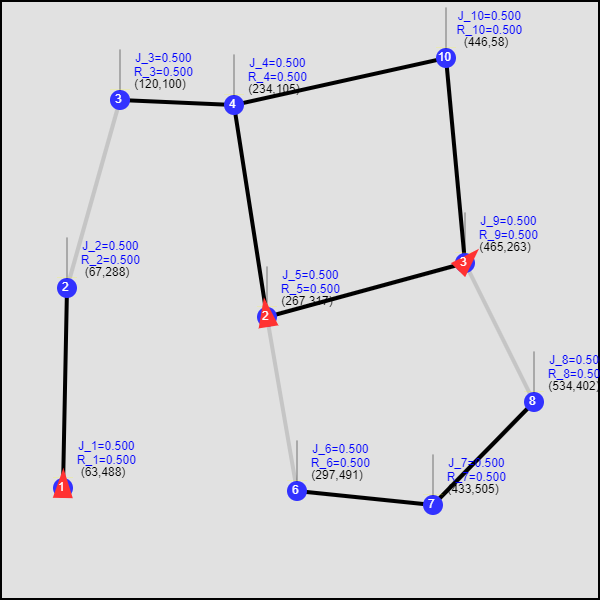}
         \caption{Clusters}
         
     \end{subfigure}
     \hfill
     \begin{subfigure}[b]{0.23\columnwidth}
         \centering
         \includegraphics[width = \textwidth]{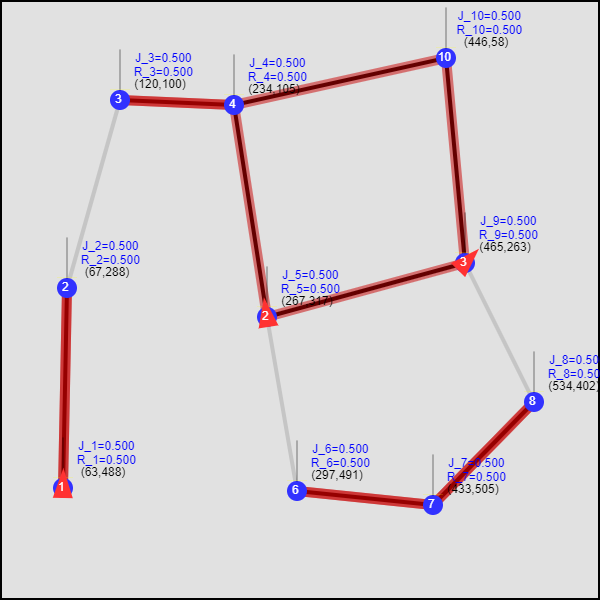}
         \caption{Cycles}
         
     \end{subfigure} 
     \hfill
     \begin{subfigure}[b]{0.23\columnwidth}
         \centering
         \includegraphics[width = \textwidth]{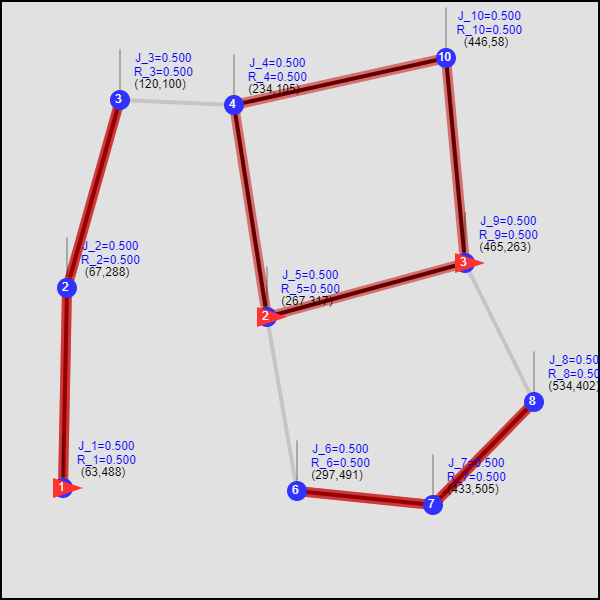}
         \caption{Exchange 1}
         
     \end{subfigure} 
     \hfill
     \begin{subfigure}[b]{0.23\columnwidth}
         \centering
         \includegraphics[width = \textwidth]{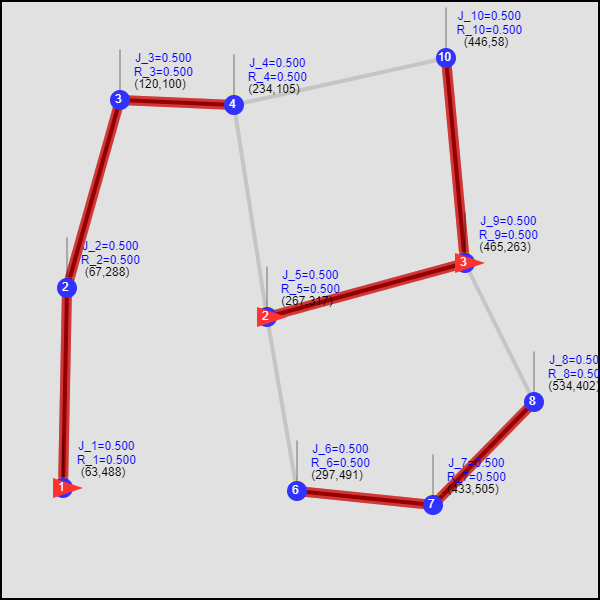}
         \caption{Exchange 2}
         
     \end{subfigure} 
    \caption{MASE1: (a) Clustering result obtained (with a properly chosen $\sigma$ value), (b) Constrained cycles generated on each sub-graph using Algorithm \ref{Alg:GreedyTargetCycleConstruction2}, (c),(d) Two inter-cluster target exchange steps, and, (d) Best target-cluster/cycle arrangement found for the MASE1.}
    \label{Fig:TargetClusteringAndRefinementProcess1}
\end{figure}

\begin{figure*}[!h]
     \centering
     \begin{subfigure}[b]{0.3\columnwidth}
         \centering
         \includegraphics[width = \textwidth]{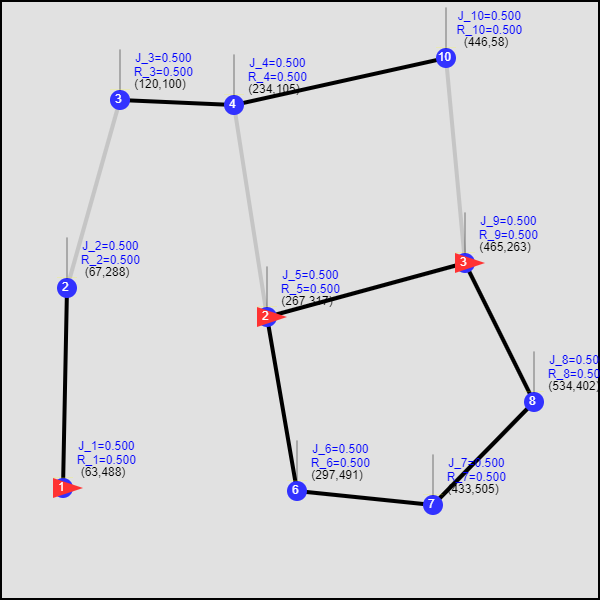}
         \caption{Clusters}
         
     \end{subfigure}
     \hfill
     \begin{subfigure}[b]{0.3\columnwidth}
         \centering
         \includegraphics[width = \textwidth]{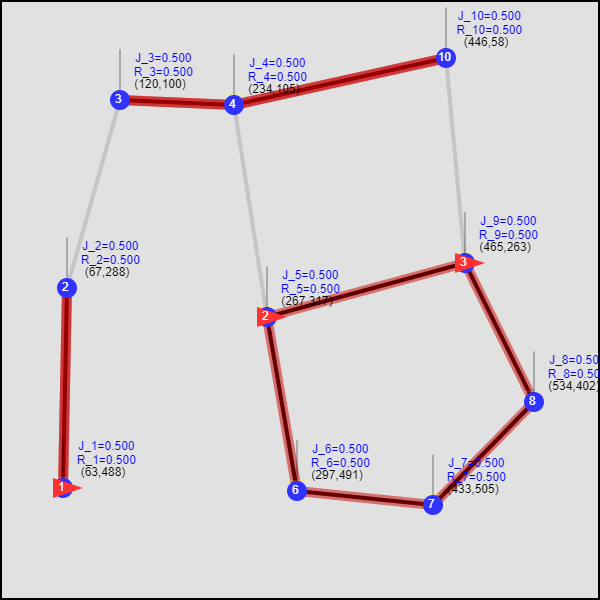}
         \caption{Cycles}
         
     \end{subfigure} 
     \hfill
     \begin{subfigure}[b]{0.3\columnwidth}
         \centering
         \includegraphics[width = \textwidth]{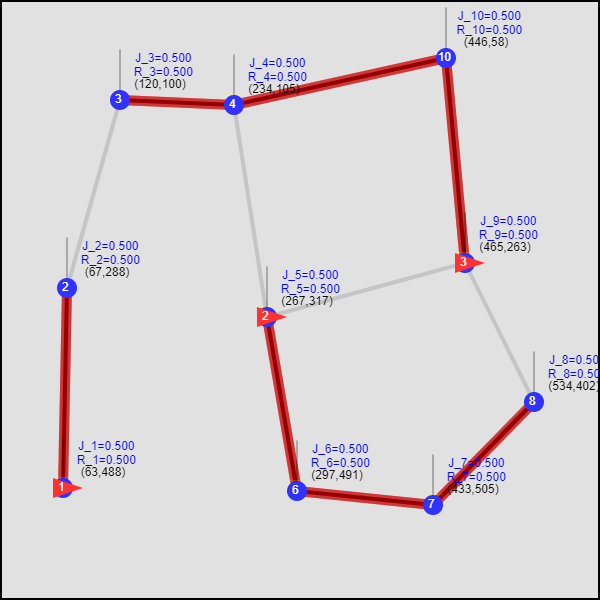}
         \caption{Exchange 1}
         
     \end{subfigure} 
     \hfill
     \begin{subfigure}[b]{0.3\columnwidth}
         \centering
         \includegraphics[width = \textwidth]{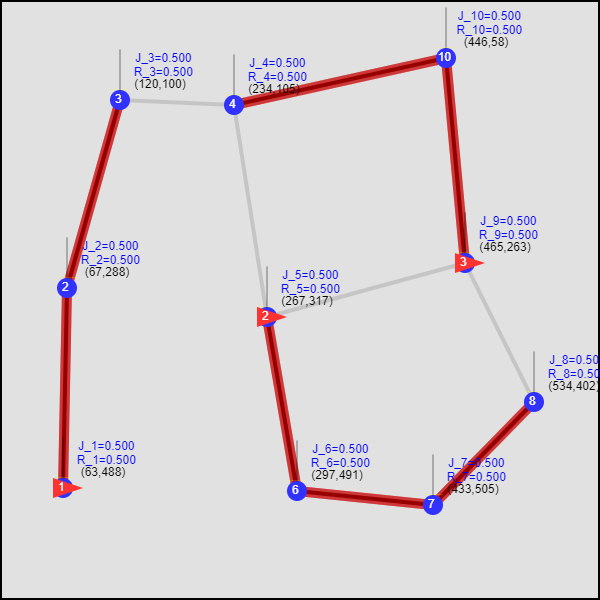}
         \caption{Exchange 2}
         
     \end{subfigure}
     \hfill
     \begin{subfigure}[b]{0.3\columnwidth}
         \centering
         \includegraphics[width = \textwidth]{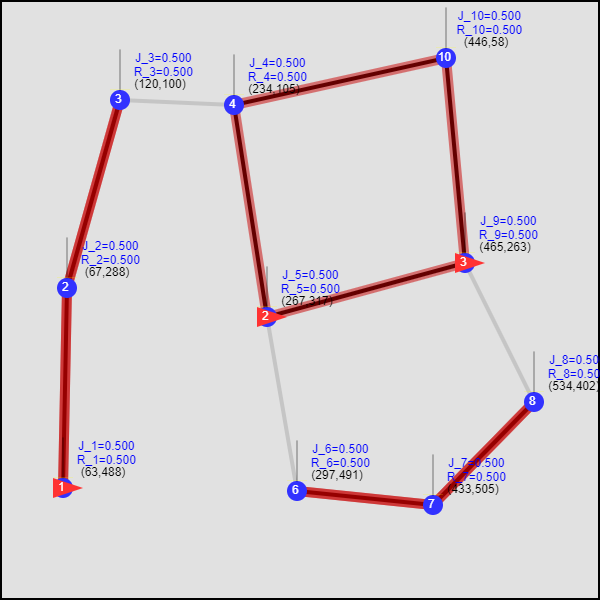}
         \caption{Exchange 3}
         
     \end{subfigure} 
     \hfill
     \begin{subfigure}[b]{0.3\columnwidth}
         \centering
         \includegraphics[width = \textwidth]{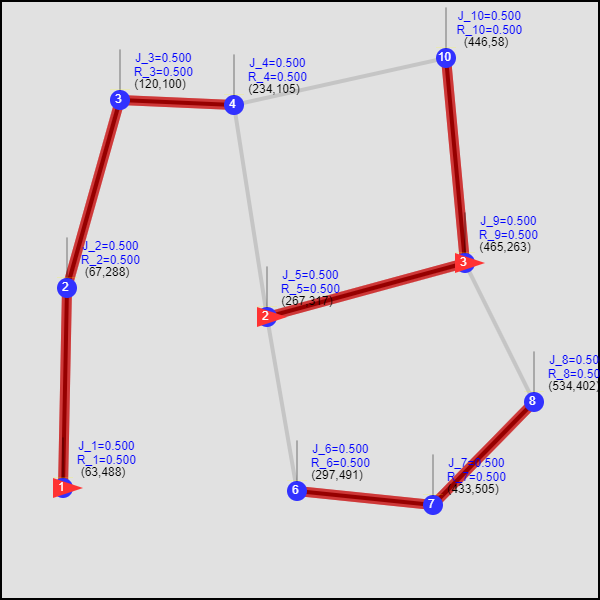}
         \caption{Exchange 4}
         
     \end{subfigure} 
    \caption{MASE1: (a) Clustering result obtained (with a poorly chosen $\sigma$ value), (b) Constrained cycles generated on each sub-graph using Algorithm \ref{Alg:GreedyTargetCycleConstruction2}, (c),(d),(e),(f) Four inter-cluster target exchange steps, and, (f) Best target-cluster/cycle arrangement found for the MASE1. (same as in Fig. \ref{Fig:TargetClusteringAndRefinementProcess1}(d)).}
    \label{Fig:TargetClusteringAndRefinementProcess2}
\end{figure*}

\paragraph*{\textbf{Effect of the Proposing Initialization Scheme (Algorithm \ref{Alg:OverallSolution})}} 
We investigate the effect of the complete proposing initialization scheme (given in Algorithm \ref{Alg:OverallSolution}) using the four multi-agent simulation examples: MASE1, MASE2, MASE3 and MASE4. Sub-figure (a) of Fig. \ref{Fig:MASE1ConstrainedCycleInitialization}, \ref{Fig:MASE2ConstrainedCycleInitialization}, 
\ref{Fig:MASE3ConstrainedCycleInitialization}, and, 
\ref{Fig:MASE4ConstrainedCycleInitialization} shows the determined graph partitions (and the derived target-cycles) of each considered problem configuration. Sub-figure (b) of each of those figures shows that the initial TCP obtained using the derived target-cycles have $J_T$ values $90.9,\ 35.1,\ 59.5$ and $59.8$ respectively. Same figures show that these $J_T$ values cannot be further improved using the gradient steps \eqref{Eq:GradientDescent}. To ensure these initial TCPs are locally optimal, in each case, at $l=100$, the derived initial TCP (i.e., $\Theta^{(0)}$) was randomly perturbed. The subsequent observations shows that $\Theta^{(l)}$ converges back to the same initial TCP found in each case. 

It is important to note that these solutions are much better than the best TCPs obtained with  random initialization of $\Theta^{(0)}$ (shown in Fig. 
\ref{Fig:MASE1RandomInitialization}, \ref{Fig:MASE2RandomInitialization}, \ref{Fig:MASE3RandomInitialization}, and \ref{Fig:MASE4RandomInitialization}). The improvement margins of the four considered examples are by: $+179.3\ (66.3\%),\ +56.6\ (61.7\%),\ +214.5\ (78.2\%),$ and $+141.5\ (70.3\%)$. All the discussed simulation results have been summarized in the Table \ref{Tab:SummaryOfTheResults}.

Also, note that, when compared to the performance of the target-cycles shown in Fig.  \ref{Fig:TargetClusteringResultsWithoughtBalancing2} (where initial graph partitions were not further refined), the improvement margins of the four considered examples (of the respective complete solutions) are by: $+22.0\ (19.4\%),\ +10.1\ (22.3\%),\ +3\ (0.48\%),$ and $+3.9\ (6.12\%)$. These achieved improvement levels emphasize the importance of the proposed graph partition refinement step. 

\begin{table}[]
\centering
\caption{Summary of the results obtained for the considered simulation examples}
\label{Tab:SummaryOfTheResults}
\resizebox{\columnwidth}{!}{%
\begin{tabular}{|l|l|l|l|l|l|l|l|}
\hline
\multicolumn{1}{|c|}{\multirow{2}{*}{
\begin{tabular}[c]{@{}l@{}} Cost of the optimal TCP $\Theta^*$ \\ (found using \eqref{Eq:GradientDescent}): $J_T(\Theta^{*})$ \end{tabular}}} & \multicolumn{3}{c|}{\begin{tabular}[c]{@{}c@{}}Single Agent \\ Simulation Examples\end{tabular}} & \multicolumn{4}{c|}{\begin{tabular}[c]{@{}c@{}}Multi-Agent \\ Simulation Examples\end{tabular}} \\ \cline{2-8} 
\multicolumn{1}{|c|}{} & \multicolumn{1}{c|}{1} & \multicolumn{1}{c|}{2} & \multicolumn{1}{c|}{3} & \multicolumn{1}{c|}{1} & \multicolumn{1}{c|}{2} & \multicolumn{1}{c|}{3} & \multicolumn{1}{c|}{4} \\ \hline
\begin{tabular}[c]{@{}l@{}}With randomly generated  \\ initial TCP $\Theta^{(0)}$\end{tabular} & 129.2 & 651.3 & 497.9 & 270.2  & 91.7 & 274.0 & 201.3 \\ \hline
\begin{tabular}[c]{@{}l@{}}With initial TCP $\Theta^{(0)}$ given \\ by the proposing Algorithm \ref{Alg:OverallSolution}\end{tabular} & 114.6 & 567.0 & 449.5 & 90.9 & 35.1 & 59.5 & 59.8 \\ \hline
\begin{tabular}[c]{@{}l@{}}Percentage improvement (\%) \end{tabular} & 10.7 & 12.9 & 9.7 & 66.3 & 61.7 & 78.2 & 70.3 \\ \hline
\end{tabular}%
}
\end{table}

\begin{figure}[!h]
     \centering
     \begin{subfigure}[b]{0.35\columnwidth}
         \centering
         \includegraphics[width = \textwidth]{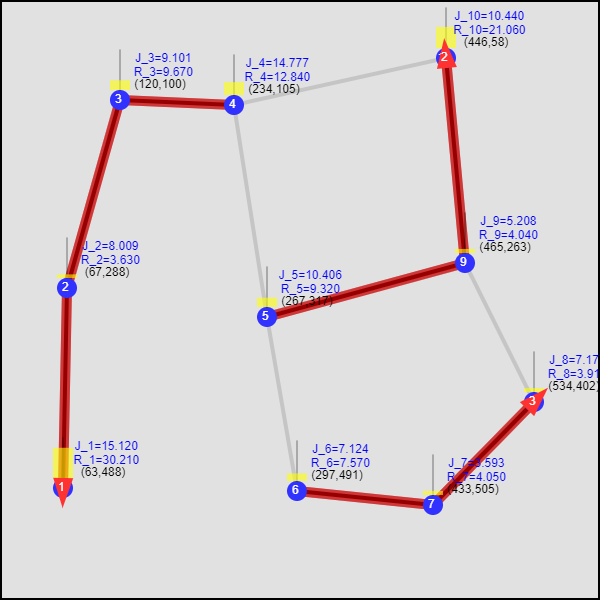}
         \caption{Config. at $t=T$.}
         
     \end{subfigure}
     \hfill
     \begin{subfigure}[b]{0.6\columnwidth}
         \centering
         \includegraphics[width = \textwidth]{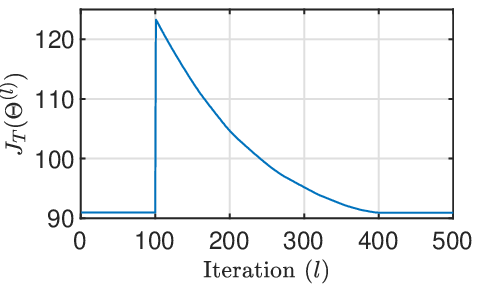}
         \caption{Cost vs iterations plot.}
         
     \end{subfigure}
    \caption{MASE1: The TCP $\Theta^{(0)}$ given by the identified cycles (the red traces in (a)) shows local optimality. At $l=100, \Theta^{(l)}$ is randomly perturbed. Yet, converges back to the initial TCP. Cost $J_T=90.9$ (Improvement $=+179.3$ compared to Fig. \ref{Fig:MASE1RandomInitialization}).}
    \label{Fig:MASE1ConstrainedCycleInitialization}
\end{figure}

\begin{figure}[!h]
     \centering
     \begin{subfigure}[b]{0.35\columnwidth}
         \centering
         \includegraphics[width = \textwidth]{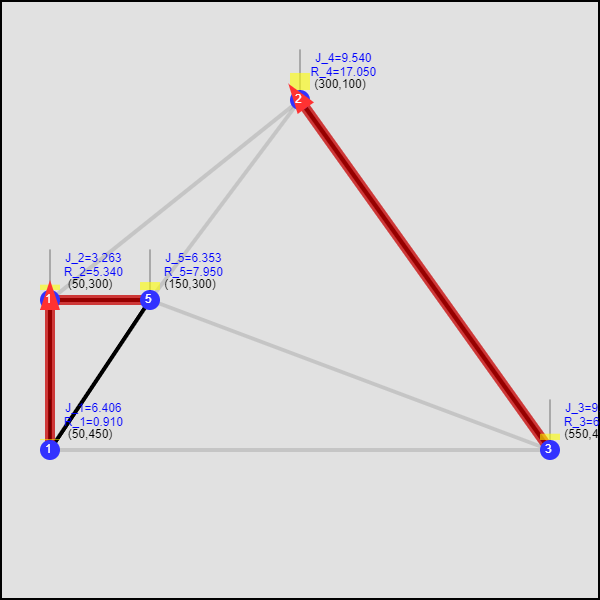}
         \caption{Config. at $t=T$.}
         
     \end{subfigure}
     \hfill
     \begin{subfigure}[b]{0.6\columnwidth}
         \centering
         \includegraphics[width = \textwidth]{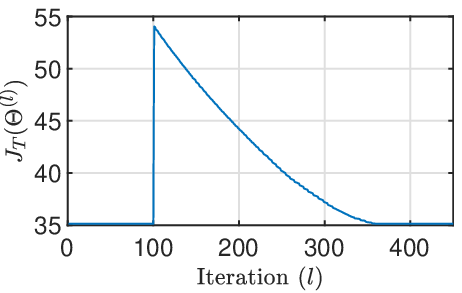}
         \caption{Cost vs iterations plot.}
         
     \end{subfigure}
    \caption{MASE2: The TCP $\Theta^{(0)}$ given by the identified cycles (the red traces in (a)) shows local optimality. At $l=100, \Theta^{(l)}$ is randomly perturbed. Yet, converges back to the initial TCP. Cost $J_T=35.1$ (Improvement $=+56.6$ compared to Fig. \ref{Fig:MASE2RandomInitialization}).}
    \label{Fig:MASE2ConstrainedCycleInitialization}
\end{figure}

\begin{figure}[!h]
     \centering
     \begin{subfigure}[b]{0.35\columnwidth}
         \centering
         \includegraphics[width = \textwidth]{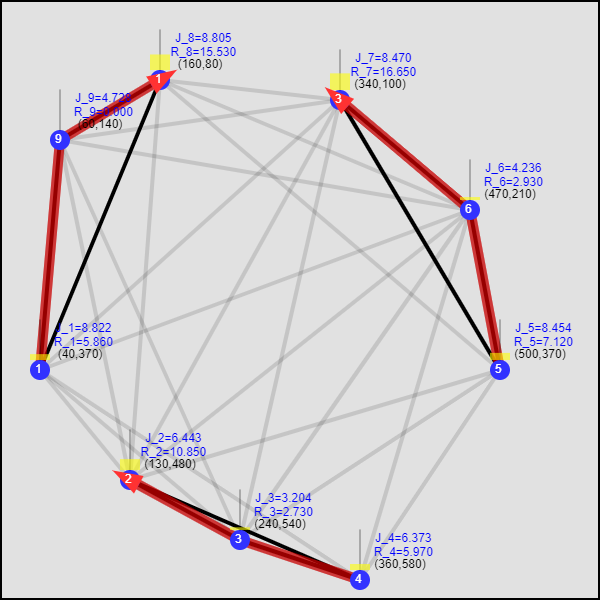}
         \caption{Config. at $t=T$.}
         
     \end{subfigure}
     \hfill
     \begin{subfigure}[b]{0.6\columnwidth}
         \centering
         \includegraphics[width = \textwidth]{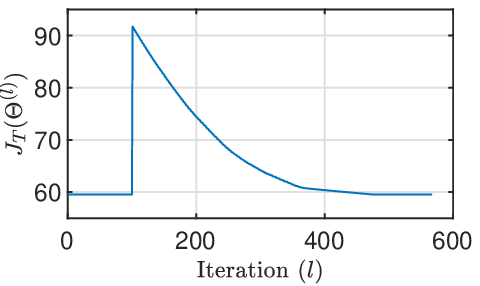}
         \caption{Cost vs iterations plot.}
         
     \end{subfigure}
    \caption{MASE3: The TCP $\Theta^{(0)}$ given by the identified cycles (the red traces in (a)) shows local optimality. At $l=100, \Theta^{(l)}$ is randomly perturbed. Yet, converges back to the initial TCP. Cost $J_T=59.5$ (Improvement $=+214.5$ compared to Fig. \ref{Fig:MASE3RandomInitialization}).}
    \label{Fig:MASE3ConstrainedCycleInitialization}
\end{figure}


\begin{figure}[t]
     \centering
     \begin{subfigure}[b]{0.35\columnwidth}
         \centering
         \includegraphics[width = \textwidth]{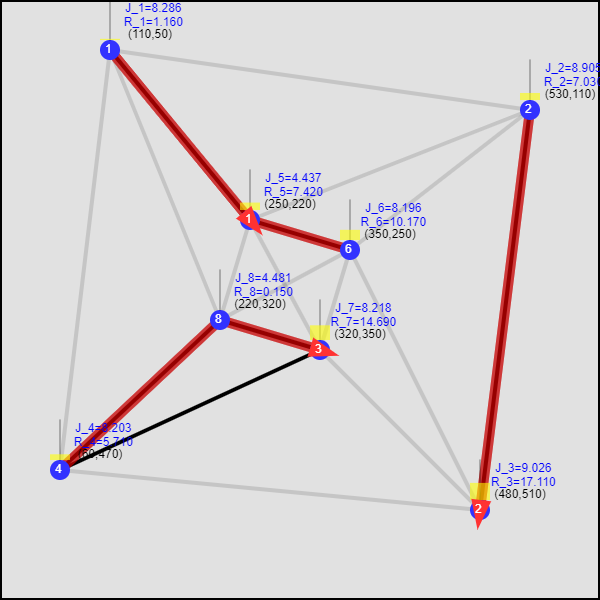}
         \caption{Config. at $t=T$.}
         
     \end{subfigure}
     \hfill
     \begin{subfigure}[b]{0.6\columnwidth}
         \centering
         \includegraphics[width = \textwidth]{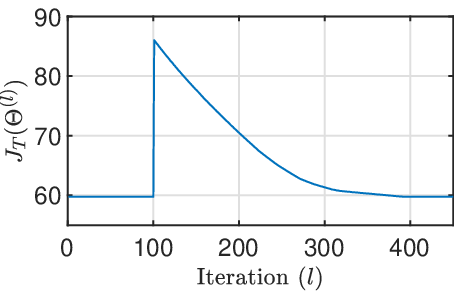}
         \caption{Cost vs iterations plot.}
         
     \end{subfigure}
    \caption{MASE4:The TCP $\Theta^{(0)}$ given by the identified cycles (the red traces in (a)) shows local optimality. At $l=100, \Theta^{(l)}$ is randomly perturbed. Yet, converges back to the initial TCP. Cost $J_T=59.8$ (Improvement $=+141.5$ compared to Fig. \ref{Fig:MASE4RandomInitialization}).}
    \label{Fig:MASE4ConstrainedCycleInitialization}
\end{figure}

As the last step, we consider a new set of multi-agent simulation examples labeled MASE5 - MASE12, as shown in Fig. \ref{Fig:MASE5_12RandomInitialization}. In each case $N=3$ and $M=15$ have been used. Further, each underlying target topology has been generated randomly - as a random geometric graph \cite{Dall2002} with communication range parameter $r=200$. When the proposing greedy initialization technique is deployed, the obtained new terminal solutions are shown in Fig. \ref{Fig:MASE5_12GreedyInitialization}. Across all these eight cases, the average improvement achieved is $+315.7\ (69.1\%)$.

In simulation examples MASE5 - MASE12, on an Intel\textregistered\ Core\texttrademark\ i7-7800 CPU \@ 3.20 GHz Processor with a 32 GB RAM, the average execution time observed for the proposing greedy initialization technique (i.e., to Algorithm \ref{Alg:OverallSolution} to generate $\Theta^{(0)}$) was $13.7 s$. Also, all such generated initial TCPs were found to be locally optimal (Similar to what we saw in MASE1-MASE4). However, when initialized with a randomly chosen TCP $\Theta^{(0)}$, the average execution time observed for the subsequent convergence of the gradient steps in \eqref{Eq:GradientDescent} was $245.8 s$. Therefore, the execution time taken for the proposing offline greedy initialization process is much smaller, and, at the same time, it is highly effective.

Finally, we recall that all the simulation examples discussed were evaluated on the developed JavaScript based simulator, which is made available at \href{http://www.bu.edu/codes/simulations/shiran27/PersistentMonitoring/}{\seqsplit{http://www.bu.edu/codes/simulations/shiran27/ PersistentMonitoring/}}. Readers are invited to reproduce the reported results and also to try new problem configurations with customized problem parameters, using the developed interactive simulator. 

\begin{figure}[!h]
     \centering
     \begin{subfigure}[b]{0.23\columnwidth}
         \centering
         \includegraphics[width = \textwidth]{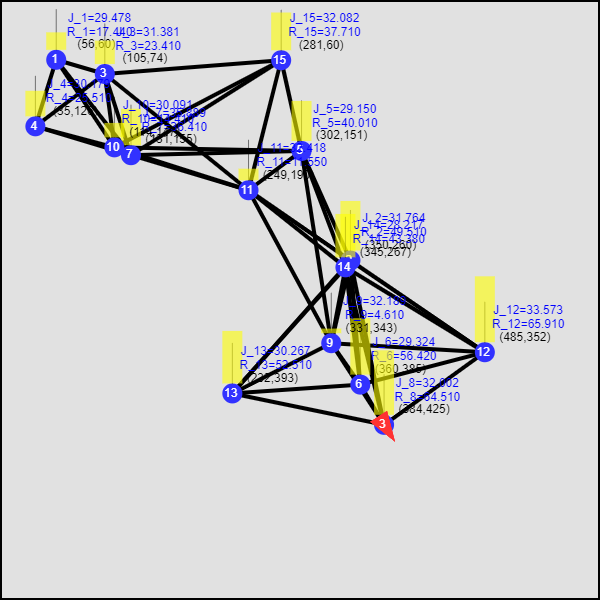}
         \caption{$J_T = 468.0$}
         
     \end{subfigure}
     \hfill
     \begin{subfigure}[b]{0.23\columnwidth}
         \centering
         \includegraphics[width = \textwidth]{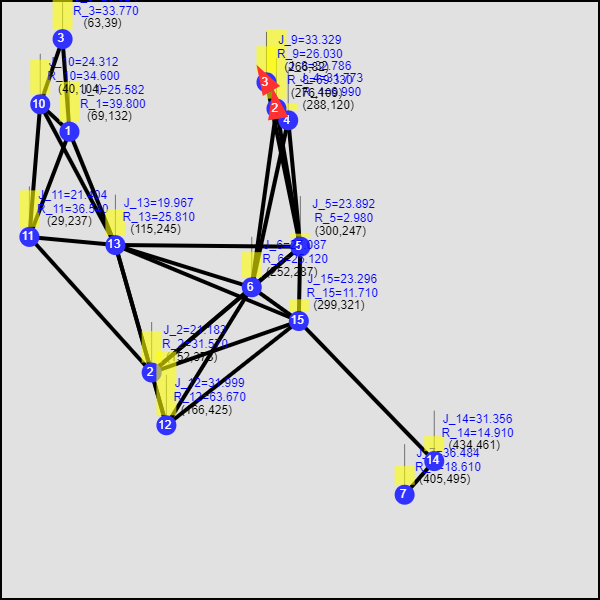}
         \caption{$J_T = 411.0$}
         
     \end{subfigure} 
     \hfill
     \begin{subfigure}[b]{0.23\columnwidth}
         \centering
         \includegraphics[width = \textwidth]{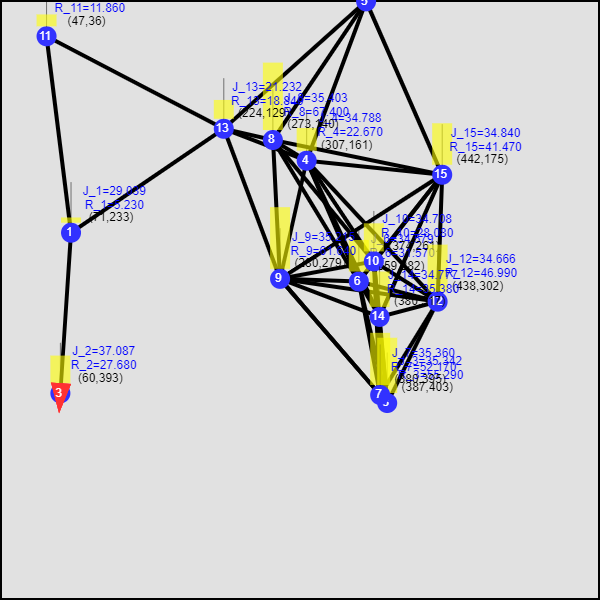}
         \caption{$J_T = 509.6$}
         
     \end{subfigure} 
     \hfill
     \begin{subfigure}[b]{0.23\columnwidth}
         \centering
         \includegraphics[width = \textwidth]{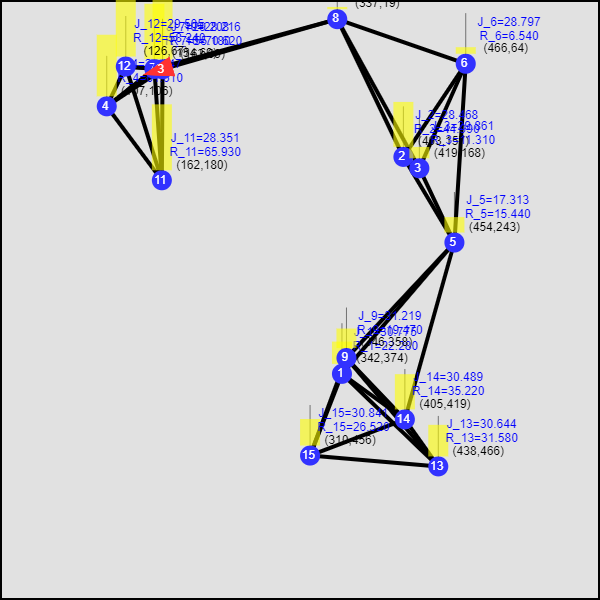}
         \caption{$J_T = 419.3$}
         
     \end{subfigure}
     \hfill
     \begin{subfigure}[b]{0.23\columnwidth}
         \centering
         \includegraphics[width = \textwidth]{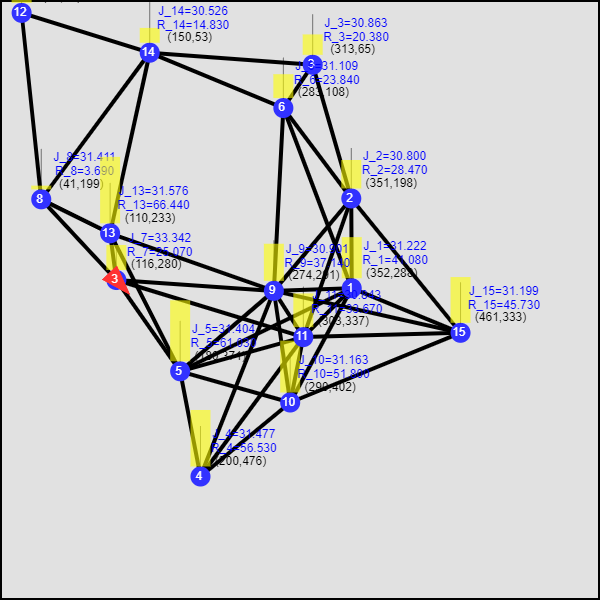}
         \caption{$J_T = 469.1$}
         
     \end{subfigure} 
     \hfill
     \begin{subfigure}[b]{0.23\columnwidth}
         \centering
         \includegraphics[width = \textwidth]{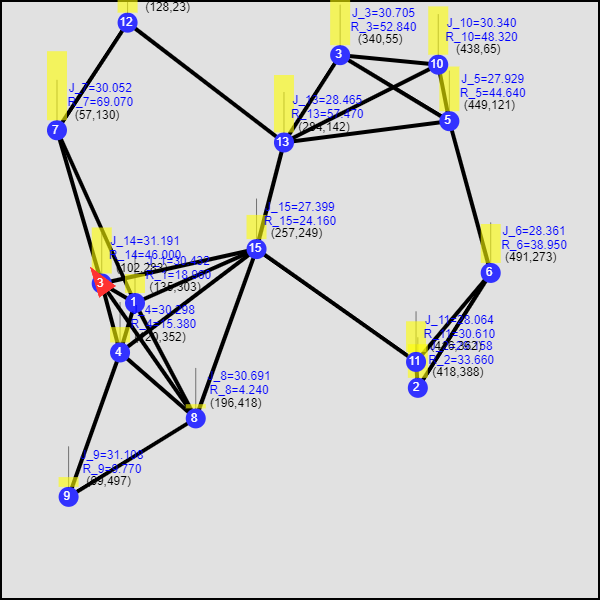}
         \caption{$J_T = 444.1$}
         
     \end{subfigure} 
     \hfill
     \begin{subfigure}[b]{0.23\columnwidth}
         \centering
         \includegraphics[width = \textwidth]{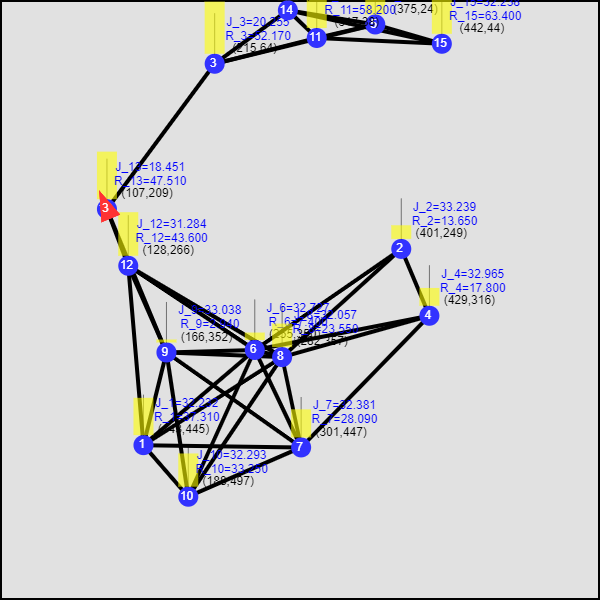}
         \caption{$J_T = 458.7$}
         
     \end{subfigure} 
     \hfill
     \begin{subfigure}[b]{0.23\columnwidth}
         \centering
         \includegraphics[width = \textwidth]{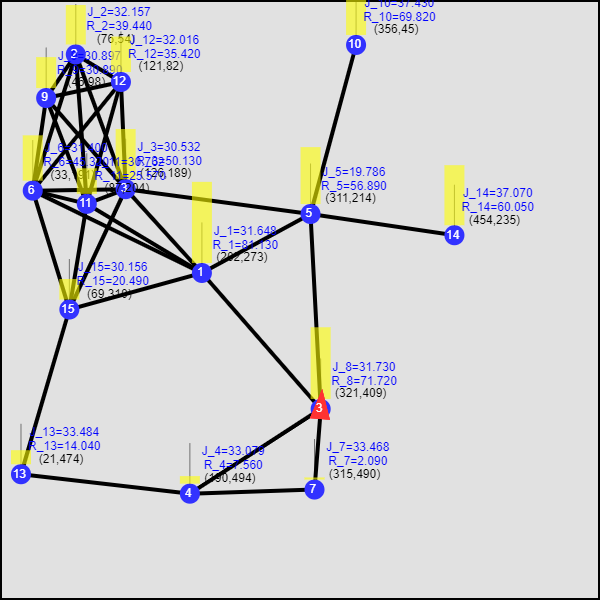}
         \caption{$J_T = 475.6$}
         
     \end{subfigure} 
    \caption{Problem configurations MASE5 - MASE12 and their respective cost values under the optimal TCP $\Theta^*$ found using \eqref{Eq:GradientDescent} when started with a randomly chosen $\Theta^{(0)}$.}
    \label{Fig:MASE5_12RandomInitialization}
\end{figure}

\begin{figure}[!h]
     \centering
     \begin{subfigure}[b]{0.23\columnwidth}
         \centering
         \includegraphics[width = \textwidth]{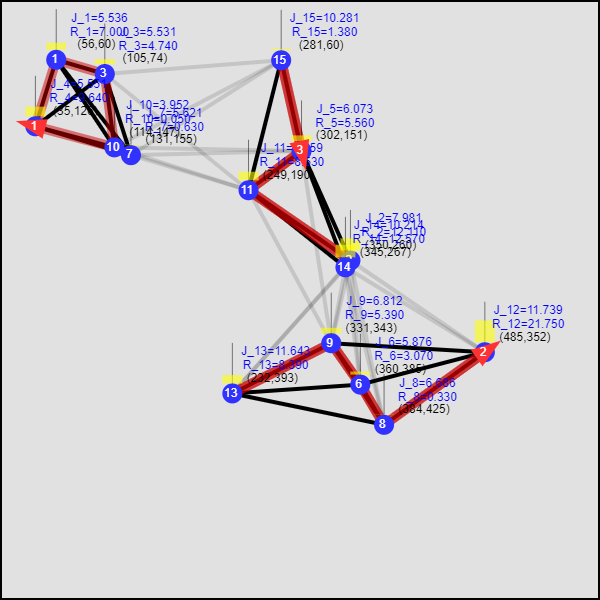}
         \caption{$J_T = 108.7$}
         
     \end{subfigure}
     \hfill
     \begin{subfigure}[b]{0.23\columnwidth}
         \centering
         \includegraphics[width = \textwidth]{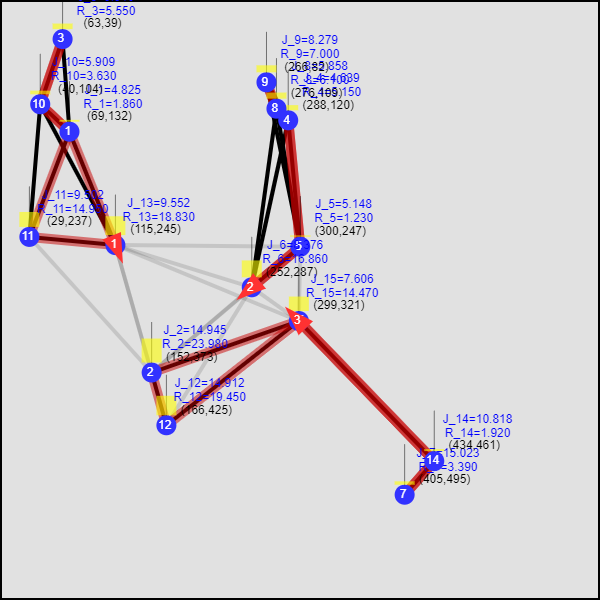}
         \caption{$J_T = 134.9$}
         
     \end{subfigure} 
     \hfill
     \begin{subfigure}[b]{0.23\columnwidth}
         \centering
         \includegraphics[width = \textwidth]{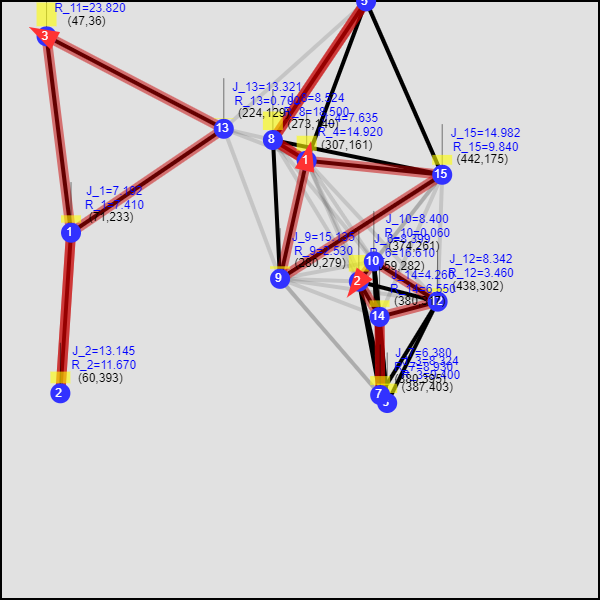}
         \caption{$J_T = 152.6$}
         
     \end{subfigure} 
     \hfill
     \begin{subfigure}[b]{0.23\columnwidth}
         \centering
         \includegraphics[width = \textwidth]{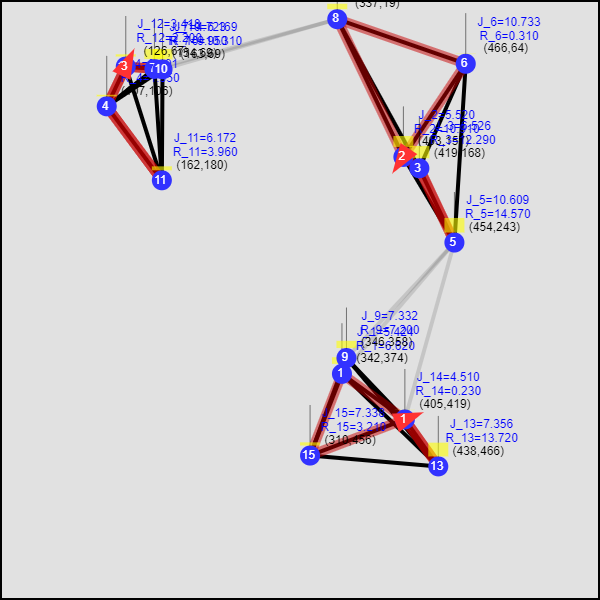}
         \caption{$J_T = 99.6$}
         
     \end{subfigure}
     \hfill
     \begin{subfigure}[b]{0.23\columnwidth}
         \centering
         \includegraphics[width = \textwidth]{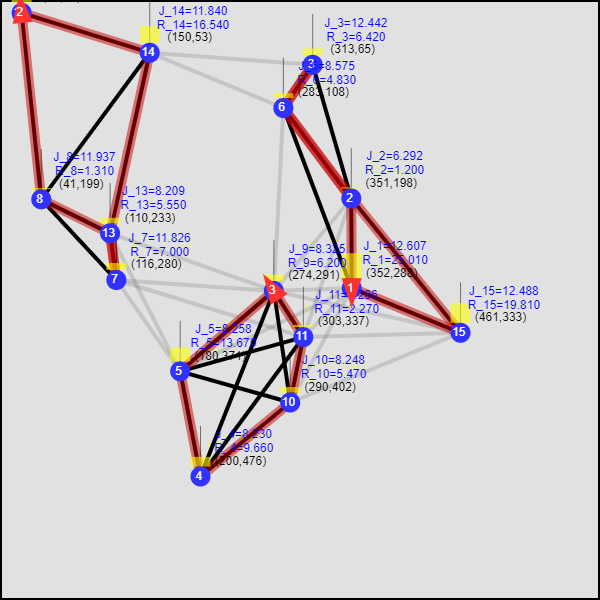}
         \caption{$J_T = 149.5$}
         
     \end{subfigure} 
     \hfill
     \begin{subfigure}[b]{0.23\columnwidth}
         \centering
         \includegraphics[width = \textwidth]{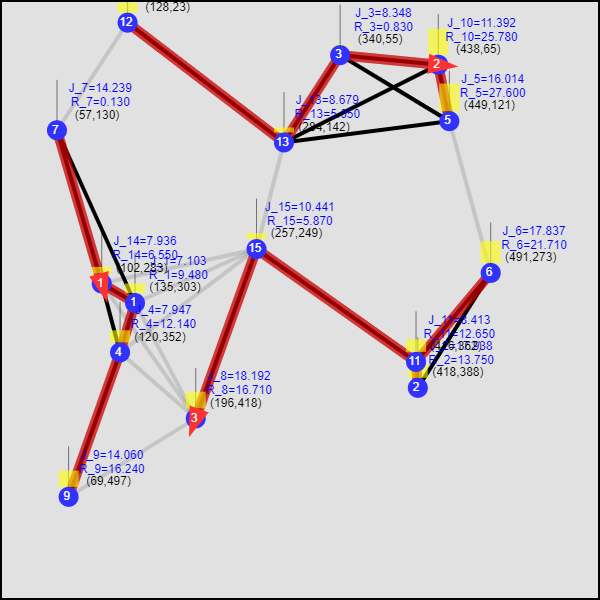}
         \caption{$J_T = 184.4$}
         
     \end{subfigure} 
     \hfill
     \begin{subfigure}[b]{0.23\columnwidth}
         \centering
         \includegraphics[width = \textwidth]{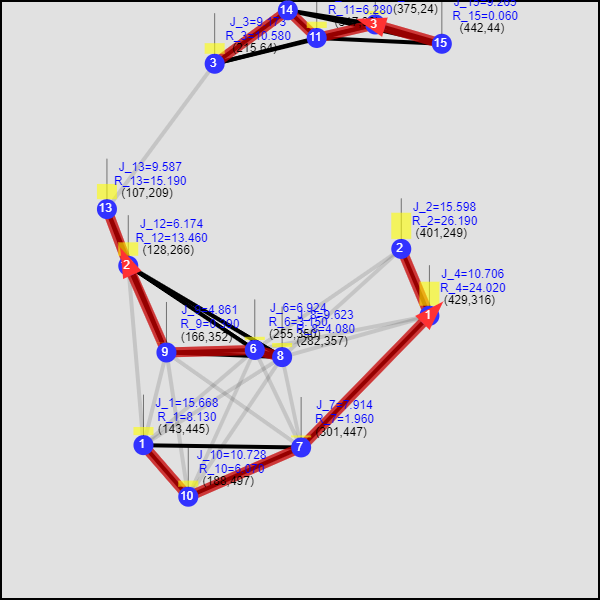}
         \caption{$J_T = 131.9$}
         
     \end{subfigure} 
     \hfill
     \begin{subfigure}[b]{0.23\columnwidth}
         \centering
         \includegraphics[width = \textwidth]{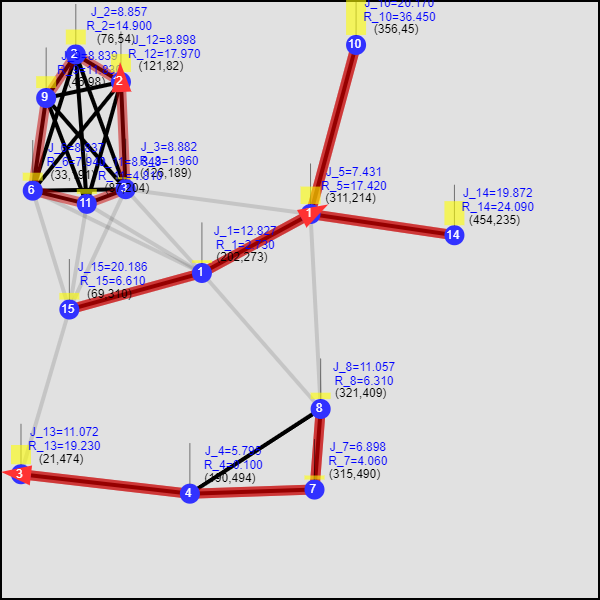}
         \caption{$J_T = 168.5$}
         
     \end{subfigure} 
    \caption{Problem configurations MASE5 - MASE12 and their respective cost values under the optimal TCP $\Theta^*$ found using \eqref{Eq:GradientDescent} when started with $\Theta^{(0)}$ given by proposing Algorithm \ref{Alg:OverallSolution}.}
    \label{Fig:MASE5_12GreedyInitialization}
\end{figure}

\section{Conclusion}\label{Sec:Conclusion}

We have considered the optimal multi-agent persistent monitoring problem on a set of targets interconnected according to a fixed graph topology. We have adopted a class of distributed threshold-based parametric controllers where IPA can be used to determine optimal threshold parameters in an on-line manner using gradient descent. Due to the non-convex nature of the problem, optimal thresholds given by gradient descent highly depend on the used initial thresholds. To address this issue, the asymptotic behavior of the persistent monitoring system was studied, which leads to a computationally efficient and effective threshold initialization scheme. Future work is directed at extending the proposed solution to PMN problems with variable travel times given by higher-order agent dynamic models.

\ifCLASSOPTIONcaptionsoff
  \newpage
\fi



\bibliographystyle{IEEEtran}
\bibliography{References}
\end{document}